\documentclass{article}
\usepackage{graphicx,amsmath,amssymb,geometry,bbm,amsthm,xcolor,bm,comment,url, mathtools}
\newtheorem{theorem}{Theorem}[section]
\newtheorem{maintheorem}{Theorem}
\newtheorem{corollary}[theorem]{Corollary}
\newtheorem{lemma}[theorem]{Lemma}
\newtheorem{definition}{Definition}

\newtheorem{proposition}[theorem]{Proposition}
\newcommand{\lra}{\leftrightarrow}
\newcommand{\wto}[1]{\widetilde{#1}}
\newcommand{\dwt}[1]{\widetilde{\widetilde{#1}}}
\newcommand{\sda}[1]{\ensuremath{\,{\buildrel #1 \over \Longleftrightarrow}\,}}
\newcommand{\sa}[1]{\ensuremath{\,{\buildrel #1 \over \longleftrightarrow}\,}}

\newcommand{\Z}{\mathbb{Z}}
\newcommand{\prob}{\mathbb{P}}
\newcommand{\edges}{\mathcal{E}^d}
\newcommand{\upiv}{\underline{\mathcal{P}}}
\newcommand{\opiv}{\overline{\mathcal{P}}}
\newcommand{\piv}{\mathcal{P}}
\newcommand{\clust}{\mathfrak{C}}

\newcommand{\cD}{\mathcal{D}}

\newcommand{\dist}{\mathrm{dist}}
\newcommand{\W}{\mathcal{W}}

\title{Convergence of $k$-point functions in\\ high dimensional percolation}

\begin{document}
   \author{Shirshendu Chatterjee \thanks{Email: shirshendu@ccny.cuny.edu.} \\ \small{The City College of New York and CUNY Graduate Center}  \and Pranav Chinmay \thanks{Email: pchinmay6@gatech.edu.} \\ \small{Georgia Institute of Technology}  \and Jack Hanson \thanks{Email: jack.hanson@uni-hamburg.de.}\\ \small{Universit\"at Hamburg, The City College of New York, and CUNY Graduate Center} \and Philippe Sosoe \thanks{Email: ps934@cornell.edu.} \\ \small{Cornell University} }

\maketitle
\vspace{1em}
\begin{center}
\textit{Dedicated to the memory of Paul Koosis (1929-2025)}
\end{center}
\vspace{1em}
\begin{abstract}We establish first-order asymptotics for the $k$-point connectivity function in critical Bernoulli percolation on $\mathbb{Z}^d$ assuming $d>6$ and Hara's two-point asymptotic. For every $k\ge 3$ and distinct macroscopic locations $\{\lfloor n y_i\rfloor\}_{i=0}^{k-1}$, the probability that the corresponding lattice points  lie in the same open cluster, rescaled by $n^{(4-d)(k-1)-2}$, converges as $n \to \infty$ to an explicit sum of integrals indexed by trees. This confirms a conjecture of Aizenman and Newman.

We also establish convergence of chemical-distance-decorated $k$-point functions. Conditionally on a set of $k$ points lying in the same cluster, this identifies the joint limiting distribution of all pairwise chemical distances after quadratic rescaling. It is described by binary trees with explicit laws along the edges. The decorated expansion also gives $k$ point asymptotics in the near-critical window. 

Our proof uses our previous results on convergence to the incipient infinite cluster to decouple macroscopic connections, as well as distributional convergence of the two-point chemical distance.
\end{abstract}

\section{Introduction}
Percolation is a classical model in probability theory and mathematical statistical mechanics. In Bernoulli bond percolation, we declare edges open independently with probability $p$ between vertices of the integer lattice $\Z^d$, and study the resulting random geometry. The model exhibits a phase transition in all dimensions $d \geq 2$; that is, there exists a {\it critical probability} $p_c = p_c(d)$ above which an infinite cluster appears on the lattice.

At criticality, the probabilities of connection events, as well as various other quantities, are expected to behave like power laws with universal exponents. In this paper, we study variants of the $k$-point connectivity functions
\[\tau_{k}(x_0,\ldots,x_{k-1})=\mathbb{P}(x_0\lra x_1,\ldots,x_0\lra x_{k-1}),\]
under the conditions $d > 6$ and the existence of the asymptotic for the two-point function \eqref{eqn: alphadef} known to hold in high dimensions. Here we write $a\lra b$ for the event that there is an open path connecting the vertices $a,b\in \mathbb{Z}^d$.

In this high-dimensional regime, the structure of the cluster of the origin in critical percolation is expected to be well-approximated by {\it critical branching random walk} (BRW). See \cite[Chapter 2]{HvH-book} for more on this analogy. In particular, on large scales, critical percolation clusters are tree-like and their scaling limit should be described by embedded critical random trees. This expected tree structure led Aizenman and Newman \cite{AN} to conjecture that the $k$-point function $\tau_k$ could be asymptotically described as a sum over trees spanning the $k$ vertices $x_i$. We establish this conjecture in Theorem \ref{thm:kpt}. For consequences of this conjecture, see the discussion at the beginning of Section \ref{sec: metric}.

The BRW heuristic is difficult to translate into a useful explicit coupling with Bernoulli percolation. Indeed, an important component of this work is identifying the relevant tree structure underlying the percolation process and controlling the dependence among connections at branch points.

Our main result concerns the limit of the rescaled $\tau_k$ in a regime where all $k$ vertices are at comparable large distances from each other. The set $\mathfrak{T}_k$ is the set of trees with $k$ leaves such that each interior (non-leaf) vertex has degree three. See the more detailed discussion in Section \ref{sec:conntree}.
\begin{maintheorem} \label{thm:kpt}
Suppose $d>6$ and the limit \eqref{eqn: alphadef} exists. Letting $x_i^{(n)}=\lfloor ny_i\rfloor \in \mathbb{Z}^d$, where $y_0,\ldots,y_{k-1}\in \mathbb{R}^d$ are distinct points,
\begin{equation}\label{eqn:  T-sum}
n^{-((4-d)(k-1)-2)}\tau_{k}(x_0^{(n)},\ldots, x_{k-1}^{(n)})\rightarrow\sum_{T\in \mathfrak{T}_{k}}\alpha^{2k-3}\kappa^{k-2}\mathcal{I}_{T}(y_0,\ldots, y_{k-1}).
\end{equation}
In the formula above
\begin{equation}\label{eqn: betadef}
    \beta=\frac{p_c}{1-p_c},
\end{equation}
with $p_c$ the critical probability and
\[\kappa=2d\beta^2\rho,\]
where the quantity $\rho$, appearing in \eqref{eq:rhodef}, is defined in terms of three copies of the incipient infinite cluster process.
$\mathfrak{T}_k$ is the set of binary branching trees on $k$ leaves (as defined in Section \ref{sec:conntree}), and
\begin{equation}\label{eqn:I_T-def}
\mathcal{I}_T(y_0,\ldots,y_{k-1}):=\int_{(\mathbb{R}^d)^{k-2}}\prod_{\{a,b\}\in E(T)} \frac{1}{|u_a-u_b|^{d-2}}\,\prod_{v\in \mathrm{int}(T)}\mathrm{d}u_v,
\end{equation}
the integration being over the interior vertices of the tree. The quantity $\alpha$ (defined in \eqref{eqn: alphadef}) is defined in terms of the two-point scaling. The powers of $\alpha$ and $\kappa$ correspond to the number of edges and degree 3 points of the graphs in $\mathfrak{T}_{k}$.
\end{maintheorem}
Thus, for separated vertices on a macroscopic scale, the binary tree expansion gives the exact leading asymptotics, with a constant factor attached to each branching vertex. The result in Theorem \ref{thm:kpt} has long been anticipated \cite[Sections 15.1-15.3]{HvH-book}. As previously mentioned, it was first conjectured by Aizenman and Newman \cite[Eqn. 4.9]{AN}, who proved an upper bound for the $k$-point function in terms of a sum over binary branching tree graphs.

A key aspect of our method is that we do not directly use the lace expansion, which has been the central tool for the analysis of several related lattice models (self-avoiding walk, lattice trees, oriented percolation, etc.) in this high dimensional setting. In percolation, the lace expansion was adapted by Hara and Slade \cite{HS90} to establish Aizenman and Newman's triangle condition, yielding continuity of the phase transition, and later refined in \cite{HHS}, \cite{Hara}, and \cite{FH}, to establish the two-point asymptotic for $d \geq 11$. Despite these impressive advances, several open questions regarding finer structural properties of clusters remained out of reach, including the analysis of the complicated diagrams associated to multi-point connection events, necessary in the study of the continuum limit \cite[Section 15.1]{HvH-book}.  

A major challenge in high dimensional percolation has therefore been to develop techniques beyond the lace expansion that can obtain concrete results regarding the structure of clusters. The present work is part of an ongoing series (\cite{CCHS1}, \cite{CCHS2}) that addresses this challenge by developing and applying the {\it incipient infinite cluster measure (IIC)} limit as a decoupling tool to analyze broad classes of diagrams that represent connection events in percolation. As a consequence, this series of works obtains, amongst other results, the exact scaling of the critical chemical distance, sharp asymptotics for $k$-point functions, and joint convergence for pairwise distances, filling in several missing pieces of the conjectural landscape.   

\section{Results: the metric structure of high-dimensional clusters}\label{sec: metric}
Theorem \ref{thm:kpt} is a special case of a more general result on $k$-point functions decorated by the distances along the branches of the tree $T$ appearing in the summation \eqref{eqn: T-sum}, which illuminate the metric structure of the cluster of the origin on large scales. The scaling limit of the cluster of the origin ``conditioned to be large" has long been expected to be an embedding of Aldous' Continuum Random Tree by Brownian motion, and correspondingly the total mass converges to the canonical Super-Brownian Excursion measure.
(Thanks in part to the current work, the scaling limit statement is now a complete theorem. See Section \ref{sec:related} for details and \cite{Slade,LeGall,HHP} for further information on superprocesses.) As a consequence, conditionally on the tree geometry, the distances along branches are Brownian segments in the continuum limit.

In \cite{CCHS2}, we proved that if $D_n=\mathrm{dist}(0,n\mathbf{e}_1)$ denotes the \emph{chemical distance}, the least number of open edges in any path from the origin $0=(0,\ldots,0)$ to $n\mathbf{e}_1$, $\mathbf{e}_1=(1,0,\ldots,0)$, then conditioned on $0\lra n\mathbf{e}_1$, 
\[Z_n:=a\frac{D_n}{n^2}\xrightarrow{d} T_d,\]
where $T_d$ is a random variable with distribution given by the density
\begin{equation}\label{eqn: limit-density}
    f_d(t)=\frac{1}{2^{\frac{d}{2}-1}\Gamma(\frac{d}{2}-1)}t^{-\frac{d}{2}}e^{-1/(2t)}, \quad t> 0.
\end{equation}
The normalization constant is
\[a:=\frac{z_d}{2d \alpha_d \alpha \beta}\]
and
\[z_d:= \frac{\Gamma(d/2-1)}{2\pi^{d/2}},\]
with $\alpha$ as in \eqref{eqn: alphadef} and $\beta$ defined in \eqref{eqn: betadef}. The remaining constant $\alpha_d$ represents the mean chemical distance between the two endpoints of a bubble associated with a bulk pivotal. See \cite[Equation 34]{CCHS2}. Similarly, there is a constant $a_{\mathrm{piv}}$ such that if $P_n$ denotes the number of pivotal edges between $0$ and $n\mathbf{e}_1$, then
\begin{equation}\label{eqn: aP-def}
    X_n:=a_{\mathrm{piv}}\frac{P_n}{n^2}\rightarrow T_d.
\end{equation}

For two vertices $x, y$, we write $\dist(x,y)$ for the chemical or graph distance between $x$ and $y$ in the open cluster: the length of the open path from $x$ to $y$ containing the fewest edges, with $\dist(x,y) = \infty$ when $\{x \not \lra y\}$ occurs. For any two vertices $x,y \in \mathbb{Z}^d$, we often abbreviate
\[D(x,y)=\mathrm{dist}(x,y)\]
in lengthy expressions and displays. With this notation in place, we have the following
\begin{maintheorem}
\label{thm: distances}

Consider instead the \emph{decorated} $k$-point function, defined for $\boldsymbol{\mu}=(\mu_i)_{i=0}^{k-1}$ by
\[\tau^{(\mu)}_k(x_0^{(n)},\ldots,x^{(n)}_{k-1})=\mathbb{E}[e^{-\sum_{i=0}^{k-1}a n^{-2}\mu_i D(0,x_i^{(n)})},x_0\lra x_1,\ldots,x_0\lra x_{k-1}].\]
Let $x_i^{(n)}=\lfloor ny_i\rfloor$, where
$y_0,\ldots,y_{k-1}$ are distinct.
Then we have
\[n^{-((4-d)(k-1)-2)}\tau^{(\boldsymbol{\mu})}_{k}(x_0^{(n)},\ldots, x_{k-1}^{(n)})\rightarrow\sum_{T\in \mathfrak{T}_{k}}\alpha^{2k-3}\kappa^{k-2}\mathcal{I}^{(\mu)}_{T}(y_0,\ldots, y_{k-1}),\]
where 
\begin{equation}
\mathcal{I}^{(\boldsymbol{\mu})}_T(y_0,\ldots,y_{k-1}):=\int_{(\mathbb{R}^d)^{k-2}}\prod_{\{a,b\}\in E(T)} \frac{\psi_d(\Lambda_{ab}|u_a-u_b|^2)}{|u_a-u_b|^{d-2}}\,\prod_{v\in \mathrm{int}(T)}\mathrm{d}u_v,
\end{equation}
and
\[\Lambda_{ab}=\sum_{i:\{a,b\}\in \gamma_{0i}}\mu_i,\]
where $\gamma_{0i}$ is the path in $T$ from $0$ to $i$.
The function $\psi_d$ is defined as the Laplace transform of the distribution \eqref{eqn: limit-density}, and satisfies
\[z_d|u|^{2-d}\psi_d(\lambda |u|^2)=\int_0^\infty e^{-\lambda t}p_t(u)\,\mathrm{d}t,\]
where $p_t$ is the heat kernel
\[p_t(u)= (2\pi t)^{-d/2}e^{-|u|^2/2t}.\]

\end{maintheorem}

As a corollary of the core step in the proof of the previous result, Theorem \ref{thm: decorated-conv}, we obtain joint convergence of all pairwise distances among $k$ separated points.
\begin{corollary}\label{cor: distances}
    Recall the abbreviation $D(x,y) = \dist(x,y)$ for the chemical distance between vertices $x, y \in \Z^d$. For $T\in\mathfrak{T}_k$ and $0\leq i<j\leq k-1$, write
$\operatorname{path}_T(i,j)$ for the path in $T$ joining the leaves $i$
and $j$.  Given
$\boldsymbol{\lambda}=(\lambda_{ij})_{0\leq i<j\leq k-1}\in
[0,\infty)^{\binom{k}{2}}$, set
\begin{equation}
\label{eq:pair-edge-parameter}
\Lambda_e^T(\boldsymbol{\lambda})
:=\sum_{\substack{0\leq i<j\leq k-1:\\
e\in\operatorname{path}_T(i,j)}}\lambda_{ij},
\qquad e\in E(T),
\end{equation}
and define
\begin{equation}
\label{eq:pair-decorated-integral}
\mathcal{I}^{\mathrm{pair}}_T(\boldsymbol{\lambda};
y_0,\ldots,y_{k-1})
:=
\int_{(\mathbb{R}^d)^{k-2}}
\prod_{\{a,b\}\in E(T)}
\frac{
\psi_d\!\left(
\Lambda_{\{a,b\}}^T(\boldsymbol{\lambda})|u_a-u_b|^2
\right)}
{|u_a-u_b|^{d-2}}\,
\prod_{v\in\mathrm{int}(T)}\mathrm{d}u_v,
\end{equation}
where $u_i=y_i$ at the leaves.  Thus
$\mathcal{I}^{\mathrm{pair}}_T(\mathbf{0};y_0,\ldots,y_{k-1})
=\mathcal{I}_T(y_0,\ldots,y_{k-1})$ in \eqref{eqn:I_T-def}.

\label{cor:all-pairs}
Let $x_i^{(n)}=\lfloor ny_i\rfloor$, where
$y_0,\ldots,y_{k-1}$ are distinct, and let
\[
\Gamma_n:=\{x_0^{(n)}\lra x_1^{(n)},\ldots,x_0^{(n)}\lra x_{k-1}^{(n)}\}.
\]
Under the hypotheses of Theorem~\ref{thm:kpt}, and with $a$ as in the
decorated $k$-point Theorem~\ref{thm: distances}, conditionally on $\Gamma_n$,
\begin{equation}
\Big(a n^{-2}D(x_i^{(n)},x_j^{(n)})\Big)_
{0\leq i<j\leq k-1}
\ \Longrightarrow\
\left(\mathcal{D}_{ij}\right)_{0\leq i<j\leq k-1}.
\end{equation}
The law of the limiting vector is characterized by
\begin{equation}
\label{eq:all-pairs-LT}
\mathbb{E}\!\Big[
\exp\!\big\{-\sum_{0\leq i<j\leq k-1}
\lambda_{ij}\mathcal{D}_{ij}\big\}\Big]
=
\frac{\displaystyle
\sum_{T\in\mathfrak{T}_k}
\mathcal{I}^{\mathrm{pair}}_T
(\boldsymbol{\lambda};y_0,\ldots,y_{k-1})}
{\displaystyle
\sum_{T\in\mathfrak{T}_k}
\mathcal{I}_T(y_0,\ldots,y_{k-1})}.
\end{equation}
\end{corollary}
Corollary \ref{cor: distances} completes the program announced in \cite{CCHS2} of describing the limiting joint distribution of the distances to distant points in the cluster of the origin. Taken together, the results above identify the scaling limit of cluster measure and its metric structure at the level of finite dimensional distributions. Formulating precise functional limit theorems requires some care and their proof needs non-trivial additional estimates. These have been provided by others since the initial appearance of Theorem \ref{thm:kpt} in an earlier manuscript. See the section on Related Work below.

We close with another application of the decorated $k$-point convergence. It concerns the $k$-point function in slightly subcritical percolation.
\begin{corollary} \label{cor:subk}
    Let $m>0$ and set 
    \[p_n=p_c-\frac{m}{n^2},\]
    and consider the subcritical $k$-point function
    \[\tau_{k,p_n}(x_0,\ldots,x_{k-1})=\mathbb{P}_{p_n}(x_0^{(n)}\lra x_1^{(n)}, \ldots,x_0^{(n)}\lra x_{k-1}^{(n)})\]
    Then
    \[n^{-((4-d)(k-1)-2)}\tau_{k,p_n}(x_0^{(n)},\ldots, x_{k-1}^{(n)})\rightarrow\sum_{T\in \mathfrak{T}_{k}}\alpha^{2k-3}\kappa^{k-2}\mathcal{I}_{T,m}(y_0,\ldots, y_{k-1}),\]
where
\[\mathcal{I}_{T,m}(y_0,\ldots,y_{k-1}):=\int_{(\mathbb{R}^d)^{k-2}}\prod_{\{a,b\}\in E(T)} \frac{\psi_d(\theta_m|u_a-u_b|^2)}{|u_a-u_b|^{d-2}}\,\prod_{v\in \mathrm{int}(T)}\mathrm{d}u_v,\]
where 
\[\theta_m:= \frac{m}{a_{\mathrm{piv}}p_c}, \quad a_{\mathrm{piv}}=\alpha_d a,\]
with $a_{\mathrm{piv}}$ is as in \eqref{eqn: aP-def}. 

\end{corollary}

\paragraph{Tree structure.} At a conceptual level, a central point is that the three-point function limit can be analyzed in terms of a measure absolutely continuous with respect to the triple product of the IIC measure. This is reflected in the ``vertex factor" $\rho$
whose importance was already highlighted by Aizenman-Newman, and which we identify in terms of (the triple product of) the IIC measure \eqref{eq:rhodef}.
The rest of the argument is an induction showing that the leading order behavior in the $k$-point function can be analyzed by repeatedly identifying triple points and falling back on the 3-point case.

Finally, we note that, as for our previous work in this series \cite{CCHS2}, we expect that the results of this paper extend in a straightforward manner to the case of spread-out percolation under only the dimensional condition $d>6$. Indeed, in this case the estimate \eqref{eqn: HHS} is known unconditionally. Moreover, the key convergence tool \cite{CCHS1} was proved also for this case.

\subsection{Context and Related Work} 
\label{sec:related}
The present work is the continuation of a series, started in \cite{CHS,CCHS1,CCHS2}, on the geometry of high-dimensional critical percolation. As in these previous works, our main result depends only indirectly on the lace expansion of Hara and Slade \cite{HS1,HS2}. To be precise, we treat the key estimates \eqref{eqn: HHS} and \eqref{eqn: alphadef}, whose proofs are obtained through the lace expansion, as black box results, but otherwise our derivations are independent of this tool. This is of special importance given the future prospect of an alternative derivation of the two-point function asymptotics offered by the recent emergence of new approaches to the high-dimensional regime,  see for example \cite{ASS,EGPS} and especially \cite{DCP}, as well as the recent work \cite{BR}. In the context of critical \emph{oriented} percolation in high dimensions, the analogous result goes back to van der Hofstad and Slade \cite{HS-oriented}.

An alternative proof of convergence of $k$-point functions, as well as joint convergence of pairwise distances, pivotal counts and resistances, was given independently in impressive recent work of Blanc-Renaudie and Hutchcroft \cite{BRH}. These authors also prove convergence of the cluster of the origin as an embedded metric measure space to the Brownian embedding of the continuum random tree, in the Gromov-Hausdorff-Prokhorov topology. Their proof is completely independent of  ours, and is based on a careful combinatorial argument to evaluate $k$ point functions with extra marked pivotals along the branches of the branching tree whenever $\min |x_i-x_j|\rightarrow\infty$.

The paper \cite{BRH} appeared on arXiv a few months after the original version of the current work, which proved only convergence of the $k$-point function (Theorem \ref{thm:kpt}). Theorem \ref{thm: distances} and its corollaries are thus additions in this revision. We received an early version of \cite{BRH} ahead of the completion of the current revision of this manuscript.

The recent preprint \cite{CCFK} by Cabezas, Croydon, Fribergh and Kawamoto, appeared within a day of \cite{BRH}. It uses our main result directly to give another proof of convergence of the cluster of the origin to the occupation measure of super-Brownian Motion and derives range convergence and applications.

\subsection{Overview of the proof}
Although the proof of our main theorems is largely written using the ``decorated'' $k$-point functions, it is easiest to conceptualize the argument in the case of Theorem~\ref{thm:kpt}, obtained by setting $\boldsymbol{\lambda} \equiv \mathbf{0}$ in all relevant expressions. We sketch the structure of the argument in this special case (the case of Theorem~\ref{thm:kpt}). The proof proceeds by induction on $k$, with the three-point function ($k = 3$) serving as the base case.  We outline the main steps. The special case $k = 3$ is treated in Section~\ref{sec: 3pt-scaling} using similar --- in fact, significantly simpler --- arguments to those used in the induction step. We thus focus here on the inductive step: establishing the case of Theorem~\ref{thm:kpt} for a value $k \geq 4$, assuming it has been established for all smaller $k$.
 
\paragraph{Step 1: Identifying the branching structure.}  Given an outcome $\omega$ in which $x_0, \ldots, x_{k-1}$ are connected, we construct a canonical \emph{connectivity tree} $T(\omega)$ that records the minimal branching structure of the connections.  This tree has $k$ leaves (the vertices $x_i$) and interior vertices corresponding to points where connections to different $x_i$ first diverge.  We show (Lemma~\ref{lem:binbran}) that with high probability, this tree is \emph{binary branching}: every interior vertex has in-degree exactly two.  Configurations producing non-binary trees involve additional coincidences of pivotal edges and contribute terms of strictly smaller order.  The key technical tool here is a set of diagrammatic estimates (Section~\ref{sec:diagrams}) showing that tree-shaped diagrams dominate, while diagrams containing cycles are suppressed by powers of $n$.
 
\paragraph{Step 2: Switching.}  Having identified a binary connectivity tree, we select two leaves $x_I, x_J$ sharing a common parent vertex $v$ and locate the \emph{last common pivotal edge} $g$ for the connections $\{x_0 \lra x_I\}$ and $\{x_0 \lra x_J\}$.  By a switching argument (Lemma~\ref{lem:newcut}), we express the probability of the original event (in which $g$ is open and serves as the pivotal) in terms of a modified event in which $g$ is closed.  This effectively removes the two leaves $x_I$ and $x_J$ from the tree, replacing them by the single vertex $\underline{g}$, and reduces the problem to the $(k-1)$-point function of the remaining vertices. A more intricate switching result in Lemma \ref{lem: GT-truncation} replaces decorated $k$-point functions with distances by approximations over a product space with an avoidance condition.
 
\paragraph{Step 3: Decoupling via the IIC.}  After switching, the event decomposes into two approximately independent parts: the connection from $x_0$ to the other vertices through a reduced tree (which involves the $(k-1)$-point function), and the connections from $x_I$ and $x_J$ to $\overline{g}$ (which contribute two-point function factors).  The key challenge is to make this decoupling precise.  We accomplish this by introducing the IIC measure $\nu$ as an intermediary.  After conditioning on the cluster of $\bar{g}$,  the connection from $x_0$ to $\underline{g}$ avoiding this cluster is then controlled using enhanced versions (Lemma~\ref{lem:iicnew2}) of our IIC convergence result \cite{CCHS1}, which also treats decorated quantities. This shows that the law of the cluster near $g$ converges to the IIC measure even when conditioning on connections to multiple distant vertices simultaneously.
 
\paragraph{Step 4: Emergence of the vertex factor $\rho$.} 
Further decoupling is achieved in Section \ref{sec:dec2}.
Once the decoupling is performed, the contribution of the branching event at $g$ is controlled by the quantity $\rho$ (Lemma~\ref{lem:genlem1}).  This quantity measures the probability, under three independent copies of the IIC based at nearby sites, that the three resulting infinite clusters avoid each other in a suitable sense.  The finiteness of $\rho$ follows from the two-point function bounds and is established in Lemma~\ref{lem:rhofinite}.
 
\paragraph{Step 5: Riemann sum approximation.}  After the inductive reduction, the sum over the position of $g$ becomes a Riemann sum approximation to the integral $\mathcal{I}_T$.  The convergence of this sum to the integral follows from a priori bounds on $\tau_k$ providing the necessary dominating function.

\subsection{Organization of the paper}
 
Section~\ref{sec:definitions} collects definitions, notation, and the key input estimates (two-point function bounds, BK inequality, convolution estimates) used throughout the paper.  Section~\ref{sec: 3pt-scaling} establishes the base case of the induction by proving the convergence of the three-point function (Proposition~\ref{prop:threept}).  Section~\ref{sec:generalk} introduces the connectivity tree construction, proves that it is generically binary, develops the diagrammatic estimates needed to control error terms, and carries out the inductive step (Section \ref{sec:inductive}).  Sections~\ref{sec:aux} and~\ref{sec:iic} contain most of the auxiliary lemmas (switching, truncation, and bubble lemmas) and the enhanced IIC convergence results, respectively, that are invoked in the main argument. Section~\ref{sec:dec2} shows the convergence of the ``decorating factors'' in the decorated $k$-point function to appropriate limiting expressions involving the function $\psi_d$. Section~\ref{sec:jointpw} contains the proof of Corollary~\ref{cor: distances}, and Section~\ref{sec:subcritk} contains the proof of Corollary~\ref{cor:subk}.

\section{Definitions and Inequalities}\label{sec:definitions}
We generally let $\edges = \{ (x, y): \, \|x - y\|_1 = 1\}$ denote the set of \emph{directed} edges of $\Z^d$; this makes certain expressions involving connectivity events easier to express. This of course means that, almost surely, $(x, y)$ is open if and only if $(y, x)$ is open. When $e = (x,y)$ is an edge, we write $\underline{e} = x$ and $\overline{e} =y$.
We work on the canonical probability space $\{0, 1\}^{\edges}$ with its Borel $\sigma$-algebra, writing $\omega$ for a generic outcome.

We write $B(R) = [-R, R]^d$ and, for general $x \in \Z^d$, we write $B(x; R) = [x-R, x+R]^d$. As above, we write $\{x \lra y\}$ for the event that $x$ and $y$ are connected by an open path, and we generalize this to multiple vertices in the natural way. We write $\{x \lra \infty\}$ for the event that $x$ is connected to infinitely many vertices by open paths. The next definition gives some convenient shorthand for these notions.

\begin{definition} \label{defin:gammadef}
    We introduce the notation $\Gamma(a_0, \dots, a_{k-1})$ for the event $\{a_0 \lra a_1 \lra \dots \lra a_{k-1}\}$ that $a_0, \dots, a_{k-1}$ are in the same open cluster. We use the notation $\tau_k(a_0, \dots, a_{k-1})$ for the $k$-point functions:
\[\tau_k(a_0, \dots, a_{k-1}) = \prob(\Gamma(a_0, \dots, a_{k-1}))\ . \]
We occasionally overload notation, writing in case $A = \{x_0, \dots, x_{k-1}\}$
\[\Gamma(A) = \Gamma(x_0, \dots, x_{k-1}), \]
or
\[\Gamma(z, A) = \Gamma(z, x_0, \dots, x_{k-1})\ ,\]
with similar extensions for e.g.~$\tau_k(A)$.

The open cluster $\clust(x)$ is the connected component of $x$ in the random open subgraph of $\Z^d$. In other words, $\clust(x) = \{y: \, \Gamma(x,y) \text{ occurs}\}$. We also introduce the notion of ``restricted'' connections and clusters. If $A \subseteq \Z^d$, we write $\{x \sa{A} y\}$ for the event that $x$ and $y$ are connected by an open path lying entirely in $A$. We write $\clust_A(x) = \{y: \, x \sa{A} y\}$ for the connected component of $x$ in the open subgraph of $A$.

We write $x \Longleftrightarrow y$ for the event that there are two edge-disjoint open connections from $x$ to $y$, and we write $x \stackrel{A}{\Longleftrightarrow} y$ for the event that there are two edge-disjoint connections from $x$ to $y$ lying entirely in the set $A$.
\end{definition}

\begin{definition}
\label{defin:pivs}
        Let $\mathcal{P}(u,v)$ denote the set of open pivotals for the event $\{ u \leftrightarrow v \}$: open edges through which every open connection from $u$ to $v$ must pass.  We consider these as directed edges in the natural way: if $\{x, y\}$ is an undirected edge traversed by every open path from $u$ to $v$, then either all these paths traverse $\{x,y\}$ from $x$ to $y$ (in which case $(x,y) \in \piv(u,v)$) or all these paths traverse $\{x,y\}$ from $y$ to $x$ (in which case $(y,x) \in \piv(u,v)$). Similarly, these pivotals can be naturally ordered in terms of ``distance'' to $u$: if $\gamma_1, \gamma_2$ are any (ordered) open paths from $u$ to $v$, the elements of $\mathcal{P}(u, v)$ appear in the same relative order along $\gamma_1$ as along $\gamma_2$. We use this ordering implicitly several times in what follows, saying ``closer to'' or ``further from'' $u$ to mean with respect to this order and not with respect to (e.g.) the Euclidean distance. We extend this notion for pivotals from $u$ to a set $A$ of vertices in the natural way.
        
        Given a vertex set $A$, we define $\underline{\mathcal{P}}(u, A)$ to be the common open pivotal which is closest to $u$ for all the connections of the form $u \lra x$, $x \in A$. We define $\overline{\mathcal{P}}(u, A)$ to be the common pivotal furthest from $u$. In case there are no common pivotals, we write $\underline{\mathcal{P}}(u, A) = \overline{\mathcal{P}}(u, A) = \varnothing$.
        \end{definition}

    \begin{definition}
        Let $x$ and $y$ be two vertices in $\mathbb{Z}^d$. We denote by $\mathrm{dist}(x,y)$ the \emph{chemical distance} between $x$ and $y$, to be the minimal number of open edges in any path connecting $x$ and $y$ if $x\lra y$, and set $\mathrm{dist}(x,y)=\infty$ otherwise. For notational convenience, we sometimes abbreviate
        \[\mathrm{dist}(x,y)=D(x,y).\]
        
        For $A\subset \mathbb{Z}^d$, $\mathrm{dist}_A(x,y)$ denotes the chemical distance, where the open paths are restricted to have all their vertices inside $A$.
    \end{definition}

A central role in this paper will be played by the incipient infinite cluster (IIC) measure first constructed in \cite{vdHJ}:
\begin{definition}
    Let $\mathbf{e}_1=(1,0,\ldots,0)$. The Incipient Infinite Cluster measure $\nu_x$ is defined by
    \begin{equation}\label{eqn: IIC-conv}
    \nu_x = \lim_{n \to \infty} \prob(\cdot \mid x \lra  n \mathbf{e}_1) \ ,
    \end{equation}
    where the limit is in the weak sense. We write $\nu = \nu_0$.
    In \cite{CCHS1}, it was shown that the point-to-point conditioning of \eqref{eqn: IIC-conv} may be replaced by conditioning on any long open connection from $x$. If $(V_n)$ and $(\mathcal{D}_n)$ are sequences of subsets of $\Z^d$ with $\limsup_n V_n = \limsup_n \mathcal{D}_n = \varnothing$, and if $x$ and $V_n$ are in the same connected component of $\Z^d \setminus \mathcal{D}_n$ for all $n$, we have
    \begin{equation}
        \label{eqn: IIC-conv2}
        \nu_x = \lim_{n \to \infty} \prob\left( \cdot\  \left|\, x \sa{\Z^d \setminus \mathcal{D}_n} V_n\right.\right)\ .
    \end{equation}
    
    The cluster of $x$ under $\nu_x$, which is a.s.~the unique infinite open cluster, is denoted by $W_x$, or $W$ when there is no ambiguity. The cluster $W$ almost surely has a single topological end. We often consider up to three independent copies of the IIC, which we decorate with tildes: $\nu_x$, $\wto{\nu}_y$, and $\dwt{\nu}_z$. The infinite clusters corresponding to $\wto{\nu}_y$ and $\dwt{\nu}_z$ will be denoted by $\wto{W}$ and $\dwt{W}$ (respectively).
\end{definition}

\begin{definition}
For an arbitrary deterministic set $\mathcal{D} \subseteq \Z^d$, we define
\begin{equation}\label{eqn: XiDdef}
    \Xi_{v}(\mathcal{D}) = \{v \sa{\Z^d \setminus \mathcal{D}} \infty \}
\end{equation} 
the event that the site $v$ is not in a finite component when the IIC $W_v$ is cut by the set $\mathcal{D}$.     
\end{definition}

\begin{definition}
    Our results involve the ``vertex factor", an averaged nonintersection probability under three independent copies of the IIC measure $\nu$ based at different sites. We define this quantity as
    {
    \begin{equation}
    \label{eq:rhodef}
\rho := \sum_{f \in \edges}  \mathbb{E}_{\nu}\left[ \mathbb{E}_{\dwt{\nu}_{\overline{f}}} \left[\wto{\nu}_{\mathbf{e}_1}\big(\Xi_{\mathbf{e}_1}(W_0 \cup \dwt{W}_{\overline{f}} \big) \mathbf{1}_{\Xi(W_0)} \right] \, \mathbf{1}_{ 0 \sda{} \underline{f}} \right].
\end{equation}
}
Letting $W_0$, $\wto{W}_{\mathbf{e}_1}$ and $\dwt{W}_{\overline{f}}$ denote independent IICs centered at $0$, $\mathbf{e}_1$ and $\overline{f}$, respectively, we define the events

\begin{align*}
E_1&=\left\{0 \Leftrightarrow \underline{f} \text{ in } W_0 \right\},\\
 E_2&= \big\{\mathbf{e}_1 \sa{\Z^d \setminus [W_0 \cup \dwt{W}_{\overline{f}} ]} \infty \text{ in } \wto{W}_{\mathbf{e}_1} \big\},\\
E_3&=\big\{\overline{f} \sa{\Z^d \setminus W_0 } \infty \text{ in }\dwt{W}_{\overline{f}} \big\}.
\end{align*}

Then $\rho$ in \eqref{eq:rhodef}     equals
\begin{equation}
\rho =\sum_{f \in \edges} (\nu_0 \otimes \widetilde{\nu}_{\mathbf{e}_1} \otimes \dwt{\nu}_{\bar{f}})\left[ E_1 \cap  E_2 \cap E_3  \right].
\end{equation}
In the high-dimensional setting of this paper, $\rho < \infty$; this is shown  in the conclusion of this section at Lemma~\ref{lem:rhofinite}.
\end{definition}

\subsection{Inequalities and Estimates}
Suppose $A \subseteq \{0, 1\}^{\edges}$ is an arbitrary event and $\omega = (\omega_e)$ an arbitrary outcome. For each finite $I \subseteq \edges$, we say that $A$ \emph{occurs on $I$} in $\omega$ if $\omega' \in A$ for each $\omega'$ agreeing with $\omega$ on the indices of $I$ (i.e., all edges of $I$ that are open [resp.~closed] in $\omega$ are open [resp.~closed] in $\omega'$).
Given events $A_1, \dots, A_k$, we write $A_1 \circ \dots \circ A_k$ for the set of outcomes $\omega$ in which $A_1, \dots, A_k$ \emph{occur disjointly}: there exist disjoint finite $I_1, \, \dots, \, I_k \subseteq \edges$ such that $A_1$ occurs on $I_1$ in $\omega$, $A_2$ occurs on $I_2$ in $\omega$, and so on.

When $A_1 \circ \dots \circ A_k$ is measurable (that is, when it is an event), the following inequality of van den Berg, Kesten, and Reimer (``BK inequality'') holds \cite{AGH}:
\begin{equation}
    \mathbb{P}(A_1 \circ A_2 \circ \dots \circ A_k)\le \prod_{i=1}^k \prob(A_i).\label{eqn: BK}
\end{equation}
Any time this inequality is applied, the measurability of $A_1 \circ \dots \circ A_k$ will be clear, and hence we will not generally comment on it.

We use the so-called ``Japanese bracket" notation:
\[\langle x\rangle:= (1+|x|^2)^{\frac{1}{2}}.\]
Then $|x|/\langle x\rangle\rightarrow 1$ at infinity but $\langle x\rangle\ge 1$ does not vanish at 0.

\subsection{Two-point function}
All new results of this paper assume as input, in addition to the inequality $d > 6$, the following asymptotic result on the two-point function: there exist $c, C > 0$ such that, for all $x, y \in \Z^d$,
\begin{equation}\label{eqn: HHS}
    c \langle x-y\rangle^{-d+2} \leq \tau(x,y):=\mathbb{P}(x\leftrightarrow y) \leq C \langle x-y\rangle^{-d+2}.
\end{equation}
In fact, in many places we also crucially use the existence of the limit
\begin{equation}\label{eqn: alphadef}
\alpha:=\lim_{|x|\rightarrow \infty} |x|^{d-2}\mathbb{P}(0\leftrightarrow x)\ ,
\end{equation}
known to hold in the same settings as \eqref{eqn: HHS}. The original proofs of both results are due to Hara \cite{Hara} in the case of nearest-neighbor percolation in sufficiently high dimensions and Hara, van der Hofstad, and Slade for the spread-out model when $d > 6$. See also \cite{SL1, SL2} for streamlined proofs. Both \eqref{eqn: HHS} and \eqref{eqn: alphadef} are now known to hold when $d\ge11$ \cite{FH}. 

We also repeatedly use the following one-arm probability bound of Kozma and Nachmias \cite{KN}:
\begin{equation}\label{eqn: KN}
\mathbb{P}(0\leftrightarrow \partial B(0,R))\le CR^{-2},       
\end{equation}
for $R\in \mathbb{Z}_+$. This result is known to hold for $d > 6$ whenever \eqref{eqn: HHS} also holds. See \cite{ASS,EGPS} for alternative proofs.

We repeatedly use the following simple estimate for convolutions of power-type (Riesz) kernels: there is a constant $C=C(d)>0$ such that if $\alpha, \beta>0$, $\alpha+\beta<d$:
\begin{equation}\label{eqn: std-conv}
\sum_{z\in\mathbb{Z}^d}\langle x-z\rangle^{-d+\alpha}\langle z-y\rangle^{-d+\beta}\le C(d,\alpha,\beta)\langle x-y\rangle^{-d+\alpha+\beta}.
\end{equation}
See \cite[Proposition 1.7]{HHS}. 

If $\alpha=0$, $0<\beta<d$, then
\begin{equation}\label{eqn:log-conv}
\sum_{z\in\mathbb{Z}^d}\langle x-z\rangle^{-d}\langle z-y\rangle^{-d+\beta}\le C(d,\beta)\langle x-y\rangle^{-d+\beta}\log \langle x-y\rangle.
\end{equation}
\begin{proof}[Proof of \eqref{eqn:log-conv}]
By translation, we can assume $x=0$. Let $R=|y|$. The sum over far vertices is
\begin{align*}
    \sum_{|z|>2R} \langle z\rangle^{-d}\langle z-y\rangle^{-d+
    \beta}&\le C\sum_{|z|>2R} \langle z\rangle^{-2d+\beta}\\
    &\le \langle x-y\rangle^{-d+\beta}.
\end{align*}
If $(1/2)R<|z|\le 2R$, we have
\begin{align*}
    \sum_{(1/2)R<|z|\le 2R} \langle z\rangle^{-d}\langle z-y\rangle^{-d+
    \beta}&\le CR^{-d}\sum_{(1/2)R<|z|\le 2R} \langle z\rangle^{-d+\beta}\\
    &\le C(\beta,d)R^{-d+\beta} \\
    &\le C(\beta,d)\langle x-y\rangle^{-d+\beta}.
\end{align*}
Finally, when $|z|<(1/2)R$, we have $|z-y|\ge (1/2)R$, so the contribution from this region is
\[R^{-d+\beta}\sum_{|z|\le (1/2)R} \langle z\rangle^{-d}\le C\frac{\log R}{R^{d-\beta}}=C \frac{\log \langle x-y\rangle}{\langle x-y\rangle^{d-\beta}}.\]
\end{proof}

We also use an estimate for sums over triple products: for $d>4$, we have
\begin{equation}
    \label{eqn: triple-sum}
    \sum_{z\in \mathbb{Z}^d}\langle x-z\rangle^{-d+2}\langle z-y\rangle^{-d+2}\langle z-w\rangle^{-d+2}\le C(d)\sum_{\substack{\text{cyclic over}\\ x,y,w}}\frac{\min\{\langle x-y\rangle, \langle x-w\rangle\}^2}{\langle x-y\rangle^{d-2}\langle x-w\rangle^{d-2}}.
\end{equation}
\begin{proof}[Proof of \eqref{eqn: triple-sum}]
It suffices to bound the sum over the region 
\[A=\{z: |z-x|< |z-y|, |z-x|<|z-w|\}.\] The other two regions are handled identically by cyclic permutation of the variables.

Let 
\[r:=\min \{|x-y|,|x-w|\}\]
and
\[B:=\{z: |z-x|<\frac{1}{2}r\}.\]
Note that on $A$,
\[|x-y|\le |z-x|+|z-y|\le 2|z-y|,\]
so
\[|z-y|\ge\frac{1}{2}|x-y|\]
and similarly
\begin{equation}\label{eqn: small-c}
|z-w|\ge\frac{1}{2}|x-w|,
\end{equation}
so we have
\begin{align*}
&\sum_{z\in A\cap B}\langle x-z\rangle^{-d+2}\langle z-y\rangle^{-d+2}\langle z-w\rangle^{-d+2}\\
\le&~C\langle x-y\rangle^{-d+2}\langle x-w\rangle^{-d+2}\sum_{z\in B}\frac{1}{\langle x-z\rangle^{d-2}}\\
\le&~Cr^2 \langle x-y\rangle^{-d+2}\langle x-w\rangle^{-d+2}.
\end{align*}
To bound the sum on $B^c$, assume for the sake of clarity and without loss of generality that $r=|x-y|$. Then, using $|z-w|>|z-x|$ and \eqref{eqn: small-c}, we have
\begin{align*}
&\sum_{z\in A\cap B^c}\langle x-z\rangle^{-d+2}\langle z-y\rangle^{-d+2}\langle z-w\rangle^{-d+2}\\
\le& C\langle x-w\rangle^{-d+2}\sum_{z: |z-x|>r}\frac{1}{\langle z-x\rangle^{d+(d-4)}}\\
\le& C\langle x-w\rangle^{-d+2}r^{4-d}.\\
\le& Cr^2 \langle x-w\rangle^{-d+2} \langle x-y\rangle^{-d+2}.
\end{align*}
\end{proof}

We conclude this section by showing that the quantity $\rho$ is finite under the assumption \eqref{eqn: HHS}.
\begin{lemma}
    \label{lem:rhofinite}
    Letting $\rho$ be as defined above at \eqref{eq:rhodef}, we have $\rho < \infty$.
\end{lemma}
\begin{proof}
    Since the IIC probabilities appearing in the definition of $\rho$ are of course bounded by one, the lemma will be proved once we show
    \begin{equation}
        \label{eq:finitebubb}
        \sum_{u \in \Z^d} \nu(0 \Longleftrightarrow u) < \infty\ .
    \end{equation}
    In turn, \eqref{eq:finitebubb} follows immediately from the fact that there is a  $C > 0$ such that
    \begin{equation}
        \label{eq:finitebubb2}
        \nu(0 \Longleftrightarrow u) \leq C \langle u\rangle^{6-2d}\ .
    \end{equation}
    We show \eqref{eq:finitebubb2} for completeness.

    We write
    \begin{align*}
        \nu(0 \Longleftrightarrow u) &= \lim_{K \to \infty} \nu\big(0 \stackrel{B(K)}{\Longleftrightarrow} u \big)\\
        &= \lim_{K \to \infty} \lim_{n \to \infty} \prob\big(0 \stackrel{B(K)}{\Longleftrightarrow} u \mid 0 \lra n \mathbf{e}_1 \big)\\
        &\leq \limsup_{n \to \infty} \prob\left(0 \Longleftrightarrow u \mid 0 \lra n \mathbf{e}_1 \right)\ ,
    \end{align*}
    and so \eqref{eq:finitebubb2} follows once we show
    \begin{equation}
        \label{eq:finitebubb3}
        \prob\left(0 \Longleftrightarrow u, \, 0 \lra n \mathbf{e}_1 \right) \leq C n^{2-d} \langle u\rangle^{6-2d} \quad \text{for $n \geq 2 |u|$}.
    \end{equation}

    To see \eqref{eq:finitebubb3}, consider an outcome in the event appearing there. Let $\gamma_1$ and $\gamma_2$ be edge-disjoint witnesses for $\{0 \lra u\}$. Considering the final vertex $v$ of an open path witnessing $\{0 \lra n \mathbf{e}_1\}$ which lies in $\gamma_1 \cup \gamma_2$, we see
    \begin{align*}
        \prob\left(0 \Longleftrightarrow u, \, 0 \lra n e_1 \right)  &\leq \sum_v \tau(0, v) \tau(v, u) \tau(0, u) \tau(v, ne_1)\\
        &\leq C \langle u\rangle^{2-d} \sum_v \langle v\rangle^{2-d} \langle v - ne_1\rangle^{2-d} \langle u-v\rangle^{2-d}\ ,
    \end{align*} 
    and so \eqref{eq:finitebubb3} and hence the desired result follows from
    \begin{equation}
        \label{eq:finitebubb4} 
        \sum_v \langle v\rangle^{2-d} \langle v - ne_1\rangle^{2-d} \langle u-v\rangle^{2-d} \leq C n^{2-d} \langle u\rangle^{4-d} \quad \text{for $n \geq 2 |u|$},
    \end{equation}
    which in turn follows from \eqref{eqn: triple-sum}.
    
\end{proof}

\section{Three-point function scaling}\label{sec: 3pt-scaling}
We begin by proving the $k = 3$ case of Theorem~\ref{thm:kpt} and Theorem~\ref{thm: distances}; in other words, establishing the scaling of the (decorated) three-point function. This will provide us with the base case for an inductive argument which will show the general version of these theorems; see Section~\ref{sec:generalk} below. It also provides us an opportunity to develop our arguments in the simplest possible setting, which will make the structure of the general argument more transparent.

By translation-invariance, it suffices to consider
\[ \tau_3(0, x_1, x_2) = \mathbb{P}(0 \leftrightarrow x_1 \lra x_2) \]
and its decorated analogue.
For legibility, we organize our claim about $\tau_3$ into its own proposition:
\begin{proposition}\label{prop:threept}
Let $x_1 = \lfloor n y_1 \rfloor $ and $x_2 = \lfloor n y_2 \rfloor$ for $0, y_1, y_2$ distinct elements of $\mathbb{R}^d$. 
Recalling our definitions \eqref{eqn: alphadef} of $\alpha$, \eqref{eqn: betadef} of $\beta$, and \eqref{eq:rhodef} of $\rho$ above,
we have
\[ \lim_{n \rightarrow \infty} n^{2d-6}\mathbb{P}(0 \leftrightarrow x_1, x_2) = 2d \alpha^3 \beta^2 \rho \int_{\mathbb{R}^d}\frac{1}{|y|^{d-2}}\frac{1}{|y-y_1|^{d-2}}\frac{1}{|y-y_2|^{d-2}}\,\mathrm{d}y . \]
\end{proposition}
This will follow immediately from the more general result on decorated three-point functions. The role of the pivotal $g$ on the left-hand side of the expression corresponds to the point $z$ on the right-hand side; in a sense, this makes clear the the ``branch point'' of the connections from $0$ to $x_1$ and $x_2$ is the variable of integration in Proposition~\ref{prop:threept}.

\begin{proposition}\label{eqn:three-pt-decorated}
    For $\lambda_0,\lambda_1,\lambda_2 \ge 0$ and distinct $0,y_1,y_2$,
\begin{align*}
&\lim_{n\to\infty} n^{2d-6}\sum_{g\in \mathcal{E}^d}
\mathbb{E}\Big[e^{-an^{-2}(\lambda_0 D(0,\underline{g})+\lambda_1 D(\bar{g},x_1)+\lambda_2 D(\bar{g},x_2))},\,
\Gamma(0,x_1,x_2),\, g = \overline{\mathcal{P}}(0,\{x_1,x_2\})\Big]\\
= ~ &2d\,\alpha^3\beta^2\rho \int_{\mathbb{R}^d}
\frac{\psi_d(\lambda_0|z|^2)}{|z|^{d-2}}
\frac{\psi_d(\lambda_1|z-y_1|^2)}{|z-y_1|^{d-2}}
\frac{\psi_d(\lambda_2|z-y_2|^2)}{|z-y_2|^{d-2}}\, \mathrm{d}z.
\end{align*}
Dividing by the undecorated limit $(\lambda_i\equiv 0)$ in the previous proposition and using
\[z_d|u|^{2-d}\psi_d(\lambda|u|^2) = \int_0^\infty e^{-\lambda t}p_t(u)\,dt\]
identifies the limiting conditional joint density of the ``branch lengths'', the distances to the last common pivotal for the connections from $0$ to $x_1$ and from $0$ to $x_2$.
\end{proposition}

Our argument will be broken into several pieces in the subsections below. To illuminate the structure of the proof, we organize several technical pieces of the argument into lemmas which will be fully presented and proved in Sections~\ref{sec:aux}, \ref{sec:iic}, and \ref{sec:dec2}. The versions of these lemmas we will invoke while proving Proposition~\ref{prop:threept} are independent of that proposition and its proof. The lemmas are written in a general form; this allows us to apply them to the case of the general $k$-point functions in the proof of Theorems~\ref{thm:kpt} and \ref{thm: distances}.

To orient the reader, we provide here a summary of places where the lemmas of Sections~\ref{sec:aux}, \ref{sec:iic}, and \ref{sec:dec2} are invoked. At~\eqref{eq:truncpreview}, we use Lemma~\ref{lem:newcut} to express the impact of closing a pivotal edge. At \eqref{eqn: PL}, we apply Lemma~\ref{lem: GT-truncation} to approximate an event that nearby clusters do not intersect by a cylinder event. At \eqref{eq:genlempreview}, we apply the result of Lemma~\ref{lem:genlem1}.

We refer to  our enhanced IIC convergence result, Lemma~\ref{lem:iicnew2} at \eqref{eq:pig} to make the parallels with the $k > 3$ arguments clear. However, as noted at \eqref{eq:pig}, the existing IIC result \eqref{eqn: IIC-conv2} also suffices in the case $k = 3$.
\subsection{Preliminary steps}
Before beginning the proof, we slightly rephrase the claim of Proposition~\ref{eqn:three-pt-decorated}. The integrals appearing in the proposition's statement naturally arise from a sum over the location where connections from $0$ to $x_1$ and from $0$ to $x_2$ branch. To make this precise, we first ensure that a pivotal edge exists at which such branching occurs. Recalling Definition~\ref{defin:pivs}, we note that when $\{0 \lra x_1\}$ and $\{0 \lra x_2\}$ occur but there is no common pivotal, then in fact $\{0 \lra x_1 \} \circ \{0 \lra x_2\}$ occurs. The BK inequality then shows
\begin{equation}\label{eqn: piv-exists}
\prob\left(0 \lra x_1 \lra x_2, \quad \opiv(0, \{x_1, x_2\}) = \varnothing\right) \leq C n^{4-2d}
\end{equation}
for a $C = C(y_1, y_2)$. 

We define
\[\mathcal{W}^{(\lambda_0,\lambda_1,\lambda_2)}:=e^{-n^{-2}a(\lambda_0 D(0,\underline{g})+\lambda_1 D(\bar{g},x_1)+\lambda_2 D(\bar{g},x_2))},\]
noting that this is not to be confused with our notation for the incipient infinite cluster.
Proposition~\ref{eqn:three-pt-decorated} will follow once we have shown
\begin{equation}
    \label{eq:threept2}
    \begin{split}
    &\lim_{n \rightarrow \infty} n^{2d-6} \sum_{g \in \edges} \mathbb{E}\big[\W^{(\lambda_0,\lambda_1,\lambda_2)}, \, 0 \leftrightarrow x_1, x_2 ;\, g = \opiv(0, \{x_1, x_2\}) \big]\\
    = &~ 2d\,\alpha^3\beta^2\rho \int_{\mathbb{R}^d}
\frac{\psi_d(\lambda_0|x|^2)}{|x|^{d-2}}
\frac{\psi_d(\lambda_1|x-y_1|^2)}{|x-y_1|^{d-2}}
\frac{\psi_d(\lambda_2|x-y_2|^2)}{|x-y_2|^{d-2}}\, \mathrm{d}x .
    \end{split}
\end{equation}
To show \eqref{eq:threept2} we use the fact, presented as Lemma~\ref{lem:newcut} below, that we may re-express each term in the sum on the left-hand side of \eqref{eq:threept2}:
\begin{equation}
\label{eq:truncpreview}
\begin{split}
             &\mathbb{E}[\W^{(\lambda_0,\lambda_1,\lambda_2)}, 0 \leftrightarrow x_1, x_2 ;\, g = \opiv(0, \{x_1, x_2\}) ]\\
             =~&\beta \mathbb{E}[e^{-an^{-2}\lambda_0 D^ {\mathbb{Z}^d\setminus \clust(\overline{g})}(0,\underline{g})} e^{-an^{-2}[\lambda_1 D(\bar{g},x_1)-\lambda_2 D(\bar{g}, x_2)]},\Gamma(\bar{g},x_1,x_2), \mathcal{P}_{\bar{g}}].
            \end{split}
\end{equation}
This is the origin of the first factor of $\beta$ appearing on the right side of \eqref{eq:threept2}.
Here, for each $v \in\mathbb{Z}^d$, we introduced shorthand for the event that $\upiv(v, \{x_1, x_2\})$ does not exist:
\begin{equation}
\label{eq:mathcalpsubv}
\mathcal{P}_v=\{\mathcal{P}(v,x_1)\cap \mathcal{P}(v,x_2) = \varnothing\} = \{\mathcal{P}(v,\{x_1, x_2\}) = \varnothing\}.
\end{equation}

We make another adjustment to \eqref{eq:threept2} before proceeding with the proof. It will be helpful for $g$ to be ``macroscopically'' far away from $0$, $x_1$, and $x_2$. For fixed $y_1, y_2$, we define
\begin{equation}
    \label{eq:farregime}
    F(\epsilon, n):= \{g \in \edges:\, \min\{|\underline{g}|,|x_1-\underline{g}|,|x_2-\underline{g}|\} \geq \epsilon n, \, |\underline{g}| \leq \epsilon^{-1} n\}\ .
\end{equation}
This represents the edges which are ``far'' from $0$, $x_1$, and $x_2$, as well as ``from infinity''.

We show that the sum in \eqref{eq:threept2} can practically be taken over $g \in F(\epsilon, n)$ for small $\epsilon$. This is the content of the following lemma:
\begin{lemma}
    \label{lem:nearregime}
    \[\lim_{\epsilon \to 0}\limsup_{n \rightarrow \infty} n^{2d-6} \sum_{g \notin F(\epsilon, n)} \mathbb{P}(0 \leftrightarrow x_1, x_2 ;\, g = \opiv(0, \{x_1, x_2\}) ) =0\ .\]
\end{lemma}
Assuming the veracity of Lemma~\ref{lem:nearregime}, and noting that $0\le W^{(\lambda_0,\lambda_1,\lambda_2)}\le 1$, the proof of Proposition~\ref{eqn:three-pt-decorated} will be complete once we show 
\begin{equation}
    \label{eq:threept3}
    \begin{split}
    &\lim_{\epsilon \to 0} \lim_{n \rightarrow \infty} n^{2d-6} \sum_{g \in F(\epsilon, n)} \mathbb{E}\big[ e^{-an^{-2}(\lambda_0 D(0,\underline{g})+\lambda_1 D(\bar{g},x_1)+\lambda_2 D(\bar{g},x_2))} , \,0 \leftrightarrow x_1, x_2 ;\, g = \opiv(0, \{x_1, x_2\})\big]\\
    =~&2d\,\alpha^3\beta^2\rho \int_{\mathbb{R}^d}
\frac{\psi_d(\lambda_0|z|^2)}{|z|^{d-2}}
\frac{\psi_d(\lambda_1|z-y_1|^2)}{|z-y_1|^{d-2}}
\frac{\psi_d(\lambda_2|z-y_2|^2)}{|z-y_2|^{d-2}}\, \mathrm{d}z.
    \end{split}
\end{equation}
We prove Lemma~\ref{lem:nearregime} in Section \ref{sec:near3}; then, in Section~\ref{sec:far3} below, we complete the proof of Proposition~\ref{prop:threept} by establishing \eqref{eq:threept3}.

\subsection{Near-regime}\label{sec:near3}
In this section, we prove Lemma~\ref{lem:nearregime} via a diagrammatic estimate.
\begin{proof}[Proof of Lemma~\ref{lem:nearregime}]
We show that there is a $C = C(y_1, y_2) > 0$ uniform in $n$ and in $\epsilon$ small relative to $|y_1|, |y_2|$ such that
    \begin{equation}\sum_{g \notin F(\epsilon, n)} \mathbb{P}(0 \leftrightarrow x_1, x_2 ;\, g = \opiv(0, \{x_1, x_2\}) )  \leq C \epsilon^2 n^{6-2d}\ .\label{eq:keyest} \end{equation}
    This clearly suffices to complete the proof.

The sum in \eqref{eq:keyest} is bounded, using \eqref{eq:truncpreview} (that is,~Lemma~\ref{lem:newcut}) by
\begin{align*}
&\beta  \sum_{g \notin F(\epsilon,n)} \mathbb{P}(0 \leftrightarrow \underline{g}\,\text{off}\,\clust(\overline{g}), \overline{g}\leftrightarrow x_1,\overline{g}\leftrightarrow x_2, \mathcal{P}(\overline{g},x_1)\cap\mathcal{P}(\overline{g},x_2)=\varnothing)\\
\leq & \beta \sum_{g \notin F(\epsilon,n)} \prob(\{0 \leftrightarrow \underline{g}\} \circ \{\overline{g}\leftrightarrow x_1\} \circ\{\overline{g}\leftrightarrow x_2\}) \ .
\end{align*}
Using \eqref{eqn: BK} and the two-point function bound \eqref{eqn: HHS}, we can bound the expression in the last display by the following, up to a constant factor:
\begin{align*}
    &\sum_{|\underline{g}| < \epsilon n} \langle \underline{g}\rangle^{2-d}\langle x_1-\overline{g}\rangle^{2-d}\langle x_2-\overline{g}\rangle^{2-d}
    +\sum_{|x_1 - \underline{g}| < \epsilon n} \langle \underline{g}\rangle^{2-d}\langle x_1-\overline{g}\rangle^{2-d}\langle x_2-\overline{g}\rangle^{2-d}\\
    +&\sum_{|x_2 - \underline{g}| < \epsilon n} \langle \underline{g}\rangle^{2-d}\langle x_1-\overline{g}\rangle^{2-d}\langle x_2-\overline{g}\rangle^{2-d}\ + \sum_{|\underline{g}| > \epsilon^{-1} n} \langle \underline{g}\rangle^{2-d}\langle x_1-\overline{g}\rangle^{2-d}\langle x_2-\overline{g}\rangle^{2-d}
\end{align*}
Extracting factors corresponding to connections of length $O(n)$, for $\epsilon > 0$ small relative to $|y_1|$ and $|y_2|$, we have the bound
\[ C \cdot n^{4-2d} \left(\sum_{|\underline{g}| < \epsilon n} \langle \overline{g}\rangle^{2-d}+ \sum_{|x_1 - \overline{g}| < \epsilon n} \langle x_1-\underline{g}\rangle^{2-d} + \sum_{|x_2 - \overline{g}| < \epsilon n} \langle x_2-\underline{g}\rangle^{2-d}\right) + C\sum_{|\overline{g}| > \epsilon^{-1} n} \langle g\rangle^{6-3d},\]
which in turn is bounded by
\[ C \epsilon^2 n^{6-2d} \]
uniformly in $n$ and small $\epsilon > 0$, as claimed in \eqref{eq:keyest}.
\end{proof}

\subsection{Far-regime \label{sec:far3}}
With Lemma~\ref{lem:nearregime} proved, we proceed to prove Proposition~\ref{prop:threept}. 
For clarity, we recall and introduce some notation. We recall Definition~\eqref{defin:gammadef} in the special case of three vertices:
\begin{equation}\label{eq:gammarecall}
\Gamma(v,x_1,x_2):=\{v\leftrightarrow x_1 \leftrightarrow x_2\}.
\end{equation}

\begin{proof}[Proof of Proposition~\ref{eqn:three-pt-decorated}]
We prove Proposition~\ref{eqn:three-pt-decorated}, as announced, by establishing \eqref{eq:threept3}. We consider the sum appearing on the left-hand side of \eqref{eq:threept3}.
We introduce a new parameter $K \geq 1$ fixed relative to $n$ which will appear in intermediate expressions. We will ultimately take $n \to \infty$ with fixed $K, \epsilon$, then take $K \to \infty$, finally taking $\epsilon \to 0$, showing
\begin{equation}
\begin{split}
    \label{eq:threept3k}
    &\lim_{\epsilon \to 0} \lim_{K \to \infty} \lim_{n \rightarrow \infty} n^{2d-6} \sum_{g \in F(\epsilon, n)} \mathbb{E}\left[\W^{(\lambda_0,\lambda_1,\lambda_2)}, \Gamma(\bar{g},x_1, x_2),\, g = \opiv(0, \{x_1, x_2\})\right] \\
    =&~2d\,\alpha^3\beta^2\rho \int_{\mathbb{R}^d}
\frac{\psi_d(\lambda_0|y|^2)}{|y|^{d-2}}
\frac{\psi_d(\lambda_1|y-y_1|^2)}{|y-y_1|^{d-2}}
\frac{\psi_d(\lambda_2|y-y_2|^2)}{|y-y_2|^{d-2}}\, \mathrm{d}y  . 
    \end{split}
\end{equation}

 We again apply \eqref{eq:truncpreview}  (Lemma \ref{lem:newcut}) to the sum appearing in \eqref{eq:threept3} and \eqref{eq:threept3k}. We then introduce an independent copy $\tilde{\mathbb{P}}$ of the percolation measure from which we sample the cluster of $\overline{g}$, treating this cluster as fixed when we sample the cluster of $0$ from $\prob$.  This yields
\begin{equation}\label{eqn: PL}
\begin{split}
&\beta \sum_{g \in F(\epsilon,n)} \mathbb{E}\big[e^{-an^{-2}\lambda_0 D^ {\mathbb{Z}^d\setminus \clust(\overline{g})}(0,\underline{g})} \W^{(\lambda_1,\lambda_2)}_{\bar{g}}, 0 \leftrightarrow \underline{g}\,\text{off}\,\clust(\overline{g}), \Gamma(\bar{g},x_1,x_2),\mathcal{P}_{\overline{g}}\big]\\
&= \beta  \sum_{g \in F(\epsilon,n)} \wto{\mathbb{E}}\Big[\W^{(\lambda_1,\lambda_2)}_{\bar{g}}, x_1 \lra \overline{g}  \circ x_2 \lra \overline{g},\, \mathbb{E}\big[e^{-an^{-2}\lambda_0 D^ {\mathbb{Z}^d\setminus \wto{\clust}(\overline{g})}(0,\underline{g})}, 0 \sa{\Z^d \setminus \wto{\clust}(\overline{g})} \underline{g} \big]\Big] - C n^{6-2d} K^{(6-d)/d}\ .
\end{split}
\end{equation}
We now have the lower bound
\[\geq \beta  \sum_{g \in F(\epsilon,n)} \wto{\mathbb{E}}\Big[\W^{(\lambda_1,\lambda_2)}_{\bar{g}}, x_1 \lra \overline{g}  \circ x_2 \lra \overline{g},\, \mathbb{E}\big[e^{-an^{-2}\lambda_0 D^ {\mathbb{Z}^d\setminus \wto{\clust}(\overline{g})}(0,\underline{g})}, 0 \sa{\Z^d \setminus \wto{\clust}_{B(\overline{g}; 2K)}(\overline{g})} \underline{g} \big]\Big] - C n^{6-2d} K^{(6-d)/d}\ .\]

where
\[\W^{(\lambda_1,\lambda_2)}_{\bar{g}}=e^{-a n^{-2} \lambda_1  D(\bar{g},x_1)-a n^{-2} \lambda_2  D(\bar{g}, x_2)}\]
and where in the second line of \eqref{eqn: PL} we used Lemma~\ref{lem: GT-truncation} (see \eqref{eq:ktrunc3pt2} below the statement of that lemma). We note that the factor $\W^{(\lambda_1,\lambda_2)}_{\bar{g}}$ now depends on the realization sampled from $\widetilde{\mathbb{P}}$; this issue will recur several times, and we trust that which measure we sample our $\W$-weights from will be clear from context. The second term in the second line of \eqref{eqn: PL} will not contribute to the limit appearing in \eqref{eq:threept3k} because the second term is much smaller than $n^{6-2d}$ for $K$ large.
The inequality
\begin{equation}
\begin{split}
&\beta \sum_{g \in F(\epsilon,n)} \mathbb{E}\big[e^{-n^{-2}\lambda_0 D^ {\mathbb{Z}^d\setminus \clust(\overline{g})}(0,\underline{g})} \W^{(\lambda_1,\lambda_2)}_{\bar{g}}, 0 \leftrightarrow \underline{g}\,\text{off}\,\clust(\overline{g}), \Gamma(\bar{g},x_1,x_2),\mathcal{P}_{\overline{g}}\big]\\
\leq~& \beta  \sum_{g \in F(\epsilon,n)} \wto{\mathbb{E}}\Big[\W^{(\lambda_1,\lambda_2)}_{\bar{g}}, x_1 \lra \overline{g}  \circ x_2 \lra \overline{g},\, \mathbb{E}\big[e^{-an^{-2}\lambda_0 D^ {\mathbb{Z}^d\setminus \wto{\clust}(\overline{g})}(0,\underline{g})}, 0 \sa{\Z^d \setminus \wto{\clust}_{B(\overline{g}; 2K)}(\overline{g})} \underline{g} \big]\Big]
\end{split}
\end{equation}
is trivial; we thus focus on the first term of the right-hand side of \eqref{eqn: PL}, showing 
\begin{equation}
    \label{eqn: PL2}
    \begin{gathered}
   \lim_{\epsilon \to 0} \lim_{K \to \infty} \lim_{n \rightarrow \infty} \frac{\beta}{n^{6-2d}}  \sum_{g \in F(\epsilon,n)} \wto{\mathbb{E}}\Big[\W^{(\lambda_1,\lambda_2)}_{\bar{g}}, x_1 \lra \overline{g}  \circ x_2 \lra \overline{g},\, \mathbb{E}\big[e^{-an^{-2}\lambda_0 D^ {\mathbb{Z}^d\setminus \clust(\overline{g})}(0,\underline{g})}, 0 \sa{\Z^d \setminus \wto{\clust}_{B(\overline{g}; 2K)}(\overline{g})} \underline{g} \big]\Big]\\
   = 2d \alpha^3 \beta^2\rho \int_{\mathbb{R}^d}
\frac{\psi_d(\lambda_0|x|^2)}{|x|^{d-2}}
\frac{\psi_d(\lambda_1|x-x_1|^2)}{|x-x_1|^{d-2}}
\frac{\psi_d(\lambda_2|x-x_2|^2)}{|x-x_2|^{d-2}}\, \mathrm{d}x  ,
   \end{gathered}
\end{equation}
which will complete the proof of the theorem.

We rewrite the left-hand side of \eqref{eqn: PL2} by partitioning over admissible values $\mathcal{D}$ of $\wto{\clust}_{B(\overline{g}; 2K)}$. 
\begin{align}
    &\beta  \sum_{g \in F(\epsilon,n)} \wto{\mathbb{E}}\left[\W^{(\lambda_1,\lambda_2)}_{\bar{g}}, x_1 \lra \overline{g}  \circ x_2 \lra \overline{g}, \mathbb{E}\Big[e^{-an^{-2}\lambda_0 D^ {\mathbb{Z}^d\setminus \clust(\overline{g})}(0,\underline{g})} , 0 \sa{\Z^d \setminus \wto{\clust}_{B(\overline{g}; 2K)}(\overline{g})} \underline{g} \Big]\right]\nonumber\\
    =~&\beta \sum_{g \in F(\epsilon,n)} \sum_{\mathcal{D}}  \wto{\mathbb{E}}\left[\W^{(\lambda_1,\lambda_2)}_{\bar{g}},\mathbbm{1}_{\wto{\clust}_{B(\overline{g}; 2K)}(\overline{g}) = \mathcal{D}}\mathbbm{1}_{\{x_1 \lra \overline{g} \} \circ \{x_2 \lra \overline{g}\}} \mathbb{E} \Big[e^{-an^{-2}\lambda_0 D^ {\mathbb{Z}^d\setminus \clust(\overline{g})}(0,\underline{g})} , 0 \sa{\Z^d \setminus \mathcal{D}} \underline{g} \Big]\right]\ .\label{eqn: apply-IIC}
\end{align}

We next replace $D^{\mathbb{Z}^d\setminus \clust(\overline{g})}(0,\underline{g})$ with $D^{\mathbb{Z}^d\setminus \clust_{B(\overline{g}; 2K)}(\overline{g})}(0,\underline{g})$. We note that these two distances can differ only on the event
\[\left\{\exists w \in \wto{\clust}(\overline{g}) \setminus \wto{\clust}_{B(2K)}(\overline{g}): \, \text{ some open path $\gamma: \, 0 \to \underline{g} $ has $w \in \gamma$}\right\}. \]

Applying  Theorem \ref{thm: avoid-dist}, we have
\[\mathbbm{1}_{\wto{\clust}_{B(\overline{g}; 2K)}(\overline{g}) = \mathcal{D}}\mathbb{E}\big[e^{-an^{-2}\lambda_0 D^ {\mathbb{Z}^d\setminus \clust(\overline{g})}(0,\underline{g})}, 0 \sa{\Z^d \setminus \wto{\clust}_{B(\overline{g}; 2K)}(\overline{g})} \underline{g} \big]=\tau(0,\underline{g}) \nu_{\underline{g}}(\Xi(\mathcal{D}))\psi_d(\lambda_0|\underline{g}|^2)(1+o(1))\]
uniformly for $g\in F(\epsilon,n)$ with $\epsilon$ and $K$ fixed, so we obtain as $n \to \infty$:
\begin{align*}
   \lim_{n \to \infty} \frac{\sum_{g \in F(\epsilon,n)} \sum_{\mathcal{D}}  \wto{\mathbb{E}}\left[\W^{(\lambda_1,\lambda_2)}_{\bar{g}}, \mathbbm{1}_{\wto{\clust}_{B(\overline{g}; 2K)}(\overline{g}) = \mathcal{D}}\mathbbm{1}_{\{x_1 \lra \overline{g} \} \circ \{x_2 \lra \overline{g}\}} \mathbb{E} \Big[e^{-an^{-2}\lambda_0 D^ {\mathbb{Z}^d\setminus \clust(\overline{g})}(0,\underline{g})} , 0 \sa{\Z^d \setminus \mathcal{D}} \underline{g} \Big]\right]}{\sum_{g\in F(\epsilon,n)} \tau(0,\underline{g})  \psi_d(\lambda_0|\underline{g}|^2) \sum_{\mathcal{D}} \nu_{\underline{g}}(\Xi_{\underline{g}}(\mathcal{D}))\widetilde{\mathbb{E}}[\W^{(\lambda_1,\lambda_2)}_{\bar{g}}, \Gamma(\bar{g},x_1,x_2), \mathcal{P}_{\bar{g}}, \mathcal{D}=\wto{\clust}_{B(\overline{g};2K)}(\overline{g})]} = 1
\end{align*}
for fixed $\epsilon$ and $K$. It therefore suffices to show that the denominator of the last display approaches the right-hand side of \eqref{eq:threept3k} when we take $n$, then $K$ to infinity, followed by taking $\epsilon \to 0$.

Performing the sum over $\mathcal{D}$, that denominator is
\begin{align}
    \sum_{g\in F(\epsilon,n)} \tau(0,\underline{g})\psi_d(\lambda_0|\underline{g}/n|^2)  \widetilde{\mathbb{E}}\left[ \W^{(\lambda_1,\lambda_2)}_{\bar{g}}\nu_{\underline{g}}(\Xi_{\underline{g}}(\widetilde{\clust}_{B(\overline{g};2K)}(\overline{g}))) \mathbbm{1}_{\Gamma(\bar{g},x_1,x_2)} \mathbbm{1}_{\mathcal{P}_{\bar{g}}}\right] =: \sum_{g\in F(\epsilon,n)} \Pi(g)\ , \label{eq:pig}
\end{align} 
where each term of the sum on the left-hand side defines the quantity $\Pi(g)$ in the sum on the right-hand side.

Recall the definition of $\rho$ in \eqref{eq:rhodef}. Lemma~\ref{lem:genlem1} below shows that 
\begin{equation}
    \label{eq:genlempreview}
    \lim_{K \to \infty} \limsup_{n \to \infty} \sup_{g \in F(\epsilon, n) }\frac{\left|\Pi(g) - \beta \rho \tau(0, \underline{g}) \tau(\overline{g}, x_1) \tau(\overline{g}, x_2)\psi_d(\lambda_1|\underline(x_1-{g})/n|^2)\psi_d(\lambda_2|\underline(x_2-{g})/n|^2)\right|}{ \tau(\overline{g}, x_1) \tau(\overline{g}, x_2)\psi_d(\lambda_1|(x_1-{g})/n|^2)\psi_d(\lambda_2|(x_2-{g})/n|^2)} = 0
\end{equation}
for each fixed $\epsilon > 0$. 
Pulling \eqref{eq:genlempreview} together with \eqref{eq:pig} and feeding back into \eqref{eqn: PL}, we have the following facts about the left-hand side of \eqref{eq:threept3}:
\begin{align}
    &\lim_{\epsilon \to 0} \lim_{n \rightarrow \infty} n^{2d-6} \sum_{g \in F(\epsilon, n)} \mathbb{E}\big[ e^{-an^{-2}(\lambda_0 D(0,\underline{g})+\lambda_1 D(\bar{g},x_1)+\lambda_2 D(\bar{g},x_2))} , \,0 \leftrightarrow x_1, x_2 ;\, g = \opiv(0, \{x_1, x_2\})\big] \nonumber\\
    =~& \beta^2 \rho \lim_{\epsilon \to 0} \lim_{n \rightarrow \infty} n^{2d-6} \sum_{g \in F(\epsilon, n)} \tau(0, \underline{g}) \tau(\overline{g}, x_1) \tau(\overline{g}, x_2)\psi_d(\lambda_0|\underline{g}/n|^2)\psi_d(\lambda_1|(x_1-\underline{g})/n|^2)\psi_d(\lambda_2|(x_2-\underline{g})/n|^2)\nonumber\\
    =~& \alpha^3 \beta^2 \rho \lim_{\epsilon \to 0} \lim_{n \rightarrow \infty} n^{2d-6} \sum_{g \in F(\epsilon, n)} \langle \underline{g}\rangle^{2-d} \langle\overline{g}- x_1\rangle^{2-d} \langle\overline{g}- x_2\rangle^{2-d}\label{eq:riemann} \\
    &\times \psi_d(\lambda_0|\underline{g}/n|^2)\psi_d(\lambda_1|(x_1-\underline{g})/n|^2)\psi_d(\lambda_2|(x_2-\underline{g})/n|^2)\ , \nonumber
\end{align}
assuming the limit exists. In the third line, we used \eqref{eqn: HHS}. The above is a Riemann sum approximation to an appropriate integral. Letting 
\[J(\epsilon) =  \{z \in \mathbb{R}^d: \, \epsilon^{-1} >  |z| > \epsilon, \, |z - y_1| > \epsilon, \, |z - y_2| > \epsilon \}\ , \]
it follows that \eqref{eq:riemann} is equal to
\begin{equation}
    \label{eq:riemann2}
    \begin{split}
    &2d \alpha^3 \beta^2 \rho \lim_{\epsilon \to 0}  \int_{J(\epsilon)} |z|^{2-d} |z - y_1|^{2-d} |z- y_2|^{2-d} \psi_d(\lambda_0|z|^2)\psi_d(\lambda_1|z-y_1|^2)\psi_d(\lambda_2|z-y_2|^2)\, \mathrm{d}z\\
    =~&(2d) \alpha^3 \beta^2 \rho  \int_{\mathbb{R}^d} |z|^{2-d} |z - y_1|^{2-d} |z- y_2|^{2-d} \psi_d(\lambda_0|z|^2)\psi_d(\lambda_1|z-y_1|^2)\psi_d(\lambda_2|z-y_2|^2)\, \mathrm{d}z\ .
    \end{split}
\end{equation}

From the fact that \eqref{eq:riemann} is identical to \eqref{eq:riemann2}, the claimed equality \eqref{eq:threept3} immediately follows, and the proposition is proved.
\end{proof}

\section{$k$-point convergence} \label{sec:generalk}

\subsection{Connectivity tree}\label{sec: connect-tree}
In this section, we associate to each configuration $\omega\in \Gamma(x_0, \dots, x_{k-1})$ a canonical directed tree $T(\omega)$ encoding the minimal branching structure in the connections implied by $\Gamma(x_0,\ldots,x_{k-1})$. The purpose of this construction is to organize subevents of $\Gamma(x_0,\ldots,x_{k-1})$ according to the tree structure $T(\omega)$ they generate.

Let $\omega\in \Gamma(x_0,\ldots, x_{k-1})$. 
\begin{enumerate}
    \item For each $1\le i\le k-1$, let
    \[\mathcal{B}_i:=\{\underline{e}: e\text{ is pivotal for } x_i\lra x_0\},\]
    where we denote by $\underline{e}$ the endpoint of $e$ first encountered on a path from $x_i$ to $x_0$. 
    This is a linearly ordered set, since every path must traverse the pivotal edges in the same order and orientation: for $u,v\in \mathcal{B}_i$, we write $u \preceq_i v$ if $v$ appears after $u$ in every open path  $x_i\rightarrow x_0$. We write $u\prec_i v$ if $u\preceq_i v$ and $u\neq v$. For vertices $u,v$ we write $u\prec v$ if $u\prec_i v$ for some $i$.
    \item For $i\neq j$, define
    \[m_{ij}:=\begin{cases}
    \text{minimal element of } \mathcal{B}_i\cap \mathcal{B}_j,& \ \text{if } \mathcal{B}_i\cap\mathcal{B}_j\neq \emptyset\\
    x_0, &\text{ if } \mathcal{B}_i\cap \mathcal{B}_j=\emptyset. 
    \end{cases}.\]
    \item We define 
    \[V(T):=\{x_0,\ldots, x_{k-1}\}\cup \{m_{ij}, 1\le i\neq j\le k-1\},\]
    identifying equal vertices.
    \item For each $v\in V(T)\setminus \{x_0\}$, define its parent
    \[p(v)=\mathrm{min}\{w\in V(T): v\prec w\}.\]
    \item \[E(T)=\{(v,p(v)):\,v\in V(T)\setminus \{x_0\}\}\] consists of all oriented edges $v\rightarrow p(v)$ with $v\in V(T)$. The root is $x_0$.
\end{enumerate}

\begin{proposition}
    For every $\omega\in \Gamma(x_0,\ldots,x_{k-1})$, the connectivity tree $T(\omega)$ is a directed tree with leaves in the set $\{x_0,\ldots,x_{k-1}
    \}$. $T$ is also a tree when viewed as an undirected graph.
    \begin{proof}
    By construction, every vertex in $T$ except $x_0$ has a unique parent. Starting at any vertex $v$ and applying the parent operation repeatedly, we obtain a path in $T$ ending at $x_0$. This implies that $T$ is connected and $\# E(T)=\#V(T)-1$, so  $T$ is a tree. Since $m_{ij}$ appears on any path from $x_i$, only the vertices $x_j$, $0\le j\le k-1$ can have 0 children.
    \end{proof}
\end{proposition}

\begin{definition}\label{def: De}
    Fix a configuration $\omega\in \Gamma(x_0,\ldots,x_{k-1})$ with connectivity tree $T = T(\omega)$. For an interior vertex $v$, let $g_v\in \mathcal{E}(\mathbb{Z}^d)$ be the pivotal edge in $\omega$ corresponding to $v$, oriented from the root to the descendant. We call $g_v$ the \emph{realizing pivotal} of the interior vertex $v\in V(T)$.

    For an edge $e$ of the tree, we define
    \begin{itemize}
        \item if $e$ joins the root $x_0$ to an interior vertex $v$, then
        \[D_e:=\mathrm{dist}(x_0,\underline{g}_v),\]
        \item if $e$ joins two interior vertices $v$ and $w$, then
        \[D_e=\mathrm{dist}(\overline{g}_v,\underline{g}_w),\]
        \item if $e$ joins an interior vertex to a leaf $x_i$, then
        \[D_e=\mathrm{dist}(\overline{g}_v,x_i).\]
    \end{itemize}
    Here, $\mathrm{dist}$ refers to the chemical distance in the configuration $\omega$.
\end{definition}
An immediate consequence of this definition is that, under the hypotheses above
\[D(x_0,x_i)=\sum_{e\in \mathrm{path}(0,i)}D_e+\#\{\text{interior vertices on path}(0,i)\}.\]

\subsection{The connectivity tree is regular with high probability \label{sec:conntree}}
For $k \geq 3$, let 
$\mathfrak{T}_{k}$
be the set of oriented trees with $k$ leaves labeled by the symbols $0, 1,\, \dots, k-1$, with the property that each non-leaf vertex has in-degree two and such that every edge is oriented toward $0$. Trees which are isomorphic as directed graphs but have different labelings are regarded as distinct elements of $\mathfrak{T}_k$.
We now define an event $E_{\mathrm{deg}}$ such that if $\omega \in \Gamma(x_0,\ldots,x_{k-1})\cap E_{\mathrm{deg}}^c$, $T(\omega)$ can be identified with an element $\mathcal{T}(\omega)\in \mathfrak{T}_{k+1}$ obtained by replacing the label $x_i$ with $i$ for $0 \leq i \leq k$, because in that case, all non-leaf vertices of $T(\omega)$ exhibit binary branching. 

We now introduce the main regime under which we prove our results. Let $\epsilon>0$. 

\begin{definition}
    The set of far-regime points  $G(\epsilon, n)$ is defined by
    \begin{equation}
    G(\epsilon,n):=\{(x_0,x_1,\dots,x_{k-1}) \in (\mathbb{Z}^d)^k:  \max_i |x_i|\le \epsilon^{-1}n, \,\min_{i\neq j} |x_i - x_j| \geq \epsilon n\}.
    \label{farregime-def}
    \end{equation}
\end{definition}
We define $E_{\mathrm{deg}}$ to be the event that:
\begin{enumerate}
\item there is $0 \le i \le k-1$ such that $x_i$ is not a leaf in $T(\omega)$, or
\item some vertex of $T(\omega)$ has in-degree $\ge 3$.
\end{enumerate}

The core result of this section is the following.
\begin{lemma} \label{lem:binbran}
Assume $d>6$ and $x_0,\ldots, x_{k-1}$ are in the far-separation regime \eqref{farregime-def}. Let $E_{\mathrm{deg}}$ denote the event defined above. Then, there is a constant $C(\epsilon,k)$ such that:
\begin{align}\label{eqn: Edeg-bd}
\mathbb{P}(E_{\mathrm{deg}}, \Gamma(x_0, \dots, x_{k-1})) &\leq C n^{(4-d)(k-1) - 2}\times \begin{cases}
    n^{-1}& \quad d=7\\
    (\log n) n^{-2}& \quad d =8\\
    n^{-2}&\quad d>8
    \end{cases}.
\end{align}
\end{lemma}
The proof appears at the end of the next section, after we develop several auxiliary results used in the proof.

\subsection{Diagrammatic Lemmas}\label{sec:diagrams}
Throughout this section, as well as Section \ref{sec:inductive}, we repeatedly use a familiar strategy in high-dimensional percolation to convert estimates into sums over diagrams. The method proceeds in two steps. 

First, we identify disjoint open connections that must be present in a given configuration, as in the next Lemma \ref{lem: Diagram}. Second, we apply the BK inequality \eqref{eqn: BK} to factor the probability into a product of two-point functions, then bound each factor using \eqref{eqn: HHS}. The resulting expression is a sum over the positions of internal vertices of a \emph{diagram}, a graph whose edges carry factors $\langle u-v\rangle^{2-d}$, which we evaluate these sums using the convolution estimates \eqref{eqn: std-conv} and \eqref{eqn: triple-sum}. 

The diagrammatic lemmas below (Propositions \ref{prop: contract}, \ref{prop: interior-delta} and their corollaries) systematize this evaluation: they show that each internal vertex of a tree-shaped diagram can be ``contracted'' at a cost of $n^{4-d}$ per internal vertex. The main point of the current subsection is that diagrams containing cycles produce terms of strictly smaller order than the leading tree-shaped contributions.

\begin{lemma}\label{lem: Diagram}
Let $v$ be an interior vertex of the connectivity tree $T(\omega)$. Suppose $v$ has $m$ children, $w_1,\ldots, w_m$ in $T$, that is, $p(w_i)=v$ for $1\le i\le m$.
\begin{enumerate}
    \item The configuration $\omega$ contains an open tree spanning the subset of $x_i$ corresponding to the subset of $\{x_0,\ldots,x_{k-1}\}$ in $T_{\prec v}$, the part of $T$ below $v$: $T_{\prec v}:=\{w\in V(T): w\prec v\}$.
    \item For any $1\le i\neq j\le m$, the configuration $\omega$ contains two edge-disjoint paths $\pi_1: w_1\rightarrow v$, resp. $\pi_2: w_2\rightarrow v$, between $w_i$ and $v$, respectively $w_j$ and $v$.
    
    \item If $m \ge 3$,  then choosing $\pi_1$ and $\pi_2$ corresponding to $w_1$ and $w_2$ respectively, in $\omega$ additionally:
    \begin{enumerate}
        \item  for each $j\ge 3$, there are vertices $a,b,c\in \mathbb{Z}^d$ such that $a\in \pi_1$, $b\in \pi_2$, such that 
        \[\omega\in \{v\lra a\}\circ \{a\lra c\}\circ \{w_j\lra c\}\circ \{c \lra b\}\circ \{b\lra v\},\]
        so that in particular $\gamma: v\rightarrow a\rightarrow c\rightarrow b\rightarrow v$ forms an open cycle,
        \item each $w_j$, $1\le j \le m$ lies in an open rooted tree, whose leaf set lies in $\{x_i\}_{0\le i\le k}$, attached to $\gamma$ at a single vertex, but is edge disjoint from it. For $j \geq 3$, choosing $a, b, c$ as above, we have that $w_1$ lies in the tree attached at $a$, $w_2$ in the tree attached at $b$, and $w_j$ in the tree attached at $c$.
    \end{enumerate}
\end{enumerate}
\begin{proof}
Let $v\in T$ and $e(v)$ be the unique pivotal edge emanating from $v$. Closing $e(v)$ disconnects some subset $C=\{y_1,\ldots,y_\ell\}\subset \{x_1,\ldots,x_k\}$ from $x_0$. There exists an open path $t_1$ in $\omega$ from $y_1$ to $v$. We then construct a spanning tree $t_m$ iteratively by attaching the remaining $y_m$ to the tree by the portion of the path $y_m\rightarrow v$ until the first point it hits $t_{m-1}$, until $C$ is exhausted.

If $w_i$, $w_j$ are direct descendants of $v$, then $v$ coincides with $m_{ij}$. The existence of $\pi_1$ and $\pi_2$ follows directly from this. 

For the third item, for $j \geq 3$ we select a path $\pi_j: w_j\rightarrow v$ disjoint from $\pi_1$ and a path $\pi_j':w_j\rightarrow v$ disjoint from $\pi_2$. This is possible by the same argument guaranteeing the existence of $\pi_1$ and $\pi_2$. We then  let $a$ be the first intersection of $\pi_j'$ with $\pi_1$, and let $b$ be the first intersection of $\pi_j$ with $\pi_2$ and $c$ be the last common vertex between $\pi_j$ and $\pi_j'$ appearing before  $a$ and $b$ on either path. (Note that some of these points may coincide.)
\end{proof}
\end{lemma}

\begin{definition}
Fix $v,w_1,\ldots, w_k\in \mathbb{Z}^d$ and let $S$ be an undirected graph with vertex set $\{0,\ldots, r\}$, with $r\ge k$. We assume the leaves (i.e. vertices of degree 1) are labeled by $\{0,\ldots, k\}$. 

We let $I(S)=\{k+1,\ldots, r\}$ denote the non-leaf (internal) vertices.
A map 
\[\psi: V(S)=\{0,\ldots,r\}\rightarrow \mathbb{Z}^d\]
is \emph{admissible} if $\psi(0)=v$ and $\psi(i)=w_i$ for $1\le i \le k$. Note that we do not require $\psi$ to be injective.

We define the \emph{valuation} of $S$ by
\begin{equation}
\label{eqn: val-def}
\mathrm{val}(S)=\sum_{\psi \text{ admissible }}\prod_{e=\{x,y\}\in E(S)}\tau\big(\psi(x),\psi(y)\big).
\end{equation}
\end{definition}
Recall that $\tau=\tau_2$ denotes the 2-point function. Note that the sum over $\psi$ in \eqref{eqn: val-def} is equivalent to a sum over 
\[(z_x)_{x\in I(S)}\in (\mathbb{Z}^d)^{\# I(S)}.\]
That is, $\mathrm{val}(S)$ is a diagram obtained by associating a two-point function factor to each edge of $S$, taking the product over the edges and then summing over possible maps of the vertices in $I(S)$. By the BK inequality \eqref{eqn: BK} and the two-point function bound \eqref{eqn: HHS}, $\mathrm{val}(S)$ bounds the probability of any event whose occurrence requires disjoint open connections along the edges of $S$.

\begin{proposition}\label{prop: contract}
Let $n\ge 1$. Suppose $|w_1-w_2|\ge \epsilon n$ and $|w_1|, |w_2|\le \epsilon^{-1}n$. Then, for $p\in \mathbb{Z}^d$,
\begin{equation}\label{eqn: G-prod}
\sum_{v\in \mathbb{Z}^d}\langle p-v\rangle^{-d+2}\langle v-w_1\rangle^{-d+2}\langle v-w_2\rangle^{-d+2}\le C(\epsilon, d)n^{4-d}\times \sum_{i=1}^2\langle p-w_i\rangle^{-d+2}.
\end{equation}
\end{proposition}

Proposition \ref{prop: contract} can be interpreted in terms of a contraction operation: let $S$ be a tree diagram with an internal vertex $v$ having a parent $p=p(v)$ and leaf children $w_1, w_2$. The \emph{contracted} tree $S^{(i)}$, $1\le i\le 2$ is obtained by deleting $v$ and attaching $p(v)$ directly to $w_i$ (and deleting the other leaf-edges from $v$). Then \eqref{eqn: G-prod} can be written as
\[\mathrm{val}(S)\le C(\epsilon,d)n^{4-d} \sum_{i=1}^2 \mathrm{val}(S^{(i)}).\]
That is, summing over $v$ replaces three edges by a single edge, introducing a factor $n^{2-d}$ for each deleted edge, and $n^d$ for the summation.
\begin{proof}[Proof of Proposition \ref{prop: contract}]
Applying \eqref{eqn: triple-sum}, we find, up to a constant factor, the following bound for the left-hand side of \eqref{eqn: G-prod}:
\begin{align}
&\langle p-w_1\rangle^{-d+2}\langle w_1-w_2\rangle^{-d+4}\label{eqn: cherry-first}\\
+~& \langle p-w_2\rangle^{-d+2}\langle w_1-w_2\rangle^{-d+4} \label{eqn: cherry-2nd}\\
+~& \langle p-w_1\rangle^{-d+2}\langle p-w_2\rangle^{-d+2}\min \{\langle p-w_1\rangle^2, \langle p-w_2\rangle^2\} \label{eqn: cherry-last}.
\end{align}
The desired estimate follows from  $|w_1-w_2|\ge \epsilon n$ in cases \eqref{eqn: cherry-first} and \eqref{eqn: cherry-2nd}.
For \eqref{eqn: cherry-last}, we use the separation condition $|w_1-w_2|\ge \epsilon n$ to find
\[R:=\max \{|p-w_1|,|p-w_2|\}\ge \frac{\epsilon}{2} n.\]
Denoting the minimum by $r$, we have
\begin{align*}
\eqref{eqn: cherry-last}&\le Cr^{4-d}R^{2-d}\\
&\le C(\epsilon)n^{-d+4}(\langle p-w_1\rangle^{-d+2}+\langle p-w_2\rangle^{-d+2}).
\end{align*}
\end{proof}

\begin{corollary}[Complete Tree Reduction]\label{cor: tree-reduct}
Let $S$ be an undirected, binary branching tree with $\ell+1$ leaves: all non-leaf vertices in $S$ have degree three. Suppose the admissible $\psi$ map the leaves to $v,w_1,\ldots,w_\ell\in \mathbb{Z}^d$, where $|w_i-w_j|\ge \epsilon n$ and $|w_i|\le \epsilon^{-1} n$.

Then
\begin{equation}
\mathrm{val}(S)\le C(\epsilon,\ell) n^{(4-d)(\ell-1)}\sum_{i=1}^\ell \langle v-w_i\rangle^{-d+2}.
\end{equation}
In particular, if 
\[|v-w_i|\ge \epsilon n, \quad i=1,\ldots,\ell\]
then
\begin{equation}\label{eqn:apriori}
\mathrm{val}(S)\le Cn^{(4-d)\ell-2}.
\end{equation}
\begin{proof}
    This is a straightforward induction using Proposition \ref{prop: contract} to contract all internal vertices.
\end{proof}
\end{corollary}

\begin{proposition}\label{prop: interior-delta}
    Let $n\ge 1$. Suppose $w_1\in \mathbb{Z}^d$. Then, for $u,v\in \mathbb{Z}^d$,
\begin{equation}\label{eqn: interior-delta}
\begin{split}
&\sum_{w\in \mathbb{Z}^d}\langle w-u\rangle^{-d+2}\langle w-v\rangle^{-d+2}\langle w-w_1\rangle^{-d+2}\\
\le &~Cn^2(n^{-d+2}+\langle u-w_1\rangle^{-d+2}+\langle v-w_1\rangle^{-d+2})\langle u -v \rangle^{-d+2}.
\end{split}
\end{equation}
\end{proposition}
Proposition \ref{prop: interior-delta} has diagrammatic interpretation as contraction along a path $u - w - v$. The right side of  \eqref{eqn: interior-delta} replaces the two edges $(u,w)$ and $(w,v)$ with a single edge $(u,v)$, deleting the vertex $w$.

\begin{proof}[Proof of Proposition \ref{prop: interior-delta}]
Applying \eqref{eqn: triple-sum}, we estimate the left-hand side of \eqref{eqn: interior-delta} by the upper bound
\begin{align}
&\langle u-v\rangle^{-d+2}\langle u-w_1\rangle^{-d+4}\label{eqn: 3-first}\\
+~& \langle u-v\rangle^{-d+2}\langle v-w_1\rangle^{-d+4} \label{eqn: 3-2nd}\\
+~& \langle u-w_1\rangle^{-d+2}\langle v-w_1\rangle^{-d+2} \min\{\langle u-w_1\rangle^2, \langle v-w_1\rangle^2\} \label{eqn: 3-last}
\end{align}
If $|u-w_1|\le n$, then the first term \eqref{eqn: 3-first} is bounded by
\[Cn^2\langle u-v \rangle^{-d+2}\langle u-w_1\rangle^{-d+2}.\]
If instead $|u-w_1|>n$, we have
\[\eqref{eqn: 3-first}\le Cn^{-d+4}\langle u-v\rangle^{-d+2},\]
provided $d>4$. An identical argument applies to \eqref{eqn: 3-2nd}, yielding
\[\eqref{eqn: 3-2nd}\le n^2(n^{-d+2}+\langle v-w_1\rangle^{-d+2})\langle u-v\rangle^{-d+2}.\]
For \eqref{eqn: 3-last}, we split into cases according to whether $|u-w_1|> \frac{1}{2}|u-v|$ or $|u-w_1|\le \frac{1}{2}|u-v|$. The latter case implies $|v-w_1|\ge\frac{1}{2}|u-v|$, so that
\[\eqref{eqn: 3-last}\le C\langle u-v\rangle^{-d+2}\langle u-w_1\rangle^{-d+4}.\]
In the former case, we have
\[\eqref{eqn: 3-last}\le C\eqref{eqn: 3-first}.\]
\end{proof}
In a graph consisting of a path decorated by binary branching subtrees, we can combine tree reduction with the above to sum over all unconstrained vertices along a path.

\begin{corollary}[Path reduction]\label{cor: path-red}
 Suppose $|w_1|, \, |w_2|, \, |w_3| \le \epsilon^{-1} n$ and $|w_i-w_j|\ge\epsilon n$ for $i\neq j$. Then, for $u,v\in \mathbb{Z}^d$, we have
\begin{equation}\label{eqn: path-reduction-eq}
\begin{split}
&\langle u-w_2\rangle^{-d+2}\langle v-w_3\rangle^{-d+2} \sum_{w\in \mathbb{Z}^d}\langle w-u\rangle^{-d+2}\langle w-v\rangle^{-d+2}\langle w-w_1\rangle^{-d+2}\\
\le &~Cn^{4-d}(\langle u-w_1\rangle^{-d+2}+\langle u-w_2\rangle^{-d+2})(\langle v-w_3\rangle^{-d+2}+\langle v-w_1\rangle^{-d+2}) \langle u -v \rangle^{-d+2}.
\end{split}
\end{equation}
\end{corollary}
This corollary says that we can sum over the interior vertex of a path in a binary tree (after having performed tree reduction on the dangling trees) and obtain a diagram of the same form with the vertex removed but with an added factor $n^{4-d}$.
\begin{proof}[Proof of Corollary \ref{cor: path-red}]
    This follows from Proposition \ref{prop: interior-delta} and the separation condition on the $w_i$, which ensures
    \[\max\{|u-w_1|,|u-w_2|\}\ge \epsilon n\] and 
    \[\max\{|v-w_1|,|v-w_3|\}\ge \epsilon n,\]
    since these imply
\[\langle u-w_2\rangle^{-d+2}\langle u-w_1\rangle^{-d+2}\le C(\epsilon) n^{-d+2}(\langle u-w_2\rangle^{-d+2}+\langle u-w_1\rangle^{-d+2}).\]
Similarly,
\[\langle v-w_3\rangle^{-d+2}\langle v-w_1\rangle^{-d+2}\le C(\epsilon)n^{-d+2}(\langle v-w_3\rangle^{-d+2}+\langle v-w_1\rangle^{-d+2}).\]
The claimed result follows at once.
\end{proof}

\begin{proposition}
    \label{cor: cycle-negl}
    Suppose $S$ is an undirected graph with $k+1$ leaves. We assume $S$ contains a cycle
\[\gamma= (u_0=u,u_1,\ldots,u_\ell,u_{\ell+1}=u_0)\]
with $\ell \ge 3$. To each $u_i$, $i=0,\ldots,\ell$ is attached a binary branching subtree $T_i$ rooted at $u_i$, such that removing the edges of $\gamma$ leaves a disjoint collection of subtrees. Define $\mathrm{val}(S)$ by \eqref{eqn: val-def}, with the admissible $\psi$ mapping the leaves to $\{w_0,\ldots, w_{k}\}$; assume that $\min_{i\neq j}\{|w_i - w_j| \} \geq \epsilon n$ and $\max\{ |w_i|\} \leq \epsilon^{-1} n$. Then, for $d>6$
    \begin{equation}\label{eqn: cycle-val}
    \mathrm{val}(S)\le Cn^{(4-d)k-2}\times \begin{cases}
    n^{-1}& \quad d=7\\
    (\log n) n^{-2}& \quad d =8\\
    n^{-2}&\quad d>8
    \end{cases}.
    \end{equation}
\end{proposition}
    \begin{proof}

We denote by $\mathcal{L}_i$ the leaves of $S$ contained in $S_i$ and set $m_i=\#\mathcal{L}_i$, so that
\begin{equation}\label{eqn: mi-sum2}
\sum_{i=0}^{\ell}m_i=k+1.
\end{equation}
We define the valuation of $S$ and the subtrees $T_i$ by the formula \eqref{eqn: val-def}. We have, with the assignment $z_i=\psi(u_i)$, and using the notation $\mathrm{val}(T_i)(z_i)$ to denote the dependence of this valuation on the location of the vertex $z_i$,
\[\mathrm{val}(S)\le \sum_{z_0,z_1,\ldots, z_\ell} \prod_{i=1}^{\ell+1}\mathrm{val}(T_i)(z_i)\langle z_i-z_{i-1}\rangle^{-d+2}.\]
Applying Corollary \ref{cor: tree-reduct} in each subtree $T_i$, $i=0,\ldots,\ell$, we find the bound
\[Cn^{(4-d)\sum_{i=0}^{\ell}( m_i-1)}
\sum_{j_0\in \mathcal{L}_0,j_1\in \mathcal{L}_1,\ldots, j_\ell\in \mathcal{L}_\ell}\sum_{z_0,z_1,\ldots, z_\ell}\langle z_0-z_\ell\rangle^{-d+2}\langle z_0-x_{j_0}\rangle^{-d+2}\prod_{i=1}^{\ell}\langle z_i-z_{i-1}\rangle^{-d+2}\langle z_i-w_{j_i}\rangle^{-d+2}.\]
We then use Corollary \ref{cor: path-red} to sum over $z_{4},\ldots,z_{\ell}$, leaving only 4 vertices in the cycle:
\begin{equation}\label{eqn: 4-cycle}
Cn^{(4-d)\sum_{i=0}^{\ell}( m_i-1)}n^{(4-d)(\ell-3)}
\sum_{\substack{j_0,j_1,j_2,j_3\\\text{ distinct}}}\sum_{z_0,z_1,z_2,z_3}\prod_{i=0}^{3}\langle z_i-z_{i-1}\rangle^{-d+2}\langle z_i-w_{j_i}\rangle^{-d+2},
\end{equation}
where we set $z_{-1}:=z_3$. The prefactor is
\begin{equation}\label{eqn: pre-factor0}
n^{(4-d)\sum_{i=0}^{\ell}( m_i-1)}n^{(4-d)(\ell-3)}=n^{(4-d)(k-3)}.
\end{equation}
We now sum over $z_1$ in \eqref{eqn: 4-cycle}, using \eqref{eqn: triple-sum}:
\begin{align*}
&\sum_{z_1}\langle z_1-z_0\rangle^{-d+2}\langle z_2-z_1\rangle^{-d+2}\langle w_{j_1}-z_1\rangle^{-d+2}\\
\le~&C\langle z_0-z_2\rangle^{-d+4}(\langle z_0-w_{j_1}\rangle^{-d+2}+\langle z_2-w_{j_1}\rangle^{-d+2})\\
+~&C\langle z_0-w_{j_1}\rangle^{-d+2}\langle z_2-w_{j_1}\rangle^{-d+2}\min\{\langle z_0-w_{j_1}\rangle^2,\langle z_2-w_{j_1}\rangle^2\}.
\end{align*}
When $|z_0-w_{j_1}|\ge \frac{1}{2}|z_0-z_2|$,
then the last line is
\[\le C\langle z_0-z_2\rangle^{-d+4}\langle z_2-w_{j_1}\rangle^{-d+2},\]
while if $|z_0-w_{j_1}|< \frac{1}{2}|z_0-z_2|$, we have the bound
\[\le C\langle z_0-z_2\rangle^{-d+4}\langle z_0-w_{j_1}\rangle^{-d+2}.\]
Put together, we obtain
\begin{equation}\label{eqn: z_1-sum}
\begin{split}
&\sum_{z_1}\langle z_1-z_0\rangle^{-d+2}\langle z_2-z_1\rangle^{-d+2}\langle w_{j_1}-z_1\rangle^{-d+2}\\
\le~&C\langle z_0-z_2\rangle^{-d+4}(\langle z_0-w_{j_1}\rangle^{-d+2}+\langle z_2-w_{j_1}\rangle^{-d+2})
\end{split}
\end{equation}

Similarly, we sum over $z_3$ to find
\begin{equation}\label{eqn: z_3-sum}
\begin{split}
&\sum_{z_3}\langle z_3-z_0\rangle^{-d+2}\langle z_2-z_3\rangle^{-d+2}\langle w_{j_3}-z_3\rangle^{-d+2}\\
\le~&C\langle z_0-z_2\rangle^{-d+4}(\langle z_0-w_{j_3}\rangle^{-d+2}+\langle z_2-w_{j_3}\rangle^{-d+2}).
\end{split}
\end{equation}
Combining \eqref{eqn: z_1-sum}, \eqref{eqn: z_3-sum} and the separation $|w_i-w_j|\ge \epsilon n$, we find an estimate
\begin{align}
&\sum_{z_1,z_3}\prod_{i=0}^{3}\langle z_i-z_{i-1}\rangle^{-d+2}\langle z_i-w_{j_i}\rangle^{-d+2}\nonumber\\
\le~& Cn^{-2d+4}\sum_{\substack{A,B\subset\{0,\ldots, k\}\\A\cap B=\emptyset}}\sum_{i\in A, j\in B} \langle z_0-w_{i}\rangle^{-d+2}\langle z_0-z_2\rangle^{-2d+8}\langle z_2-w_{j}\rangle^{-d+2}\label{eqn: cycle-pre0}
\end{align}
Lemma~\ref{prop: one-loop} evaluates the sum over $z_0$ and $z_2$ in \eqref{eqn: cycle-pre0}. Combined with the factor in \eqref{eqn: pre-factor0}, we obtain that \eqref{eqn: 4-cycle}, and thus $\mathrm{val}(S)$, is bounded by
    \[\mathrm{val}(S)\le Cn^{(4-d)(k-3)}n^{-2d+4}n^{-d+4}=Cn^{(4-d)k-4}\]
    if $d>8$,
    \[\mathrm{val}(S)\le C(\log n)n^{(4-d)(k-3)}n^{-2d+4}n^{-d+4}=C(\log n)n^{(4-d)k-4}\]
    if $d=8$
    and
    \[\mathrm{val}(S)\le Cn^{(4-d)(k-3)}n^{-2d+4}n^{-2d+12}=Cn^{(4-d)k+4-d}\]
    if $6<d< 8$.
    \end{proof}

\begin{lemma}[One loop diagram]\label{prop: one-loop}
Let $|x_1-x_2|\ge \epsilon n$, and define
\[D_{\mathrm{loop}}:=\sum_{z_0,z_2}\langle w_1-z_0\rangle^{-d+2}\langle z_0-z_2\rangle^{-2d+8}\langle z_2-w_2\rangle^{-d+2}.\]
Then, we have
\[D_{\mathrm{loop}}\le C\begin{cases}
    n^{-d+4} & \quad d>8,\\
    n^{-4}\log n& \quad d=8,\\
    n^{-2d+12}& \quad 6<d<8
\end{cases}.\]
\begin{proof}
The case $6<d<8$ is a direct application of \eqref{eqn: std-conv} (twice). The case $d>8$ follows from 
\[\sum_{z_2} \langle z_0-z_2\rangle^{-d-(d-8)}\langle z_2-w_2\rangle^{-d+2}\le C\langle z_0-w_2\rangle^{-d+2}\]
followed by an application of \eqref{eqn: std-conv}. For $d=8$, we first use \eqref{eqn:log-conv}:
\[\sum_{z_2}\langle z_0-z_2\rangle^{-d}\langle z_2-w_2\rangle^{-d+2}\le C\frac{\log\langle w_2-z_0\rangle}{\langle w_2-z_0\rangle^{d-2}}.\]
The result then follows by summing over $z_0$.
\end{proof}
\end{lemma}

\subsubsection{Proof of Lemma \ref{lem:binbran}}
\begin{proof}
If there is an $i\in \{1,\ldots, k-1\}$ such that $x_i$ is an interior vertex of $T(\omega)$, then there is a decomposition  $A\cup B$ of $\{x_0,\ldots, x_{k-1}\}$ such that $A\cap B=\{x_i\}$ with $|A|,|B|<k$ and $x_j \prec x_i$ for $j\in A$ , and for all $j\in B$, we have $m_{ij}=x_0$, in particular $\{x_j\lra x_0\}\circ\{x_i \lra x_0\}$. This implies the occurrence of the event
\[\Gamma(A)\circ \Gamma(B).\]
From this, by applying the BK inequality, we obtain the bound
\begin{align*}
\mathbb{P}(\Gamma(x_0,\ldots,x_{k-1}), \exists \,1\le i\le k-1: x_i \text{ not a leaf in } T)\le&~ \sum_{A\cup B=\{x_0,\ldots,x_{k-1}\}}\mathbb{P}(\Gamma(A)\circ \Gamma(B))\\
\le &~\sum_{A\cup B=\{x_0,\ldots,x_{k-1}\}}\mathbb{P}(\Gamma(A))\mathbb{P}(\Gamma(B))\\
\le&~Cn^{(4-d)(\#A-1)-2+(4-d)(\#B-1)-2}\\
\le&~Cn^{(4-d)(k-1)-4}.
\end{align*}
In the last step we used the Aizenman-Newman tree diagram bound \cite[(4.3)]{AN} (or, equivalently, part (1) of Lemma \ref{lem: Diagram}) to estimate the probability that the vertices in $A$, resp. $B$ are all connected:
\[\tau_\ell(x_0,\ldots,x_{\ell-1})\le{\sum_{G}}' \sum_{y_1,\ldots,y_{\ell-2}}\prod_{\{z,z'\}\in E(G)}\tau(z,z'),\]
where the outer sum is over binary branching graphs $G$ with leaves $x_0,\ldots,x_{\ell-1}$ and internal vertices $y_1,\ldots,y_{\ell-2}$. The sums are estimated by \eqref{eqn:apriori} in Corollary \ref{cor: tree-reduct}.

If $x_0$ is a non-leaf vertex of $T$, then there are two disjoint subtrees of the connectivity tree rooted at $x_0$, and the same argument as above applies to show
\[\mathbb{P}(x_0 \text{ not a leaf })\le Cn^{(4-d)(k-1)-4}.\]

We may thus assume that the $k$ leaves of $T$ are $\{x_0,\ldots,x_{k-1}\}$. The result now follows from Proposition \ref{cor: cycle-negl}, since
\[\mathbb{P}(E_{\mathrm{deg}}, \Gamma(x_0,\ldots,x_{k-1}), \text{leaves}(T)=\{x_0,\ldots,x_{k-1}\})\le \sum_{S\in \mathcal{S}}\mathrm{val}(S),\]
where $\mathcal{S}$ is the finite set of graphs with $k$ leaves $\{0,\ldots,k-1\}$ containing a cycle as in that proposition.
\end{proof}

\subsection{The Inductive Step \label{sec:inductive}}
We decompose the $k$-point function $\prob(\Gamma(x_0, \dots, x_{k-1}))$ and its decorated counterpart into a sum over trees. The induction will essentially show that Theorem~\ref{thm:kpt} holds for a given value of $k$, assuming that it holds for all smaller values of $k$. In the process of the induction, we remove a number of error terms, similar to our argument for the case $k = 3$. The remaining portion of $\prob(\Gamma(x_0, \dots, x_{k-1}))$ after removing error terms will be shown to converge to the quantity appearing in Theorem~\ref{thm:kpt}.  The main term will ultimately be shown to be asymptotically bounded below by $\min\{|x_i - x_j|: \, i \neq j\}^{(4-d)k - 2}$ for $x_i\in G(\epsilon,n)$, as indicated in Theorem~\ref{thm:kpt}. The error terms will be shown to be of smaller order, and hence not survive in the limit.

If $T \in \mathfrak{T}_{k}$, we define
\begin{equation}\label{eqn: WTl-def}
\W^{(\lambda)}_T:=\exp\big(-an^{-2}\sum_{e\in E(T)}\lambda_e D_e\big).
\end{equation}

We write
\[\tau_T^{(\lambda)}(x_0, \dots, x_{k-1}) = \mathbb{E}[\W^{(\lambda)}_T,\Gamma(x_0, \dots, x_{k-1}), \mathcal{T}(\omega) = T]\ .\]
Lemma~\ref{lem:binbran} shows that for each fixed $\epsilon > 0$, if $\min_{j\neq i}\{|x_i - x_j|\} \geq \epsilon n$, we have
\[\tau_{k}^{(\lambda)}(x_0, \dots, x_{k-1}) = \sum_{T \in \mathfrak{T}_{k}} \tau_T^{(\lambda)}(x_0, \dots, x_{k-1}) +  O\big(n^{(4-d)(k-1) -2}r_d(n)\big), \]
where
\[r_d(n)=\begin{cases}
n^{-1}, &\quad d=7\\
n^{-2}\log n,& \quad d=8\\
n^{-2},&\quad d>8.
\end{cases}\]

The following is the key inductive portion of the proof of Theorem~\ref{thm: distances}.
\begin{theorem}\label{thm: decorated-conv}
For every $k$, $T\in\mathfrak{T}_k$, and $\boldsymbol{\lambda} = (\lambda_e)_{e\in E(T)}\in[0,\infty)^{E(T)}$,
\[
n^{(d-4)(k-1)+2}\,
\tau_T^{(\boldsymbol{\lambda})}(x_0^{(n)},\ldots,x_{k-1}^{(n)})
\longrightarrow \alpha^{2k-3}(2d\beta^2\rho)^{k-2}\, \mathcal{I}_T^{(\boldsymbol{\lambda})}(y_0,\dots,y_{k-1}),
\]
where 
\[\mathcal{I}_T^{(\boldsymbol{\lambda})}=\int_{(\mathbb{R}^d)^{k-2}}\prod_{\{r,s\}\in E(T)}|u_r-u_s|^{2-d}\psi_d(\lambda_{rs}|u_r-u_s|^2)\prod_{v\in \mathrm{int}(T)}\,\mathrm{d}u_v. \]
\end{theorem}

This result will imply the result on joint convergence of the distances by the exact identity
\[D(x_0,x_i)=\sum_{e\in\mathrm{path}(0,i)}D_e+\#\{\text{ interior vertices on the path }\}\]
which holds for $i=1,\ldots,k-1$.

\begin{proposition}\label{prop: tree-reduct} For each $k \geq 3$,
\begin{equation}
\label{eq:inedge}
\begin{gathered}
\text{each $T \in \mathfrak{T}_k$ has a non-leaf vertex $v$ whose two in-edges originate from leaves labeled $I, J \neq 0$.}
\end{gathered}
\end{equation}
\end{proposition}
\begin{proof}
Letting $I$ be the label of a leaf at furthest graph distance in $T$ from the vertex labeled $0$, and letting $v$ be the vertex adjacent to the vertex labeled $I$, we see this $v$ satisfies \eqref{eq:inedge}, with $J$ the other leaf adjacent to $v$.
\end{proof} 

By the previous proposition, there is a lexicographically largest pair $(I,J)$ such that the vertex $I$ appears as a leaf as in \eqref{eq:inedge}. We denote it by $I(T)$ once and for all and write $J(T)$ and $V(T)$ for the corresponding vertices $J, v$ as above.
For $k \geq 4$ and $T \in \mathfrak{T}_k$, we set $\phi(T)$ to be the tree in $\mathfrak{T}_{k-1}$ obtained by 
\begin{itemize}
\item Deleting the vertices labeled $I$ and $J$, along with their incident edges to $v$,
\item Retaining their parent $v$, which now becomes a leaf,
\item Decrementing the labels of already labeled vertices to fill gaps while preserving relative order, and then assigning the label $k-2$ to $v$.
\end{itemize}

We write $\phi(\ell)$ for the label assigned in $\phi(T)$ to the vertex labeled by $\ell$ in $T$; we treat $\phi(I)$ and $\phi(J)$ as empty, so that the leaves $(\{x_i\}_{i\neq I,J},v)$ form a vector with $k-1$ entries. For a sequence $\boldsymbol{\lambda} = (\lambda_e)_{e\in E(T)}$, we denote its restriction to the edges of $\phi(T)$ by $\phi(\boldsymbol{\lambda})$.

\begin{theorem} \label{thm:treeinduct}
Fix $k \geq 3$, $T \in \mathfrak{T}_{k}$, and $\epsilon > 0$. 
Suppose that there exists a $c > 0$ such that, uniformly in $\widetilde{\boldsymbol{\lambda}}=(\widetilde{\lambda}_e)_{e\in E(\phi(T))}$ with $|\tilde{\boldsymbol{\lambda}}|=\sum_e |\tilde{\lambda}_e|<\Lambda$,
\begin{equation}\label{eqn: inductive-num-lwr-bd}
\liminf_{n \to \infty} \inf_{\big((x_i)_{i\neq I,J},v\big) \in G(\epsilon, n)} \frac{\tau^{(\tilde{\mathbf{\lambda}})}_{\phi(T)}(x_0, \dots, x_{k-2})}{n^{(4-d)(k-2) - 2}} \geq c_{\epsilon,\Lambda,T}>0\ . 
\end{equation}

Then we have,  uniformly in $\boldsymbol{\lambda}$ with $|\boldsymbol{\lambda}| < \Lambda,$
\begin{equation*}
\begin{gathered}
    \sup_{(x_0, \dots, x_{k-1}) \in G(\epsilon, n)} \Big|\frac{\tau_T^{(\boldsymbol{\lambda})}(x_0, \dots, x_{k-1})}{  2d \beta^2 \alpha^2 \rho \sum_{v \in \Z^d} \tau_{\phi(T)}^{(\phi(\boldsymbol{\lambda}))}((x_i)_{i\neq I,J}, v)\psi_d(\lambda_{e_J}\frac{|x_J-v|^2}{n^2})\psi_d(\lambda_{e_I}\frac{|x_I-v|^2}{n^2}) \langle x_{J} - v\rangle^{2-d} \langle x_{I} - v\rangle^{2-d}}- 1 \Big|\\
    \to \quad 0 \qquad \text{as $n \to \infty$,} \ 
    \end{gathered}
\end{equation*}
where $I = I(T)$ and $J = J(T)$.
\end{theorem}

We begin with the following a priori bound on the $k$-point function: 
\begin{lemma} \label{lem:apriori}
    For each $\epsilon > 0$, there is a $C_k=C_k(\epsilon)$ such that if  $(x_1, \dots, x_{k-1}) \in G(\epsilon, n)$ then, uniformly in $\lambda$, 
\[\tau_{k}^{(\boldsymbol{\lambda})}(x_0, x_1, \dots, x_{k-1}) \leq C_k(\epsilon) \min_{1\le i\le k-1}  |x_0-x_i|^{4-d} n^{(4-d) (k-2) - 2}. \]
Dependence of $C(\epsilon)$ on $\epsilon$ is polynomial:
\begin{equation}\label{eqn:poly-C}
C_k(\epsilon)\le c\epsilon^{-10dk},
\end{equation}
where $c$ does not depend on $\epsilon$.
\end{lemma}
\begin{proof}
     Since $0\le \W_T^{(\boldsymbol{\lambda})}\le 1$, it suffices to estimate the undecorated probabilities with $\boldsymbol{\lambda} \equiv 0$. We prove it by induction on $k$; the case $k = 2$ follows from usual two-point asymptotic \eqref{eqn: HHS}. In this case, we can take  $C(\epsilon)\le c\epsilon^{-d+2}$ since $|x_0-x_1|\ge \epsilon n$ by assumption.
    
    Assuming the lemma holds for $k-1$, we show it for $k$. We assume we have shown the second claim of the lemma holds for $\tau_{\ell}$ for $\ell < k$, and  
 we show it for $\tau_{k}$. We write using a weaker form of the tree-graph decomposition
\begin{align}
\label{eq:weaktree}
\tau_{k}(x_0, \dots, x_{k-1}) \leq \sum_{T \in \mathfrak{T}_k} \sum_{w_1, \dots, w_{k-1}} \prob( u_a \lra u_b \quad \forall \{a, b\} \in E(T))\ ,
\end{align}
where $w_1, \dots, w_{k-1}$ are the $\mathbb{Z}^d$-representatives of interior vertices of the tree, and $u_a$, $u_b$ are lattice sites labeled by the tree vertices $a$ and $b$. We fix a value of $T$ in the outer sum and show the claimed bound holds uniformly for the corresponding term of the inner sum.

Since $x_k$ is a leaf of $T$, it appears in exactly one edge $\{z, x_k\}$ of $T$. Removing $x_k$ from $T$, the vertex $z$ now has degree two. Splitting into the two connected sets on either side of $z$ produces a graph with two disjoint components $G_1$ and $G_2$ each containing a copy of $z$ as a leaf. Letting $A_i = \{x_j: j \neq k, \, x_j \in G_i\}$, we have $\max\{|A_1|, |A_2|\} \leq k-1$ and $|A_1| + |A_2|= k$. By relabeling if necessary, we assume that $x_1$ is the nearest element of $\{x_i\}_{i=1}^{k-1}$ to $x_k$ and that $x_1 \in A_1$. We write $E_i$ for the edges of $T$ internal to $A_i$ for $i = 1, 2$.

We use the BK inequality and the inductive hypothesis, as well as the first part of the theorem to upper bound the typical term of \eqref{eq:weaktree} by
\begin{align*}
&\sum_{w_1, \dots, w_{k-1}}  \tau(x_k, z) \tau(G_1) \tau(G_2)\\
\leq~& C\epsilon^{-d+2}C_{|A_1|}(\epsilon)C_{|A_2|}(\epsilon)n^{(4-d)(k-1) - 4} \sum_{z} \left[ \mathrm{dist}(z, A_1)^{4-d} + \mathrm{dist}(z, A_2)^{4-d} \right] \langle z-x_k\rangle^{2-d}\ ,
\end{align*}
where $C$ does not depend on $\epsilon$.
By our assumption about $x_1$, the dominant term is the one with the factor $\mathrm{dist}(z, A_1)^{4-d}$, so we upper bound by
\[C\epsilon^{-10d(|A_1|+|A_2|)-d+2} n^{(4-d)(k-1) - 4} \sum_{z} \mathrm{dist}(z, A_1)^{4-d} \langle z-x_k\rangle^{2-d} . \]

Since $d > 6$ and $|A_1|+|A_2|\le k-1$, the above is further bounded by
\[C\epsilon^{-10d(k-1)-d+2} n^{(4-d)(k-1) - 4} \sum_{z} \langle z-x_1\rangle^{4-d} \langle z-x_k\rangle^{2-d} \leq C n^{(4-d)(k-1) -4} \mathrm{dist}(x_k, \{x_i\}_{i=1}^{k-1})^{6-d} . \]
Finally, the last expression is at most
\[  C\epsilon^{-10dk} n^{(4-d)(k-1) -2} \mathrm{dist}(x_k, \{x_i\}_{i=1}^{k-1})^{4-d}.\]
\end{proof}

\begin{lemma}\label{lem:lcpiv}
Consider a set $\{x_0, \dots, x_{k-1}\} \subseteq G(\epsilon,n)$ and set $A = \{x_1, \dots, x_{k-1}\}$. There is a $C = C(k, \epsilon) > 0$ such that
\[\prob(\Gamma(x_0, \dots, x_{k-1}), \, \upiv(x_0, A) = \varnothing) \leq C n^{(4-d)(k-1) - 4}\ .  \]

\end{lemma}
By the assumption \eqref{eqn: inductive-num-lwr-bd},  $\prob(\Gamma(x_0, \dots, x_{k-1}))$ is at least $c n^{(4-d) (k-1) - 2}$, so the above represents an error term.
 
\begin{proof}[Proof of Lemma~\ref{lem:lcpiv}]
This is a consequence of Lemma \ref{lem:binbran}, since $\upiv(x_0, A) = \varnothing$ implies the existence of (at least) two disjoint subtrees of the connectivity tree spanning $\{x_0,\ldots,x_{k-1}\}$ and rooted at $x_0$.
\end{proof}

\subsection{Proof of Theorem \ref{thm:treeinduct}}
        We induct on $k$. The case $k = 3$ involves (up to relabeling) only one tree $T$; we have treated this case above in Section~\ref{sec: 3pt-scaling}. We assume that the statement holds for all $T \in \mathfrak{T}_\ell$ for $\ell \leq k$ and prove it also holds in the case $\ell = k+1$. 
        
        We fix an arbitrary $T \in \mathfrak{T}_{k+1}$.
        As at Proposition \ref{prop: tree-reduct}, we let $I, J \neq 0$ refer to two indices of leaves in $T$ which are at graph distance $2$ from each other in $T$.
        Applying Lemma~\ref{lem:lcpiv} gives that it suffices to show the claimed asymptotic holds for
    \begin{equation}
        \label{eq:truncA}
        \mathbb{E}[\W_T^{(\lambda)},\, \Gamma(x_0, \dots, x_{k}), \, \mathcal{T}(x_0, \dots, x_k) = T, \, \piv(x_0, \{x_I, x_J\}) \neq \varnothing].
    \end{equation}
    We sum over the value $g$ of $\opiv(x_0, \{x_I, x_J\})$:
    \[\eqref{eq:truncA} = \sum_{g \in \edges} \mathbb{E}[\W_T^{(\lambda)},\, \Gamma(x_0, \dots, x_{k}), \, \mathcal{T}(x_0, \dots, x_k) = T, \, \opiv(x_0, \{x_I, x_J\}) = g]. \]
    In an outcome of the event from the last display, the connectivity tree $\mathcal{T}$ associated to $(x_{\phi(0)}, \dots, x_{\phi(k)}, \underline{g})$ is  $\phi(T)$. We apply Lemma~\ref{lem:newcut} to express \eqref{eq:truncA} as
    \begin{equation}
    \begin{split}
        \label{eq:truncA2}
        &\beta \sum_{g \in \edges} \mathbb{E}\left(\W^{(\lambda)}_{T}, \begin{array}{c}\Gamma(\{x_i\}_{i\neq I,J},\underline{g}), \mathcal{T}(\{x_i\}_{i\neq I,J},\underline{g}) = \phi(T), \\
        x_I, x_J \lra \overline{g} \text{ off }\,\clust(\underline{g}),
        \mathcal{P}(\overline{g},\{x_I, x_J\})=\varnothing
        \end{array}\right)\\
        =~&\beta \sum_{g \in \edges} \mathbb{E}\left(\W^{(\lambda)}_{T},\begin{array}{c}\Gamma(\{x_i\}_{i\neq I,J},\underline{g}), \mathcal{T}(\{x_i\}_{i\neq I,J},\underline{g}) = \phi(T), \\
        \{x_I \lra \overline{g} \text{ off }\,\clust (\underline{g})\} \circ \{x_J \lra \overline{g} \text{ off }\,\clust(\underline{g})\}
        \end{array}\right)\ .
        \end{split}
    \end{equation}

    The completion of the proof of Theorem~\ref{thm:treeinduct} proceeds by analysis of \eqref{eq:truncA2}. 
    We again introduce an auxiliary small parameter $\epsilon > 0$. 
    In Section~\ref{sec:knear}, we show that the contribution to the sum in \eqref{eq:truncA2} from $g$ near some $x_i$ --- that is, such that $(x_0, \dots, x_k, \underline{g}) \notin G(\epsilon^R,n)$ --- is negligible as $n \to \infty$. In Section~\ref{sec:kfar}, we control the remaining terms of \eqref{eq:truncA2} for $n$ large, and we combine the two estimates to establish Theorem~\ref{thm:treeinduct}.

\subsubsection{Near-regime \label{sec:knear}}
We bound the contribution to \eqref{eq:truncA2} from edges $g$ such that $(\underline{g}, x_{\phi(0)}, \dots, x_{\phi(k)}) \notin G(\epsilon^R,n)$ for $R\ge 1$ to be determined. Since $0\le \W^{(\lambda)}_T\le 1$, it suffices to control the following partial sum of \eqref{eq:truncA2}:
\begin{equation}\label{eq:truncA20}
     \sum_{\substack{g \in \edges: \\ |\underline{g} -x_i| < \epsilon^R n \text{ for some $i$}  }} \mathbb{P}\left(\begin{array}{c}\Gamma(x_{\phi(0)}, \dots, x_{\phi(k)}, \underline{g}), \mathcal{T}(x_{\phi(0)}, \dots, x_{\phi(k)}, \underline{g}) = \phi(T), \\
        \{x_I \lra \overline{g} \text{ off }\,\clust(\underline{g})\} \circ \{x_J \lra \overline{g} \text{ off }\,\clust(\underline{g})\}
        \end{array}\right),
\end{equation}
where $\tau_k:=\tau_k^{(0,\ldots,0)}$. This in turn is bounded by
\begin{align}
        \label{eq:truncA3}
 & \sum_{\substack{g \in \edges: \\ |\underline{g} -x_i| < \epsilon^R n \text{ for some $i$}  }} \tau_k(x_{\phi(0)}, \dots, x_{\phi(k)}, \underline{g}) \tau(x_I, \overline{g}) \tau(x_J, \overline{g})\\
 \le &~ \sum_{j=0}^k \sum_{\substack{g \in \edges: \\ |g- x_j| < \epsilon^R n}} \tau_k(x_{\phi(0)}, \dots, x_{\phi(k)}, \underline{g}) \tau(x_I, \overline{g}) \tau(x_J, \overline{g}). \label{eq:Anear}
\end{align}
For the term $j = I$ of \eqref{eq:Anear}, note that when $|g - x_I| \leq \epsilon^R n/2$  we have 
\[|g - x_i| > \epsilon n-\epsilon^Rn/2\ge \epsilon n/2\]
for all $i \neq I$.  We can therefore bound the factors of \eqref{eq:Anear} other than $\tau(x_I, \overline{g})$ using \eqref{eqn: HHS} and \eqref{eqn:apriori}, yielding
    \begin{align}
        & \sum_{\substack{g \in \edges: \\ \mathrm{dist}(g, x_I) < \epsilon^R n}} \tau_k(x_{\phi(0)}, \dots, x_{\phi(k)}, \underline{g}) \tau(x_I, \overline{g}) \tau(x_J, \overline{g}) \nonumber\\
        \leq~&C(\epsilon) n^{(4-d)(k-1) - 2 + 2 - d}\sum_{\substack{g \in \edges: \\ \mathrm{dist}(g, x_I) < \epsilon^R n}} \langle x_I - \overline{g}\rangle^{2-d}\nonumber\\
        \leq~&C(\epsilon) \epsilon^{2R} n^{(4-d) k - 2}\ .\label{eq:sumI}
    \end{align}
    where the constant $C(\epsilon)$ depends on $\epsilon$ but not on $R$.
    A bound identical to \eqref{eq:sumI} holds for the term $j = J$ of \eqref{eq:Anear}.

    We now bound the remaining terms of \eqref{eq:Anear}, using Lemma~\ref{lem:apriori}:
    \begin{align*}
        & \sum_{\substack{g \in \edges: \\ \mathrm{dist}(g, x_j) < \epsilon^R n}} \tau_k(x_{\phi(0)}, \dots, x_{\phi(k)}, \underline{g}) \tau(x_I, \overline{g}) \tau(x_J, \overline{g})\\
        &\leq~C(\epsilon) n^{4-2d} n^{(4-d)(k-2)} \sum_{\substack{g \in \edges: \\ \mathrm{dist}(g, x_j) < \epsilon^R n}} \langle \overline{g} - x_j\rangle^{4-d}\\
        &\leq C(\epsilon) \epsilon^{2R} n^{(4-d) k - 2}\ .
    \end{align*}
    Here we note that the constant in Lemma \ref{lem:apriori} has polynomial order dependence on $\epsilon^{-1}$ which does not depend on $R$, so choosing $R$ sufficiently large, the last quantity can be bounded by $C\epsilon$ for some $C$ independent of $\epsilon$.
    Pulling the last display together with \eqref{eq:sumI} (and the analogous $J$ term) gives that 
    \begin{equation}
        \label{eq:lastsub}
        \eqref{eq:truncA20} \leq C \epsilon n^{(4-d) k - 2}\ .
    \end{equation} 
    This is an error term compared to the scale in Theorem~\ref{thm:kpt}.

\subsubsection{Far-regime \label{sec:kfar}}
To complete the proof of Theorem~\ref{thm:treeinduct}, we analyze the terms of \eqref{eq:truncA2} corresponding to $(\overline g, x_0, \dots, x_k) \in G(\epsilon^R,n)$. The partial sum of these terms is:
\begin{equation}
    \label{eq:truncAfar}
     \beta \sum_{\substack{g \in \edges: \\ (x_0, \dots, x_k, \underline{g}) \in G(\epsilon^R,n)}} \mathbb{E}\left(\W^{(\phi(\boldsymbol{\lambda}))}_{\phi(T)}\W^{(\lambda_{e_I},\lambda_{e_J})},\begin{array}{c}\Gamma(x_{\phi(0)}, \dots, x_{\phi(k)}, \underline{g}), \mathcal{T}(x_{\phi(0)}, \dots, x_{\phi(k)}, \underline{g}) = \phi(T), \\
        \{x_I \lra \overline{g} \text{ off }\,\clust(\underline{g})\} \circ \{x_J \lra \overline{g} \text{ off }\,\clust(\underline{g})\}
        \end{array}\right),
\end{equation} 
where
\[\W^{(\lambda_{e_I},\lambda_{e_J})}=e^{-an^{-2} (\lambda_{e_I} D(\bar{g},x_I)+\lambda_{e_J} D(\bar{g},x_J))}.\]
The result \eqref{eq:lastsub} shows that
\begin{equation}
    \label{eq:truncAeps}
    \lim_{\epsilon \to 0} \limsup_{n \to \infty} \left| \frac{\eqref{eq:truncAfar}}{\eqref{eq:truncA2}} - 1 \right| = 0 \ .
\end{equation}

We recall the expression appearing in Theorem~\ref{thm:treeinduct}:
\begin{equation}\label{eq:truncAfar2}
    2d\beta^2 \alpha^2 \rho \sum_{v \in \Z^d} \tau^{(\phi(\boldsymbol{\lambda}))}_{\phi(T)}(x_{\phi(0)}, \dots, x_{\phi(k)}, v) \langle x_{J} - v\rangle^{2-d}\psi_d\left(\lambda_{e_J}|y_J-\frac{v}{n}|^2\right) \langle x_{I} - v\rangle^{2-d}\psi_d\left(\lambda_{e_I}|y_I-\frac{v}{n}|^2\right)\ .
\end{equation}
By \eqref{eq:truncAeps}, to complete the proof of the theorem, it suffices to show
\begin{equation}
    \label{eq:truncAeps2}
        \lim_{\epsilon \to 0} \limsup_{n \to \infty} \left| \frac{\eqref{eq:truncAfar}}{\eqref{eq:truncAfar2}} - 1\right| = 0\ .
\end{equation}

We turn to the sum \eqref{eq:truncAfar}; as in the case of the three-point function, the portions of the event depending on the clusters of $\underline{g}$ and $\overline{g}$ can be decoupled. For this, we introduce a new approximating event.
Let $\mathcal{D} \subseteq \Z^d$ be a fixed finite set. For arbitrary $x_0, x_1, \dots, x_{k-1}$ such that $(x_0, x_1, \dots, x_{k-1}) \in G(\epsilon,n)$ and $T \in \mathfrak{T}_{k+1}$, and for each $0 \leq i \leq k$, set
\begin{equation}
    \label{eq:GTdef}
    \Gamma_{T, i}(x_0, \dots, x_{k-1}; \mathcal{D}) = \bigcap_{j=1}^{k-1}\left\{x_0 \sa{\Z^d \setminus \mathcal{D}} x_j  \right\} \cap \{\mathcal{T}_{k+1} = T \}\ . 
\end{equation}

We also introduce a new auxiliary parameter $K$, playing a virtually identical role as it did in \eqref{eq:threept3k}. By considering a $K$-dependent approximation using the events \eqref{eq:GTdef}, we show
\begin{equation}
    \label{eq:truncAeps3}
        \lim_{\epsilon \to 0} \lim_{K \to \infty} \limsup_{n \to \infty} \left| \frac{\eqref{eq:truncAfar}}{\eqref{eq:truncAfar2}} - 1\right| = 0\ ,
\end{equation}
which establishes \eqref{eq:truncAeps2} and completes the proof of Theorem~\ref{thm:treeinduct}.

As in the proof of Proposition~\ref{prop:threept}, we introduce a new independent copy of our probability space and an associated independently distributed outcome $\wto{\omega}$, writing $\wto{\prob}$ for probabilities with respect to this independent percolation process.   Lemma~\ref{lem: GT-truncation} shows that for each fixed $\epsilon > 0$, uniformly in $n$ and in $(x_0, \dots, x_{k-1}) \in G(\epsilon,n)$, if we consider the expression
\begin{equation}
    \label{eq:trunchelp}
    \beta \sum_{\substack{g \in \edges: \\ (x_0, \dots, x_k, \underline{g}) \in G(\epsilon^R, n)}}  \widetilde{\mathbb{E}}\left[\W^{(\lambda_{e_I},\lambda_{e_J})}\mathbf{1}_{\{x_I \lra \overline{g} \} \circ \{x_J \lra \overline{g}\}} \mathbb{E}[\W^{(\phi(\boldsymbol{\lambda}))}_{\phi(T)},\Gamma_{\phi(T), k}(x_{\phi(0)}, \dots, x_{\phi(k)}, \underline{g}; \clust_{B(\overline{g}; 2K)}(\overline{g}))]\right],
\end{equation}
we have
\[\left|\eqref{eq:truncAfar} - \eqref{eq:trunchelp} \right| \leq C n^{(4-d)k - 2} K^{(6-d)/d}\ ,\]
and so, using \eqref{eqn: inductive-num-lwr-bd}:
\[\lim_{\epsilon \to 0}\lim_{K \to \infty}\limsup_{n \to \infty}\left|\frac{\eqref{eq:trunchelp}-\eqref{eq:truncAfar}}{\eqref{eq:truncAfar2}}\right| = 0\ . \]

We apply the enhanced IIC result, Lemma~\ref{lem:iicnew2}, to control \eqref{eq:trunchelp} as $n \to \infty$. Similar to the estimates below \eqref{eqn: apply-IIC}, if we define the expression
\begin{equation}
    \label{eq:iicmorehelp}
   \beta \sum_{\substack{g \in \edges: \\ (x_0, \dots, x_k, \underline{g}) \in G(\epsilon^R,n)}}  \tau^{(\phi(\boldsymbol{\lambda}))}_{\phi(T)}(x_{\phi(0)}, \dots, x_{\phi(k)}, \underline{g})  \wto{\mathbb{E}}\left[e^{-a n^{-2} (\lambda_{e_I} D(\bar{g},x_I)+\lambda_{e_J} D(\bar{g},x_J))}\mathbf{1}_{\{x_I \lra \overline{g} \} \circ \{x_J \lra \overline{g}\}} \nu_{\underline{g}}(\Xi_{\underline{g}}(\clust_{B(\overline{g}; 2K)}(\overline{g}))) \right]
\end{equation}
we have
\begin{equation}
    \label{eq:iicmorehelp2}
   \lim_{n \to \infty}  \sup_{(x_0, \dots, x_k) \in G(\epsilon,n)}\left|\frac{\eqref{eq:trunchelp}}{\eqref{eq:iicmorehelp}} - 1 \right| = 0\quad \text{for each fixed $K \geq 1$, $\epsilon > 0$.}
\end{equation}

Applying Lemma~\ref{lem:genlem1} and recalling that each vertex is the endpoint of $2d$ edges, we see
\begin{align*} \lim_{\epsilon \to 0} \lim_{K \to \infty} \lim_{n \to \infty} &\frac{\eqref{eq:iicmorehelp} - 2d\alpha^2 \beta^2 \rho \sum_{v \in \Z^d} \tau^{(\phi(\lambda))}_{\phi(T)}((x_i)_{ i\neq I,J}, v) \langle x_{J} - v\rangle^{2-d}\psi_d(\lambda_{e_J}|y_J-\frac{v}{n}|^2) \langle x_{I} - v\rangle^{2-d}\psi_d(\lambda_{e_I}|y_I-\frac{v}{n}|^2) }{2d\alpha^2 \beta^2 \rho \sum_{v \in \Z^d} \tau_{\phi(T)}((x_i)_{i\neq I,J}, v) \langle x_{J} - v\rangle^{2-d} \langle x_{I} - v\rangle^{2-d}}\\
&= 0\ .
\end{align*}
This completes the proof of Theorem~\ref{thm:treeinduct}.\qed

\subsection{Proof of Theorem~\ref{thm: distances}}
In this section, we prove Theorem~\ref{thm: distances} (and hence Theorem~\ref{thm:kpt}) using Theorem~\ref{thm:treeinduct}.

We prove the result inductively in $k$. The induction involves a decomposition over trees of $\mathfrak{T}_{k+1}$. 

The statement on which we perform the induction is 
\begin{equation}
    \label{eq:howtoind}
    \begin{gathered}
    \text{for all $\ell < k$, all $T \in \mathfrak{T}_{\ell},$ all $\boldsymbol{\lambda}$ with each $\lambda_e > 0$, and all distinct $y_0, \dots, y_{\ell-1} \in \mathbb{R}^d$,}\\
    \text{ setting $x_i^{(n)} = \lfloor n y_i \rfloor$ for $0 \leq i \leq \ell$,}\\
    n^{-((4-d)(\ell-1)-2)}\tau_{T}^{(\boldsymbol{\lambda})}(x_0^{(n)},\ldots, x_{\ell-1}^{(n)})\rightarrow \alpha^{2\ell-3}(2d\beta^2 \rho)^{\ell-2}\mathcal{I}_{T}^{(\boldsymbol{\lambda})}(y_0,\ldots, y_{\ell-1})\ .
    \end{gathered}
\end{equation}
By Lemma~\ref{lem:binbran}, with each $y_i$ and $x_i^{(n)}$ as above,  (make sure the decorations between the two notions of $\tau$ are compatible) we have
\[\lim_{n \to \infty} \frac{\tau_{\ell+1}^{(\boldsymbol{\lambda})}(x_0^{(n)}, \dots, x_\ell^{(n)})}{\sum_{T \in \mathfrak{T}_{\ell+1}} \tau_T^{(\boldsymbol{\lambda})}(x_0^{(n)}, \dots, x_\ell^{(n)}) } = 1\ , \]
and so establishing \eqref{eq:howtoind} for all $k\geq 2$ will complete the proof of Theorem~\ref{thm:kpt}.

Since there is only one tree $T \in \mathfrak{T}_{3}$, and since we proved the special case of Theorem~\ref{thm:kpt} for $k = 3$ as Proposition~\ref{prop:threept}, we just need to prove the inductive step. We assume the case $k-1$ of \eqref{eq:howtoind} holds and then establish it for the case $k$.

In turn, applying Theorem~\ref{thm:treeinduct}, instead of showing \eqref{eq:howtoind} in case $k$, we need only show
\begin{equation} \label{eq:peel} \begin{split}&2d\alpha^2 \beta^2 \rho n^{-(4-d)k+2} \sum_{v \in \Z^d} \frac{\tau^{(\phi(\boldsymbol{\lambda}))}_{\phi(T)}(x_{\phi(0)}, \dots, x_{\phi(k-1)}, v) \psi_d\left(\lambda_{e_J}\frac{|x_J-v|^2}{n^2}\right)\psi_d\left(\lambda_{e_I}\frac{|x_I-v|^2}{n^2}\right)}{ \langle x_{J} - v\rangle^{d-2} \langle x_{I} - v\rangle^{d-2}}\\
\to~&\alpha^{2k-3}(2d\beta^2 \rho)^{k-2}\mathcal{I}_{T}^{(\boldsymbol{\lambda})}(y_0,\ldots, y_k)\quad \text{as $n \to \infty$} . \end{split} \end{equation}
We now define functions on $\mathbb{R}^d$ by
\[h_1^{(n)}(z) = n^{-(4-d)(k-1)+2}\tau^{(\phi(\boldsymbol{\lambda}))}_{\phi(T)}(x_{\phi(0)}, \dots, x_{\phi(k)}, \lfloor nz\rfloor) \]
and
\[h_2^{(n)}(z) = n^{(2d-4)}  \langle x_{J} - \lfloor nz\rfloor \rangle^{2-d} \langle x_{I} - \lfloor nz \rfloor\rangle^{2-d}\ . \]
These are constant on cubes 
\[Q_v=\frac{v}{n}+\big[0,\frac{1}{n}\big)^d.\]
Setting $H^{(n)}(z):=h_1^{(n)}(z)h^{(n)}_2(z)$, we rewrite the normalized sum on the left-hand side of \eqref{eq:peel} as
\begin{equation}
    \label{eq:peel2}
     \int_{\mathbb{R}^d} H^{(n)}(z) \, \mathrm{d}z = \int_{\mathbb{R}^d} h_1^{(n)}(z) h^{(n)}_2(z) \, \mathrm{d}z\ .
\end{equation}

The inductive hypothesis gives the pointwise convergence
\[h_1^{(n)}(z) \rightarrow \alpha^{2(k-1)-3}(2d\beta^2 \rho)^{(k-1)-2}\mathcal{I}^{(\phi(\lambda))}_{\phi(T)}(y_{i, i\neq I, J}, z)\ . \]
Similarly,
\[h_2^{(n)}(z) \to |y_J-z|^{2-d} |y_I - z|^{2-d}\ . \]
The a priori bounds of Lemma~\ref{lem:apriori} give the existence of a constant $C> 0$, dependent on the $y_j$s but not on $z$, such that
\[H^{(n)}(z) \leq C \sum_{i \neq I, J} |y_i - z|^{4-d} |y_I - z|^{2-d} |y_J - z|^{2-d}\ .  \]
Since $d > 6$, the upper bound of the last display is integrable. We may therefore apply the dominated convergence theorem 
with the representation \eqref{eq:peel2} to see that the left-hand side of \eqref{eq:peel} converges as $n \to \infty$ to
\[ \alpha^{2k-3}(2d\beta^2 \rho)^{k-2} \int_{\mathbb{R}^d} \mathcal{I}^{(\phi(\boldsymbol{\lambda}))}_{\phi(T)}(y_{i, i\neq I,J}, z) |y_I - z|^{2-d} |y_J - z|^{2-d} \, \mathrm{d}z\ . \]

The edges $\{y_I, z\}$ and $\{y_J, z\}$ are exactly what is removed from $T$ to produce $\phi(T).$ We thus see that the right-hand side of the last display is identical to the right-hand side of \eqref{eq:peel}, completing the proof.

\section{Auxiliary Lemmas \label{sec:aux}}
\subsection{Switching}
The following lemma allows us to control the probability of closing a single pivotal edge in an open cluster. We choose this edge to separate the vertices $x_I, x_J$ from Proposition~\ref{prop: tree-reduct} from the rest of the cluster. In so doing, we reduce the number of leaf vertices in the resultant connectivity tree, allowing us to argue via induction on the size of the tree.
\begin{proposition}\label{eqn: switching-D}
    Let $a, b\in \mathbb{Z}^d$ be distinct. Let $\omega$ be an outcome in which the edge $f$ is open. Suppose $f$ is pivotal for $a\lra b$ and denote by $\underline{f}$ the endpoint of $f$ connected to $a$ even after closing $f$. Let $c$ be a third vertex in $\omega$ and $f\notin \mathcal{P}(a,c)$. Then, no self-avoiding open path from $a$ to $c$ can traverse $f$. In particular, letting $\phi:\{0,1\}^{\mathcal{E}^d}\rightarrow\{0,1\}^{\mathcal{E}^d}$ where $\phi(\omega)$ is the configuration identical to $\omega$ except that the edge $f$ is closed.
    \[D(a,c)[\omega]=D(a,c)[\phi(\omega)]\]
    and
    \[\mathcal{P}(a,c)[\omega]=\mathcal{P}(a,c)[\phi(\omega)].\]
    Moreover, no self-avoiding open path from $\bar{f}$ to $b$ uses $f$ so $D(\bar{f},b)$ and $\mathcal{P}(\bar{f},b)$ are invariant by $\phi$. Likewise, no self-avoiding open paths from $a$ to $\underline{f}$ can use $f$, so $D(a,\underline{f})$ and $\mathcal{P}(a,\underline{f})$ are also invariant by $\phi$.
    \begin{proof}
        In the configuration $\phi(\omega)$, $a$ is not connected to $b$ by pivotality, but $\overline{f}\lra b$. On the other hand, $a\lra c$, since $f\notin \mathcal{P}(a,c)$. Thus, $c\notin \clust(\bar{f})[\phi(\omega)]$, otherwise $a\lra c \lra \overline{f}$. Suppose a self-avoiding open path $a\rightarrow c$ uses $f$. If it crosses $\underline{f}\rightarrow\bar{f}$, the segment $c\rightarrow \bar{f} $ is open in $\phi(\omega)$, a contradiction to $c\notin \clust(\bar{f})[\phi(\omega)]$. If it crosses $\overline{f}\rightarrow\underline{f}$, then $\overline{f}\in \clust(a)[\phi(\omega)]$, a contradiction to pivotality and $\bar{f}\rightarrow b$ in $\phi(\omega)$. The remaining statements follow from similar arguments.
    \end{proof}
\end{proposition}

\begin{lemma}
\label{lem:newcut}
 Fix $k \geq 2$; let $T \in \mathfrak{T}_{k+1}$, and let $x_0, \dots, x_k \in \mathbb{Z}^d$ be distinct. Suppose $x_{k-1}$ and $x_k$ are the vertices $x_I$, $x_J$ as defined in Proposition \ref{prop: tree-reduct} and set 
\[D_T(x_0, \dots, x_k, g) = \Gamma(x_0, \dots, x_k) \cap \{\mathcal{T}(x_0, \dots, x_k) = T\} \cap \{g = \opiv(x_0, \{x_{k-1}, x_k\}) \}\ . \]
Let $E_{\mathrm{leaf}}\subset E(T)$ denote the two terminal edges $\{g,x_k\}$ and $\{g,x_{k-1}\}$ of $T$ and $E_{\mathrm{root}}=E(T)\setminus E$ the remaining, root-side edges.
Then the following identity holds:
\begin{align*}
&\mathbb{E}[e^{-an^{-2}\sum_e \lambda_e D_e},D_T(x_0, \dots, x_k, g)]\\
=&\beta \mathbb{E}\left[e^{-an^{-2}\sum_{e\in E_{\mathrm{root}}}\lambda_e D_e^{\Z^d\setminus \clust(\bar{g})}}e^{-an^{-2}\sum_{f\in E_{\mathrm{leaf}}} \lambda_{e_f} D_f},H_T(x_0,\ldots,x_k, g)
\right]\ .
\end{align*}

\[H_T(x_0,\ldots,x_k, g)=\begin{array}{c}
\Gamma(x_0, \dots, x_{k-2}, \underline{g}) \cap \{\mathcal{T}(x_0, \dots, x_{k-2}, \underline{g}) = \phi(T)\}\\
\cap \{x_{k-1}, x_k \lra \overline{g} \text{ off } \clust(\underline{g})\} \cap \{\piv(\overline{g}, \{x_{k-1}, x_k\}) = \varnothing \} \end{array} \]
\end{lemma}
\begin{proof}
We give the proof in the case $\boldsymbol{\lambda} \equiv \mathbf{0}$ to simplify notation; the necessary modifications in the presence of the ``decorating'' factors are obvious.
We note that, for each outcome in the event $D_T(x_0, \dots, x_k, g)$, the edge $g$ is open. Consider the mapping $\phi: \, \{0, 1\}^{\edges} \to \{0, 1\}^{\edges}$, where $\phi(\omega)$ is obtained from $\omega$ by closing $g$ (if open in $\omega$) and leaving the status of all other edges unchanged. It is immediate that,  $\phi: D_T(x_0, \dots, x_k, g) \to \phi(D_T(x_0, \dots, x_k, g))$ is a bijection, and that 
\[\prob(D_T(x_0, \dots, x_k, g)) = \beta \prob(\phi(D_T(x_0, \dots, x_k, g)))\ .\]
In fact, $\phi$ is a bijection from the event $\{g \text{ is open}\}$ onto its image $\{g \text{ is closed}\}$, and so the inverse is a mapping from the second set onto the fiest.
The claim will then follow as soon as we show
\begin{equation}
    \label{eq:psiD}
    \phi(D_T(x_0, \dots, x_k, g)) = H_T(x_0,\ldots,x_k, g)
\end{equation}

We show this by showing that the right-hand side of \eqref{eq:psiD} is contained in the left-hand side, and vice-versa. Let us start by considering an outcome $\omega$ in the right-hand side of \eqref{eq:psiD}, noting that $g$ must be closed in $\omega$ so that $\overline{g}$ and $\underline{g}$ lie in different open clusters. We show that $\phi^{-1}(\omega) \in D_T(x_0, \dots, x_k, g)$, which suffices to show the claimed containment. To do this, we show that $\phi^{-1}(\omega)$ lies in each of the three events whose intersection defines $D_T(x_0, \dots, x_k, g)$.

In $\omega,$ there are open connections from $\underline{g}$ to each $x_i$ for $i < k-1$, and there are open connections from $\overline{g}$ to $x_{k-1}$ and $x_k$. Since $g$, and all edges in open connections from the previous sentence, are open in $\phi^{-1}(\omega)$, it follows that $\phi^{-1}(\omega) \in \Gamma(x_0, \dots, x_k)$. Since $x_{k-1}$ and $x_k$ are not connected to $x_0$ in $\omega$, it follows that $g \in \piv(x_0, \{x_{k-1}, x_k\})$ in $\phi^{-1}(\omega)$.   In $\phi^{-1}(\omega)$, all open paths from $x_0$ to $x_{k-1}$ or $x_k$ traverse $g$ from $\underline{g}$ to $\overline{g}$. If $\phi^{-1}(\omega)$ were not an element of $\{g = \opiv(x_0, \{x_{k-1}, x_k\}) \}$, there would be another pivotal $f$ in $\piv(x_0, \{x_{k-1}, x_k\})$ appearing after $g$ in open paths from $x_0$. Then $f \in \piv(\overline{g}, \{x_{k-1}, x_k\})$ in $\phi^{-1}(\omega)$, and hence in $\omega$. But this would contradict the fact that $\omega \in \{\piv(\overline{g}, \{x_{k-1}, x_k\}) = \varnothing\}$, and so we see that in fact $\phi^{-1}(\omega) \in \{g = \opiv(x_0, \{x_{k-1}, x_k\}) \}$.

The fact that $\phi^{-1}(\omega) \in \{\mathcal{T}(x_0, \dots, x_k) = T\}$ follows easily from the previous observations and the fact that $x_{k-1}, x_k$ are $x_I$ and $x_J$ from below Proposition~\ref{prop: tree-reduct}. Indeed, this fact implies that for every retained leaf $x_\ell$, $\ell\le k-2$
\[\piv(x_0, x_\ell) \cap \piv(x_0, x_{k-1}) = \piv(x_0, \underline{g}) \cap \piv(x_0, x_{k-1})\ ,\]
with a similar statement holding when $x_{k-1}$ is replaced by $x_k$. Since $\omega \in \{\mathcal{T}(x_0, \dots, x_{k-2}, \underline{g}) = \phi(T)\}$, this ensures that $\mathcal{T}(x_0, \dots, x_{k})$ in $\omega$ is produced from $\phi(T)$ by adjoining two children, namely $x_{k-1}$ and $x_k$, to $\underline{g}$. This completes the proof that the right-hand side of \eqref{eq:psiD} is contained in the left-hand side.

The proof of the fact that the left-hand side of \eqref{eq:psiD} is contained in the right follows from similar considerations, so we sketch it. If $\omega \in D_T(x_0, \dots, x_k, g)$, then by the fact that $x_{k-1}$ and $x_k$ are adjacent leaves in $T$, we see $\phi(\omega) \in \Gamma(x_0, \dots, x_{k-2}, \underline{g})$, and as in the previous paragraph, we see that in $\phi(\omega)$, we have $\mathcal{T}(x_0, \dots, x_{k-2}, \underline{g}) = \phi(T)$. The fact that $\phi(\omega)$ is an element of the remaining two events from the right-hand side of \eqref{eq:psiD} follow from the fact that in $\omega$, the edge $g$ is the extremal pivotal for $\{x_0 \lra x_{k-1}\}$ and $\{x_0 \lra x_k\}$. This proves \eqref{eq:psiD} and hence the lemma.
\end{proof}

\subsection{Modified truncation lemma}
We recall the definition \eqref{eq:GTdef} here: for a fixed vertex set $\mathcal{D}$, we set 
\begin{equation*}
    \Gamma_{T,k}(x_0, \dots, x_k; \mathcal{D}) = \bigcap_{j=1}^k\left\{x_0 \sa{\Z^d \setminus \mathcal{D}} x_j  \right\} \cap \{\mathcal{T}_{k+1} = T \}
\end{equation*}

and, analogously to \eqref{eqn: WTl-def},
\[W_T^{\boldsymbol{\lambda}}(\mathcal{D}) =  \exp\big(-an^{-2}\sum_{e\in E(T)}\lambda_e D_e^{\Z^d \setminus \cD}\big) \mathbf{1}_{\Gamma_{T,k}(x_0, \dots, x_k; \mathcal{D})}
\]
where the superscript indicates that distances are computed in the open subgraph of $\Z^d \setminus \cD$. The indicator ensures in particular that this function is continuous at points where a coordinate of $\boldsymbol{\lambda}$ equals zero.

We also introduce the notation $\partial_E \mathcal{D}$ for the external edge boundary of $\mathcal{D}$: the set of edges $\{x,y\} \in \edges$ with $x \in \cD$ and $y \notin \cD$.

\begin{lemma}\label{lem: GT-truncation}
    Let $R>0$. Suppose that $(x_0, \dots, x_k, \underline{g}) \in G(\epsilon^R, n)$. Let $x_I, x_J$ be chosen as at Proposition \ref{prop: tree-reduct}. Let $0 \leq G \leq 1$ be a function defined on subsets of $B(\underline{g}, 2K)$. Consider
 \begin{align}
        &\mathbb{E}\left[W^{(\lambda_1,\lambda_2)}_{\overline{g}}\W^{(\phi(\boldsymbol{\lambda}))}_{\phi(T)} G\left(\clust_{B(\underline{g}; 2K)}\left(\{x_0, \overline{g}\}\right)\right) ,\begin{array}{c}\Gamma\big((x_i)_{i\neq I,J}, \underline{g}\big), \mathcal{T}((x_i)_{i\neq I,J}, \underline{g}) = \phi(T), \\
        \{x_I \lra \overline{g} \text{ off }\,\clust(\underline{g})\} \circ \{x_J \lra \overline{g} \text{ off }\,\clust(\underline{g})\}
        \end{array}\right]\label{eq:ktrunc1} \text{ and}\\
        & \widetilde{\mathbb{E}}\left[\begin{array}{c}
        e^{-an^{-2}(\lambda_{e_I} D(x_I,\overline{g})+\lambda_{e_J}D(x_J,\overline{g}))}\mathbbm{1}_{\{x_I \lra \overline{g} \} \circ \{x_J \lra \overline{g}\}} \\ \times \mathbb{E}\left[\left. G\left(\clust_{B(\underline{g}; 2K)}\left(x_0\right) \cup \wto{\clust}_{B(\underline{g}; 2K)}\left(\overline{g}\right)\right) \W_T^{(\phi(\boldsymbol{\lambda}))} \mathbbm{1}_{\Gamma_{\phi(T), k}((x_i)_{i \neq I, J})}  \,\right|\,\partial_E \wto{\clust}_{B(\underline{g}; K^d)}(\overline{g}) \text{ closed}\right]
        \end{array}\right] \label{eq:ktrunc2}.
        \end{align}
        
        We have,  uniformly in $K$, in $G$, in $n$, and in $T \in \mathfrak{T}_{k+1}$, that
        \begin{align}
        |\eqref{eq:ktrunc1} - \eqref{eq:ktrunc2}| \leq C(\epsilon,R) n^{(4-d)k - 2 - d} K^{6-d}\ .\label{eq:ktruncstar}
    \end{align}
    
    Consider instead a simpler setting: the case that $(0, x_1, x_2) \in G(\epsilon, n)$ and $g$ is an edge in $B(K^{1/d} / 4)$. There exists a constant $C$ depending only on $\epsilon$ such that, considering the quantities
    \begin{equation} \label{eq:simplercase1}
    \begin{split}
        \mathbb{E}\left[G\left(\clust_{B(\underline{g}; 2K)}\left(\{0, \overline{g}\}\right)\right) \exp\left(-n^{-2} \left[\lambda_1 D(0, x_1) + \lambda_2 D(\overline{g}, x_2)   \right]\right); 0 \sda{B(2K)} \underline{g}, \, 0 \leftrightarrow x_1,\, \overline{g} \sa{\Z^d \setminus \clust(\underline{g})} x_2 \right]
        \end{split}
    \end{equation}
    and
    \begin{equation}\label{eq:simplercase2}
        \begin{split}
        \wto{\mathbb{E}}\left[\begin{array}{c}\exp\left(-n^{-2} \lambda_1 D(0, x_1)\right) \mathbf{1}_{0 \sda{B(2K)} \underline{g}, \, 0 \leftrightarrow x_1}\\
        \qquad \times \mathbb{E}\left[\left. G\left(\wto{\clust}_{B(\underline{g}; 2K)}\left(0\right) \cup \clust_{B(\overline{g}; 2K)(\overline{g})} \right) \exp \left( - n^{-2}\lambda_2 D(\overline{g}, x_2)   \right)   \mathbbm{1}_{\overline{g} \lra x_2} \,\right|\, \partial_E \wto{\clust}_{B(K^d)}(\underline{g}) \text{ closed} \right] \end{array}\right]
        \end{split}
    \end{equation}
    we have
    \begin{equation}
        \label{eq:simplercase3}
        |\eqref{eq:simplercase1} - \eqref{eq:simplercase2}| \leq C n^{4-2d} K^{(6-d)/d}\ .
    \end{equation}
    
\end{lemma}
In the special case that $k = 2$ and $G = 1$, taking $x_0 = 0$ for specificity,  there is only one possible realization of the connectivity tree, and the expression \eqref{eq:ktrunc1} is
\begin{equation*}
    \mathbb{E}\left[\W_T^{(\boldsymbol{\lambda})},  0 \lra  \underline{g},\, \{x_1 \lra \overline{g} \text{ off }\,\clust(\underline{g})\} \circ \{x_2 \lra \overline{g} \text{ off }\,\clust(\underline{g})\}
         \right]\ ;
\end{equation*}
the expressions \eqref{eq:ktrunc2} and \eqref{eq:ktruncstar}  are then, respectively,
\begin{align}
    \label{eq:ktrunc3pt2}
    &\wto{\mathbb{E}}\Big[\W^{(\lambda_1,\lambda_2)}_{\bar{g}}\, \mathbb{E}\big[e^{-an^{-2}\lambda_0 D(0, \underline{g})}, \mathbbm{1}_{0 \lra \underline{g}} \mid \partial_E \wto{\clust}_{B(\overline{g}; K^d)}(\overline{g}) \text{ closed} \big] ,\, x_1 \lra \overline{g}  \circ x_2 \lra \overline{g} \Big],\\
    &\qquad\text{and } C n^{6-3d} K^{(6-d)/d}\ .  \nonumber
    \end{align}

\begin{proof}
We begin with the more intricate first case of the lemma, where $(x_0, \dots, x_k, \underline{g}) \in G(\epsilon^R, n)$.
We assume that $x_I = x_{k-1}$ and $x_J = x_k$, since the argument is virtually identical otherwise; then $\phi(i) = i$ for $i < k-1$.

With this convention, the expression \eqref{eq:ktrunc1} is
\begin{equation}
\mathbb{E}\left(\W^{(\lambda_1,\lambda_2)}_{\overline{g}}\W^{(\phi(\boldsymbol{\lambda}))}_{\phi(T)} G\left(\clust_{B(\underline{g}; 2K)}\left(\{x_0, \overline{g}\}\right)\right),\begin{array}{c}\Gamma(x_0, \dots, x_{k-2}, \underline{g}), \mathcal{T}(x_0, \dots, x_{k-2}, \underline{g}) = \phi(T), \\
         \{\clust(\underline{g}) \cap \clust(\overline{g}) = \varnothing \} \cap \left[\{x_{k-1} \lra \overline{g}\} \circ \{x_k \lra \overline{g}\}\right]
        \end{array}\right)\ .
        \label{eq:ktrunc1a}
\end{equation}
As in our argument for the three-point function at \eqref{eqn: apply-IIC} above, we diagrammatically localize the non-intersection probability. We introduce two copies of our probability space and two independent copies $\prob, \wto{\prob}$ of our probability measure; we let $\omega$ and $\wto{\omega}$ denote corresponding typical sample points. 
The expression \eqref{eq:ktrunc1} above may be rewritten as
\begin{align}
\widetilde{\mathbb{E}}\left[ \W^{(\lambda_1,\lambda_2)}_{\overline{g}}\mathbbm{1}_{\left\{\overline{g} \lra x_{k-1}\right\}\circ \left\{\overline{g} \lra x_{k}\right\}}   \mathbb{E}\left( G\left(\wto{\clust}_{B(\underline{g}; 2K)}(0) \cup\clust_{B(\underline{g}; 2K) 
\setminus \widetilde{\clust}(\overline{g})}(\underline{g})\right) \W^{(\phi(\boldsymbol{\lambda}))}_{\phi(T)}(\wto{\clust}(\overline{g})) \right) \right]\label{eq:minustree0}\ .
\end{align}

For legibility, we will abbreviate
\[\wto{\clust}^{(K)} := \widetilde{\clust}_{B(\underline{g}; K^d)}(\overline{g}); \]
our initial goal is to compare \eqref{eq:minustree0} to
\begin{equation}
    \label{eq:minustree1}
\widetilde{\mathbb{E}}\left[ \W^{(\lambda_1,\lambda_2)}_{\overline{g}}\mathbbm{1}_{\left\{\overline{g} \lra x_{k-1}\right\}\circ \left\{\overline{g} \lra x_{k}\right\}}   \mathbb{E}\left(G\left( \wto{\clust}_{B(\underline{g}; 2K)}(0) \cup \clust_{B(\underline{g}; 2K) 
\setminus \wto{\clust}^{(K)} }(\underline{g})\right) \W^{(\phi(\boldsymbol{\lambda}))}_{\phi(T)}(\wto{\clust}^{(K)}) \right) \right]\ .
\end{equation}
 F ix an outcome $\wto{\omega}$ in the outer expectation, abbreviate  
 \[ \mathcal{D} = \wto{\clust}(\overline{g}) \setminus \wto{\clust}^{(K)}, \]
 and consider an outcome $\omega$ for which the 
 quantities in the respective inner $\mathbb{E}$ expectations differ. There are two ways that this can happen: either the respective decorated two-point functions differ, or the function $G$ applied to the respective vertex sets evaluated in $\omega$ differ.
 

Consider first the case that, in $\omega$,
\[ \W^{(\phi(\boldsymbol{\lambda}))}_{\phi(T)}(\wto{\clust}(\overline{g})) \neq \W^{(\phi(\boldsymbol{\lambda}))}_{\phi(T)}(\wto{\clust}^{(K)}). \]
In this $\omega$, we have $\W^{(\phi(\boldsymbol{\lambda}))}_{\phi(T)}(\wto{\clust}(\overline{g})) \neq 0$, and in particular, there exist geodesics for the chemical distance between each pair of vertices in $\{x_0, \dots, x_{k-2}\}$. Consider the subgraph of $\mathbb{Z}^d$ consisting of the union of these geodesics; this graph has leaves $x_0, \dots, x_{k-2}, \underline{g}$. By the choice of $\omega$, one of these geodesics must intersect $\mathcal{D}$. We can thus choose some $w \in \mathcal{D}$ which is a vertex of this subgraph. Either $w$ is a leaf of the union of geodesics, or  $w$ lies on a geodesic $\gamma$ between distinct leaves $x_i$ and $x_j$. We modify the union of geodesics previously constructed by removing any open paths which connect vertices of portions of $\gamma$ before $w$ with portions of $\gamma$ after $w$. 

We now consider a tree of the so-trimmed union of geodesics spanning $\{x_0, \dots, x_{k-2}\}$; $w$ is a vertex of this tree due to the trimming. We see that in $\omega$  there exists a nontrivial partition $A \cup B = \{x_0, \dots, x_{k-2}, \underline{g}\}$ such that 
\[\omega \in \Gamma(\{w\} \cup A) \circ \Gamma(\{w\} \cup B)\ .\]
In particular, the difference between \eqref{eq:minustree1} and \eqref{eq:minustree0} is bounded above (using the fact that the quantity in the inner expectation lies in $[0, 1]$) by
\begin{equation}
    \label{eq:minustree2}
    \begin{split}
    \sum_{A, B} \sum_{w \in \Z^d} &\mathbb{P}\left( \Gamma(\{w\} \cup A) \circ \Gamma(\{w\} \cup B) \right)\\ &\widetilde{\mathbb{P}}\left( \left\{w \in \wto{\clust}(\overline{g}) \setminus \wto{\clust}^{(K)}\right\}  \cap\left[ \left\{\overline{g} \lra x_{k-1}\right\}\circ \left\{\overline{g} \lra x_{k}\right\} \right]  \right).
    \end{split}
\end{equation}
The sum over $A,B$ is over nontrivial partitions as described above.

We wish to simplify the final probability appearing in \eqref{eq:minustree2}.
For any outcome $\wto{\omega}$ in this event, either $w \notin B(\underline{g}; 2K)$, or $w \in B(\underline{g}; 2K)$, but every open path in $\wto \omega$ from $\overline{g}$ to $w$ exits $B(\underline{g}; 2K^d)$.
In the former case, for $w \notin B(\underline{g}; 2K)$, there is some vertex $z$ on an open path from $\overline{g}$ to either $x_k$ or $x_{k-1}$ such that
\begin{equation}\label{eq:wbox}
\begin{split}
    \Big(\bigcup_z \left[\{\overline{g} \lra x_{k-1}\} \circ \{\overline{g} \lra z \}\circ \{z \lra w\} \circ \{z \lra x_k\} \right]\Big)
       \\
       \cup \Big(\bigcup_z \left[\{\overline{g} \lra x_{k}\} \circ \{\overline{g} \lra z \}\circ \{z \lra w\} \circ \{z \lra x_{k-1}\} \right]\Big)
       \end{split}
\end{equation}
occurs. 
In the latter case, we choose disjoint $\gamma_1$, $\gamma_2$ realizing the open connections from $\overline{g}$ to $x_{k-1}$ and $x_{k}$ respectively and choose an open path from $w$ until its first intersection with $\gamma_1 \cup \gamma_2$. Choosing subpaths of these open paths, we find open paths witnessing 
\begin{equation}
    \label{eq:wbox2}
    \begin{split}
        &\left\{\overline{g} \lra x_{k-1}\right\}\circ \left\{\overline{g} \lra x_{k}\right\} \circ \{w \lra \partial B(\underline{g}; 2K^d) \}\\
        \cup & \left\{w \lra x_{k-1}\right\}\circ \left\{\overline{g} \lra x_{k}\right\}  \circ \{\overline{g} \lra \partial B(\underline{g}; 2K^d) \}\\
                \cup & \left\{w \lra x_{k-1}\right\}\circ \left\{w \lra x_{k}\right\}  \circ \{\overline{g} \lra \partial B(\underline{g}; 2K^d) \}\ .
    \end{split}
\end{equation}

Returning to \eqref{eq:minustree2} and the associated discussion, and applying \eqref{eq:wbox} and \eqref{eq:wbox2}, we see that \eqref{eq:minustree2} is bounded above by
\begin{equation}
\label{eq:minusterms}
\begin{split}
     \sum_{A, B}\sum_{\substack{w \notin B(\overline{g}; 2K) \\ z \in \Z^d}} &\prob \left( \Gamma(w, A)  \right) \prob \left( \Gamma(w, B)  \right)\tau(\overline{g}, z) \tau(z, w) \left[\tau(\overline{g}, x_{k-1}) \tau(z, x_k) + \tau(\overline{g}, x_{k}) \tau(z, x_{k-1}) \right]\\
      + \sum_{A, B}\sum_{\substack{w \in B(\overline{g}; 2K) \\ z \in \Z^d}} \Big[ &\prob \left( \Gamma(w, A)  \right) \prob \left( \Gamma(w, B)  \right)\prob(0 \lra \partial B(K^d)) \\
      &\left[\tau(\overline{g}, x_{k-1}) \tau(\overline{g}, x_k) + \tau(w, x_{k-1}) \tau(\overline{g}, x_k) + \tau(\overline{g}, x_{k-1}) \tau(w, x_k)\right] \Big]\ .
     \end{split}
\end{equation}

Before evaluating the estimate \eqref{eq:minusterms}, we return to the other way in which the quantities in \eqref{eq:minustree0} and \eqref{eq:minustree1} can differ: when the argument of the function $G$ differs, or in other words, when
\[ \clust_{B(\underline{g}; 2K) 
\setminus \widetilde{\clust}(\overline{g})}(\underline{g}) \neq   \clust_{B(\underline{g}; 2K) 
\setminus \wto{\clust}^{(K)}}(\underline{g}) .\]
This can only happen if there is some $w \in B(2K)$ and some neighbor $w'$ of $w$ such that $w \in \wto{\clust}(\overline{g}) \setminus \wto{\clust}^{(K)}$. Applying this observation as at \eqref{eq:wbox2} shows that the contribution of such configurations to the difference between \eqref{eq:minustree0} and \eqref{eq:minustree1} is bounded by
\begin{align}
    \nonumber \sum_{\substack{w \in B(\overline{g}; 2K) \\ z \in \Z^d}} \Big[ \mathbb{P}( 
     \Gamma_{\phi(T), k-1} )\prob(0 \lra \partial B(K^d)) \left[\tau(\overline{g}, x_{k-1}) \tau(\overline{g}, x_k) + \tau(w, x_{k-1}) \tau(\overline{g}, x_k) + \tau(\overline{g}, x_{k-1}) \tau(w, x_k)\right] \Big]\ . \nonumber
\end{align}

This is controlled much in the same way as the second term of \eqref{eq:minusterms}. We therefore devote the remainder of the proof to estimating the error terms in \eqref{eq:minusterms}.

We upper bound the first term from \eqref{eq:minusterms} for a fixed choice of $A$ and $B$. Applying Lemma~\ref{lem:apriori} to control the probabilities of the $\Gamma$ events, we see the  term above is bounded by
\begin{align*}
    C n^{(4-d)k-6}\sum_{\substack{w \notin B(\overline{g}; 2K) \\ z \in \Z^d}} \langle z-w\rangle^{2-d} \langle z-x_k\rangle^{2-d} \langle \overline{g}-z\rangle^{2-d} \mathrm{dist}(w, \{x_0, \dots, x_{k-2}, \underline{g}\})^{4-d}\ .
\end{align*}
The dominant contribution to the above sum comes from $w$ and $z$ within distance of order $n$ from $\{x_0, \dots, x_{k-1}, g\}$. We thus bound by
\begin{align}
        \leq&~C n^{(4-d)k-2-d} n^{-2}\sum_z\sum_{\substack{w \notin B(\overline{g}; 2K)}} \langle z-w\rangle^{2-d} \langle \overline{g}-z\rangle^{2-d}  \mathrm{dist}(w, \{x_0, \dots, x_{k-2}, \underline{g}\})^{4-d} \nonumber\\
        \leq&~C n^{(4-d)k-2-d} n^{-2}\sum_{\substack{w \notin B(\overline{g}; 2K)}} \langle \overline{g}-w\rangle^{8-2d}\nonumber \\
        \leq&~C n^{(4-d)k - 2 - d} K^{6-d}\ , \label{eq:Kest1}
\end{align}
where we have used the standard convolution estimate \eqref{eqn: std-conv} and isolated the dominant contribution with $w$ closest to $g$.

The other term is bounded similarly, now also using the one-arm probability estimate \eqref{eqn: KN}. This leads to the bound
\begin{align}
    \nonumber &C K^{-2d} n^{(4-d)(k-1)-4} \sum_{\substack{w \in B(\overline{g}; 2K)}} \mathrm{dist}(w, \{x_0, \dots, x_{k-2}, g\})^{4-d}\\
    \nonumber &\qquad  \left[ \langle\overline{g}- x_{k-1}\rangle^{2-d} \langle\overline{g}- x_k\rangle^{2-d} + \langle w- x_{k-1}\rangle^{2-d} \langle\overline{g}- x_k\rangle^{2-d} + \langle\overline{g}- x_{k-1}\rangle^{2-d} \langle w- x_k\rangle^{2-d}\right]\ ,
\end{align}
which is at most
\begin{align}
    \nonumber &C K^{-2d} n^{(4-d)(k-1)-4} \sum_{\substack{w \in B(\overline{g}; 2K}}   \langle w - \overline{g} \rangle^{4-d} n^{4-2d}\\
    \label{eq:Kest2} \leq~&C K^{-d} n^{(4-d)k-4 -d}\ .
\end{align}

Applying the estimates \eqref{eq:Kest1} and \eqref{eq:Kest2} in \eqref{eq:minusterms} shows that \eqref{eq:ktrunc1} can be approximated, up to an error of order
\[C n^{(4-d)k - 2 - d} K^{6-d}, \]
by the quantity \eqref{eq:minustree1}. The latter expression is easily seen to equal \eqref{eq:ktrunc2}, and the first case of the lemma is proved.

We now describe the necessary changes to treat the somewhat simpler cases of \eqref{eq:simplercase1} and \eqref{eq:simplercase2}. Just as at \eqref{eq:minustree0} above, we can write \eqref{eq:simplercase1} as
\begin{align}
    &\mathbb{E}\left[G(\clust_{B(2K)}(\{0, \overline{g}\})) \exp\left(-n^{-2} \left[\lambda_1 D(0, x_1) + \lambda_2 D(\overline{g}, x_2)   \right]\right); 0 \sda{B(2K)} \underline{g}, \, 0 \leftrightarrow x_1,\, \overline{g} \sa{\Z^d \setminus \clust(\underline{g})} x_2 \right]\nonumber \\
    = &\widetilde{\mathbb{E}}\left[ \exp\left(-n^{-2} \lambda_1 D(0, x_1)
    \right) \mathbbm{1}_{0  \sda{B(2K)} \underline{g}, \, 0 \leftrightarrow x_1} \mathbb{E}\left[\begin{array}{c} G\left(\wto{\clust}_{B(\underline{g}; 2K)}(0) \cup \clust_{B(\underline{g}; 2K) 
\setminus \widetilde{\clust}(\underline{g})}(\overline{g}) \right)\times\\ \exp\left(-n^{-2} \lambda_2 D^{\Z^d \setminus \wto{\clust}(\underline{g})}(\overline{g}, x_2) \mathbbm{1}_{ \overline{g} \sa{\Z^d \setminus \wto{\clust}(\underline{g})} x_2}  \right) \end{array} \right] \right]. \label{eq:simpfold}
\end{align}

We wish to replace the $\widetilde{\clust}(\underline{g})$ in the inner expectation with $\widetilde{\clust}_{B(K^d)}(\underline{g})$. Again, the quantity in the last display can change under such a replacement only when there is a vertex $w$ in $\widetilde{\clust}(\underline{g}) \setminus \widetilde{\clust}_{B(K^d)}(\underline{g})$, and a vertex $w'$ which is either $w$ or one of its neighbors, such that
\begin{itemize}
    \item $w \in B(2K)$, or
    \item  $w \notin  B(2K)$ and $\omega \in \{ \overline g \lra w \} \circ \{w \lra x_2 \}$.
\end{itemize}

Thus, the error incurred when replacing the $\widetilde{\clust}(\underline{g})$ in the inner expectation with $\widetilde{\clust}_{B(2K^d)}(\underline{g})$ in \eqref{eq:simpfold} is bounded above by the sum of
\begin{align}
    \label{eq:simplerdiagram0}
    \sum_{w \in B(2K)} \left(\tau(0, x_1) \prob(w \lra  \partial B(2K^d)) + \tau(w, x_1) \prob(0 \lra  \partial B(2K^d))\right) \tau(\overline{g}, x_2) \leq    C K^{-d} \tau(0, x_2) \tau(0, x_1) 
\end{align}
from the first bullet point, and
\begin{align}
\sum_{w \notin B(2K)} \tau(\overline{g}, w) \tau(w, x_2) \wto{\mathbb{P}}\left(\underline{g} \leftrightarrow x_1, \,  w \in \wto{\clust}(\underline{g}) \setminus \wto{\clust}_{B(2K)}(\underline{g}) \right)\label{eq:simplerdiagram}
\end{align}
from the second.

 The first term \eqref{eq:simplerdiagram0} has already been put into workable form; the second, \eqref{eq:simplerdiagram}, is a simplified version of the diagram appearing above. 
We bound the inner probability by
\begin{align*}
\wto{\mathbb{P}}\left(\underline{g} \leftrightarrow x_1, \,  w \in \wto{\clust}(\underline{g}) \setminus \wto{\clust}_{B(2K)}(\underline{g}) \right) \leq \sum_{z \in \Z^d} \tau(\overline{g}, z) \tau(z, w) \tau(z, x_1)\ .
\end{align*}
The remainder of the argument proceeds similarly to the summations used in concluding the first part of the lemma.

\end{proof}

\subsection{Bubble lemmas}
We have proved several lemmas about the tree-like behavior of open clusters conditional on $\Gamma(x_0, \dots, x_k)$. We also require some refined estimates for the behavior near an interior vertex (i.e.~near a first common pivotal edge) of this tree. The analysis has much the same spirit, but is slightly more complex in some ways because an interior vertex necessarily exhibits multiple long disjoint open connections. On the other hand, the inductive analysis of Section \ref{sec:inductive} focuses on particularly simple interior vertices which are adjacent to at least two leaves of the connectivity tree. 

In this section, we prove two lemmas which provide the necessary control on clusters near an interior vertex. The first result, Lemma~\ref{beta-lem2}, is similar in spirit to Lemma~\ref{lem:newcut}:

\begin{lemma}   \label{beta-lem2}
    Let $G$ be any  real function defined on vertex subsets of $B(2K)$; suppose $x_1, x_2 \notin B(2K)$. For each edge $f$, we have
    \begin{equation}
\label{eq:fexpandlem}
\mathbb{E}\left[G(\clust_{B(2K)}(0)) e^{-an^{-2}(\lambda_1 D(0,x_1)+\lambda_2 D(0,x_2))}; \,  \Gamma(0, x_1, x_2) \cap \mathcal{P}_0 \cap \{f = \upiv(0, x_2)\} \right]
\end{equation}
is equal to
\begin{equation}
\label{eq:fexpandlem2}
\beta  \mathbb{E}\left[G(\clust_{B(2K)}(\{0, \overline{f}\})) e^{-an^{-2}\lambda_1 D(0,x_1)}e^{-an^{-2}\lambda_2 (D(0,\underline{f})+1+D^{\Z^d\setminus \clust(\underline{f})}(\bar{f},x_2))}; \,  0 \Longleftrightarrow \underline{f}, \, 0 \leftrightarrow x_1, \, \overline{f} \leftrightarrow x_2 \text{ off } \clust(\underline{f}) \right]\ .
\end{equation}
\end{lemma}

\begin{proof} 
We note that, for each outcome in the event from \eqref{eq:fexpandlem}, the edge $f$ is open. Consider the mapping $\varphi: \, \{0, 1\}^{\edges} \to \{0, 1\}^{\edges}$, where $\varphi(\omega)$ is obtained from $\omega$ by closing $f$ (if open in $\omega$) and leaving the status of all other edges unchanged. It is immediate that, if $A \subseteq \{f \text{ is open}\}$ is measurable, then $\prob(A) = \beta \prob(\varphi(A))$, and $\varphi: A \to \varphi(A)$ is a bijection.

Note further that, if $f$ is open in $\omega$ and $0 \lra \underline{f}$ in $\omega$, then
\[\clust_{B(2K)}(0)[\omega] = \clust_{B(2K)}(\{0, \overline{f}\})[\varphi(\omega)] \]
where we introduce square brackets to denote dependence on the configuration. The result will thus follow if we can show
\begin{equation}
\label{eq:pop1}
\varphi(\{\Gamma(0, x_1, x_2) \cap \mathcal{P}_0 \}\cap \{f = \upiv(0, x_2)\})  
\end{equation}
is equal to
\begin{equation}
\label{eq:pop2}
\{  0 \Longleftrightarrow \underline{f}, \, 0 \leftrightarrow x_1, \, \overline{f} \leftrightarrow x_2 \text{ off } \clust(\underline{f})\}\ . 
\end{equation}

We show that the event in \eqref{eq:pop1} is contained in the event from \eqref{eq:pop2}; the other containment is proved similarly. Consider an outcome $\omega$ from the event inside $\varphi$ in \eqref{eq:pop1}.   Since $f$ is a pivotal edge for $\{0 \lra x_2\}$ and this event occurs, there is an open connection from $0$ to $\underline{f}$ that does not use $f$. Since $f$ is the first such pivotal, Menger's theorem implies that there are two edge-disjoint such connections. When applying $\varphi$, these connections still exist, showing that $\varphi(\omega)$ exhibits the connections from $0$ to $\underline{f}$ described in \eqref{eq:pop2}.

We note that $\omega \in \{0 \lra x_1\}$. Since $\omega \in \mathcal{P}_0$, the edge $f$ cannot be pivotal in $\omega$ for $\{0 \lra x_1\}$; thus, $\varphi(\omega) \in \{0 \lra x_1\}$. Finally, we note $\overline{f}$ has an open connection $\gamma$ to $x_2$, with $f \notin \gamma$ in $\omega$. There can be no open path avoiding $f$ from $\underline{f}$ to $\gamma$, since then $f$ would not be pivotal for $\{0 \lra x_2\}$. In particular, $\overline{f} \notin \clust(\underline{f})$ in $\varphi(\omega)$, but since $\gamma$ is still open in $\varphi(\omega)$, we see $\varphi(\omega) \in \{\overline{f} \leftrightarrow x_2 \text{ off } \clust(\underline{f})\}$.

Pulling the arguments in the last two paragraphs together, we see the claimed inclusion of \eqref{eq:pop1} into \eqref{eq:pop2}.

Finally, by pivotality, on $\mathcal{P}_0 \cap \{f = \upiv(0, x_2)\}$, we have the decomposition:
\begin{align*}
D(0,x_2)&=D(0,\underline{f})+1+D(\overline{f},x_2)\\
&=D(0,\underline{f})+1+D^{\Z^d\setminus \clust(\underline{f})}(\overline{f},x_2).
\end{align*}
None of the distances in the last display change when applying the map $\varphi$ to an outcome of $\mathcal{P}_0 \cap \{f = \upiv(0, x_2)\}$.
\end{proof}

Our second lemma controlling clusters near interior vertices appears below as Lemma~\ref{lem: bub-trunc}. Open clusters are not typically trees; when $\{0 \lra x_1 \} \circ \{0 \lra x_2\}$ occurs, there are often open paths from $x_1$ to $x_2$ which do not contain $0$. Lemma~\ref{lem: bub-trunc} controls in a certain sense the maximal distance between such an open path and $0$. 
\begin{lemma}
\label{lem: bub-trunc}
Let $\epsilon > 0$ be fixed. There exists a $C = C(\epsilon) > 0$ such that, uniformly in $M \geq 1$, in all $n \geq 2 M/ \epsilon$, and all $|x_1-x_2|\ge \epsilon$, $|x_1|,|x_2|\le \epsilon^{-1}n$, we have
\begin{equation}\label{eq:bubble-trunc}
    \sum_{f \not\in B(M)} \widetilde{\mathbb{P}}( 0\Longleftrightarrow\underline{f},\,0\leftrightarrow x_1,\overline{f}\leftrightarrow x_2\,\text{off}\,\widetilde{\clust}(\underline{f}), \mathcal{P}(0,x_1)\cap\mathcal{P}(\overline{f},x_2)=\varnothing) \leq C n^{4-2d} M^{6-d} \ .
\end{equation}
\end{lemma}
\begin{proof}
    Consider an outcome $\omega$ in the event appearing in \eqref{eq:bubble-trunc}. In $\omega$, there exist disjoint open paths $\gamma_1, \gamma_2$ connecting $0$ to $\underline{f}$ and a third disjoint open path connecting $\overline{f}$ to $x_2$. Following an open path of $\omega$ from $x_1$ to its first intersection with $\gamma_1 \cup \gamma_2$, we find a vertex $z\in \Z^d$ and witnesses for
    \[\{0 \lra z \} \circ \{z \lra \underline{f}\} \circ \{0 \lra \underline{f}\} \circ \{z \lra x_1\} \circ \{\overline{f} \lra x_2\}\ . \]
    Applying the BK inequality, we see the sum appearing in \eqref{eq:bubble-trunc} is bounded by
    \begin{align*}
        &\sum_z \sum_{f\not\in B(M)} \tau(0, z) \tau(z, \underline{f}) \tau(0, \underline{f}) \tau(z, x_1) \tau(\overline{f}, x_2)\\
        \leq~& C \sum_z\sum_{w \not\in B(M)} \langle z\rangle^{2-d} \langle z - w \rangle^{2-d} \langle w \rangle^{2-d} \langle z - x_1 \rangle^{2-d} \langle w - x_2 \rangle^{2-d}\ .
        \end{align*}

        Applying \eqref{eqn: triple-sum} to perform the $z$ sum and noting that the dominant contribution comes from $w$ which are closer to $0$ than to $x_1$ or $x_2$, we see the previous expression is at most
        \begin{align*}
         C \sum_{w \not\in B(M)}  \langle w \rangle^{6-2d} \langle w - x_1 \rangle^{2-d} \langle w - x_2 \rangle^{2-d}\ .
    \end{align*}
    Since $x_1, x_2 \in G(\epsilon, n)$ and since $d > 6$, the sum is again dominated by $w$ in the regime $|w - x_i| \geq \epsilon n/4$.  We apply \eqref{eqn: HHS} to bound the last line by
    \[C n^{4-2d} \sum_{w \not\in B(M)}  \langle w \rangle^{6-2d} \leq C n^{4-2d} M^{6-d}\ , \]
    and the result follows.
\end{proof}

The final lemma in this subsection is of a very similar type to the last lemma; it again localizes branching near an interior vertex of a large open cluster.
\begin{lemma} \label{lem: double-connection-trunc}
Let $\epsilon > 0$ be fixed. There exists a $C = C(\epsilon) > 0$ such that, for all $1 \leq M \leq K \leq \epsilon n / 2$ and edges $f \in B(M)$, vertices $x_1, x_2 \in F(\epsilon, n)$, we have
\begin{align*}
\wto{\mathbb{P}}\left( 0\Longleftrightarrow \underline{f}, \, 0 \leftrightarrow x_1 \right)
- \wto{\mathbb{P}}\left(0\stackrel{B(2K)}{\Longleftrightarrow} \underline{f}, \, 0 \leftrightarrow x_1 \big) \right) \leq C K^{-2} n^{2-d}\ .
\end{align*}
\end{lemma}

\begin{proof}
For parameters $M, K, n$ in the above-mentioned ranges, we upper-bound the difference of probabilities by
\begin{align}
    \wto{\mathbb{E}}\left[\mathbf{1}_{ 0\Longleftrightarrow \underline{f}, \, 0 \leftrightarrow x_1} - \mathbf{1}_{ 0\stackrel{B(2K)}{\Longleftrightarrow} \underline{f}, \, 0 \leftrightarrow x_1}  \right] =\wto{\mathbb{E}}\left[ \mathbf{1}_{ 0\Longleftrightarrow \underline{f} \text{ through } \Z^d \setminus B(2K), \, 0 \leftrightarrow x_1} \right]\ .\label{eq:doublelong}
\end{align}
We  claim that for outcomes
\begin{equation} \label{eq:doublewto} \wto{\omega} \in \left\{0\Longleftrightarrow \underline{f} \text{ through } \Z^d \setminus B(2K)\right\} \cap \{ 0 \leftrightarrow x_1\}\ ,\end{equation}
we also have
\begin{equation} \label{eq:doubletwo2} \wto{\omega} \in \{0 \leftrightarrow x_1\}\circ \{0\leftrightarrow \partial B(2K)\}.\end{equation}

Indeed, in an outcome $\wto{\omega}$ as in \eqref{eq:doublewto}, the vertex $0$ has two edge-disjoint open connections $\gamma_1, \gamma_2$ to $\underline{f}$, with (by relabeling if necessary)  $\gamma_1 \setminus B(2K) \neq \varnothing$. Defining the cycle $\gamma = \gamma_1 \cup \gamma_2$, 
there is a vertex $z$ (possibly equal to $x_1$) on $\gamma$ such that $z \sa{\Z^d \setminus \gamma} x_1$ occurs in $\wto{\omega}$. Then one of the disjoint subpaths of $\gamma$ from $0$ to $z$ exits $B(2K)$; this, along with the other subpath of $\gamma$ and the disjoint open path from $z$ to $x_1$ provide witnesses for the connections in \eqref{eq:doubletwo2}, showing that equation holds.

From this, we find that \eqref{eq:doublelong} is bounded by
\[\wto{\mathbb{P}}\big(0\leftrightarrow x_1\circ 0\leftrightarrow \partial B(2K)\big)\le CK^{-2}\tau(0,x_1),\]
where we used the Kozma-Nachmias bound \eqref{eqn: KN}. Applying \eqref{eqn: HHS} completes the proof.
\end{proof}

\section{Extensions of IIC convergence \label{sec:iic}}
We need the following upgraded version of the IIC convergence result \eqref{eqn: IIC-conv2}. It allows us to condition on connections to many vertices at once.

Recall the definition of $\Gamma_{T, i}(\cdot)$ from \eqref{eq:GTdef} above.
\begin{lemma} \label{lem:iicnew2}
    Fix an integer $k \geq 1$ and a tree $T \in \mathfrak{T}_{k+1}$.
    For each $n \geq 1$, suppose we have 
    a sequence $( \{ x_0^{(n)}, \dots, x_k^{(n)}\})_{n=1}^\infty$ where for each $n$,
    $(x_0^{(n)}, x_1^{(n)}, \dots, x_k^{(n)}) \in G(\epsilon, n)$, and  where $x_i^{(n)} = 0$ for all $n$ for some arbitrary $0 \leq i \leq k$. 
    Let
    \[\W_T^{\boldsymbol{\lambda}}:=\prod_e e^{-a \lambda_e n^{-2} D_e}.\]
    Suppose also that there exists a $c > 0$ such that
    \begin{equation}\liminf_{n \to \infty} n^{(d-4)k + 2}\tau_T^{(\lambda)}\left(x_0^{(n)}, \dots, x_{k}^{(n)} \right) \geq  c \ .
    \label{eq:assnt}
    \end{equation}
    \begin{enumerate}
    \item For each cylinder event $E$:
    \begin{equation} \label{eq:iic1} \lim_{n \to \infty} \frac{\mathbb{E}[\W^{\boldsymbol{\lambda}}_T\mathbf{1}_E,\,  \Gamma(x_0^{(n)}, \dots, x_k^{(n)}), \mathcal{T}(x_0^{(n)}, \dots, x_k^{(n)}) = T]}{\mathbb{E}[\W^{\boldsymbol{\lambda}}_T,\,  \Gamma(x_0^{(n)}, \dots, x_k^{(n)}), \mathcal{T}(x_0^{(n)}, \dots, x_k^{(n)}) = T]} = \nu(E)\ . \end{equation}
    \item For each fixed finite $\mathcal{D} \subseteq \Z^d$ such that $\nu(\Xi(\mathcal{D}))>0$, we have
    \begin{equation} \label{eq:iic2}
    \frac{\mathbb{E}\big[\W^{\boldsymbol{\lambda}}_T;\,
\Gamma_{T,i}(x_0^{(n)},\dots,x_k^{(n)},\mathcal{D}),\, \mathcal{T}(x_0^{(n)}, \dots, x_k^{(n)})=T\big]}
{\mathbb{E}\big[\W^{\boldsymbol{\lambda}}_T ;\,
\Gamma(x_0^{(n)},\dots,x_k^{(n)}),\, \mathcal{T}(x_0^{(n)}, \dots, x_k^{(n)})=T\big]\;\nu(\Xi(\mathcal{D}))}
\longrightarrow 1 \ . 
    \end{equation}
    \end{enumerate}
\end{lemma}
We have translated $x_i^{(n)}$ to the origin to simplify the statement of the above lemma (in particular, to avoid issues with translating $\mathcal{D}$); the obvious modification of the above clearly holds for general $(x_0^{(n)}, \dots, x_k^{(n)})$. A simple modification of the above holds for point-to-point connections, without the need to assume an analogue of \eqref{eq:assnt}. The proof is essentially identical but simpler, so we simply state the result here without proof:
\begin{corollary}
    \label{cor:iictwo}
    Let $x = \lfloor n y \rfloor$ for some $y \in \mathbb{R}^d$ with $y \neq 0$. Let $\lambda \geq 0$ be arbitrary, and let $E$ be a cylinder event. We have
    \[ \lim_{n \to \infty} \frac{\mathbb{E}[\exp(-a \lambda n^{-2} D(0, x)),\, E, 0 \lra x]}{\mathbb{E}[\exp(-a \lambda n^{-2} D(0, x))\,, 0 \lra x]} = \nu(E).\]
    Furthermore,  for each fixed finite $\mathcal{D} \subseteq \Z^d$ such that $\nu(\Xi(\mathcal{D}))>0$, we have
    \[ \lim_{n \to \infty} \frac{\mathbb{E}[\exp(-a \lambda n^{-2} D(0, x)),\, 0 \sa{\Z^d \setminus \cD} x]}{\mathbb{E}[\exp(-a \lambda n^{-2} D(0, x))\,, 0 \lra x]\ \nu(\Xi(\cD))} = 1.\]
\end{corollary}

We note that the lemma is applied in the argument for Theorem~\ref{thm:treeinduct} in an essentially inductive manner. The assumption \eqref{eq:assnt} is also an assumption of Theorem~\ref{thm:treeinduct}. In turn, Theorem~\ref{thm:treeinduct} is used in an inductive proof of Theorem~\ref{thm:kpt}. In the inductive step of that proof, we assume that \eqref{eq:assnt} holds for a fixed $k$ for all $T$ and all $(x_0^{(n)}, \dots, x_{k}^{(n)})$ and use that assumption to prove the same statement with $k$ replaced by $k+1$. This ensures there is no circularity in the argument.

\begin{proof}[Proof of Lemma~\ref{lem:iicnew2}]
    We give the argument in the case that $i = 0$; the other cases follow via an essentially identical argument. We denote by $e_0$ the edge of $T$ ending at the vertex corresponding to $x_0^{(n)}$. For clarity and without loss of generality we suppose $x_0^{(n)}=0$. Let $K$ be fixed relative to $n$, but large enough that $E$ is measurable with respect to the edges in $B(K)$. We will ultimately take $K \to \infty$ after taking $n \to \infty$ in the proof of \eqref{eq:iic2}.
    
    We define the event
    \[G_{n, K} =\Gamma\Big(0, \dots, x_k^{(n)}\Big) \cap \Big\{\mathcal{T}(0, \dots, x_k^{(n)}) = T \Big\} \cap \Big\{\upiv\left(B(K), \{x_1^{(n)}, \dots, x_k^{(n)}\}\right) \in B(K^d)\Big\}\ . \]
    Here, we have abused notation and denoted by $\upiv\left(B(K), \{x_1^{(n)}, \dots, x_k^{(n)}\right)$ the first common pivotal for all the connections $\{0\lra x_i^{(n)}\}$ with at least one endpoint outside $B(K)$.
    It follows immediately from Lemma~\ref{lem:apriori} that
    \begin{equation}
        \label{eq:binbran2}
        \lim_{K \to \infty} \limsup_{n \to \infty} n^{(d-4)k + 2}\mathbb{E}\left(\W^{\boldsymbol{\lambda}}, \Gamma\left(0, \dots, x_k^{(n)}\right)\cap \left\{\mathcal{T}(0, \dots, x_k^{(n)}) = T \right\} \setminus G_{n, K}\right) =0 \ . 
    \end{equation}
    Indeed, suppose the event appearing in \eqref{eq:binbran2} occurs with $n$ large enough that $x_1, \dots, x_k \notin B(2K^d)$. Then either a) there exist disjoint nonempty subsets $A, B$ partitioning $\{x_1, \dots, x_k\}$ such that the event
    \[\{x_i^{(n)} \lra B(K) \text{ for all $x_i^{(n)} \in A$}\} \circ \{x_i^{(n)} \lra B(K) \text{ for all $x_i \in B$}\} \]
    occurs, or b) there exists a vertex $z \notin B(K^d)$ such that
    \[\{B(K) \Longleftrightarrow z\} \circ \Gamma(z, x_1^{(n)}, \dots, x_k^{(n)})\]
    occurs. The probability of case a) is dealt with via an identical argument to the one appearing in the proof of Lemma~\ref{lem:binbran}. For case b), we sum over vertices of $B(K)$ to bound the probability of the above event by
    \[C K^d |z|^{4-2d} n^{(4-d)(k-1) - 2} \mathrm{dist}(z, \{x_1, \dots, x_k\})^{4-d}\ .  \]
    Summing over $z \notin B(K^d)$ and taking $n$ and then $K$ to infinity shows \eqref{eq:binbran2}.

    By assumption \eqref{eq:assnt} and $0\le \W^{\boldsymbol{\lambda}} \le 1$, it suffices to show that
    \begin{equation}
        \label{eq:Gtree}
        \lim_{K \to \infty} \limsup_{n \to \infty} \left| \prob(E \mid G_{n, K})  - \nu(E) \right| = 0
    \end{equation}
    and
    \begin{equation}
        \label{eq:Gtree2}
        \lim_{K \to \infty} \limsup_{n \to \infty} \left| \frac{\mathbb{E}\big[\W^{\boldsymbol{\lambda}};\,
\Gamma_{T,i}(\ldots ;\mathcal{D})\cap G_{n,K}\big]}
{\mathbb{E}\big[\W^{\boldsymbol{\lambda}};\,
G_{n,K}\big]\;\nu(\Xi(\mathcal{D}))}-1 \right| = 0\ .
    \end{equation}
    First, define
        \[G_{n, K}(f) =\Gamma(0, \dots, x_k^{(n)}) \cap \left\{\mathcal{T}\left(0, \dots, x_k^{(n)}\right) = T \right\} \cap \{\upiv(B(K), \{x_1^{(n)}, \dots, x_k^{(n)}\})= f\} \]
        and then decompose
    \begin{align}
        \label{eq:Edecomp} \prob(E \cap G_{n, K}) &= \sum_{f \in B(K^d)} \prob(E \cap G_{n,K}(f))\ ,\\
        \label{eq:Edecomp2} \prob(\Gamma_T(0, \dots, x_k; \mathcal{D}) \cap G_{n, K}) &=  \sum_{f \in B(K^d)} \prob(\Gamma_T(0, \dots, x_k; \mathcal{D}) \cap G_{n, K}(f))\ .
    \end{align}
    
    We explore the clusters of $x_1^{(n)}, \dots, x_k^{(n)}$ outside $B(K)$; let
    \[\mathfrak{C}(K) = \bigcup_{j=1}^k \clust_{\Z^d \setminus B(K)}\left(x_j^{(n)}\right) \]
    and
    \[Q(K) = \{z \in B(K): \, \{y, z\} \text{ open for some } y \in \mathfrak{C}(K)\}\ . \]
    We write $f=(\underline{f},\overline{f})$ with $\underline{f}$ the endpoint on the $Q(K)$ side and $\overline{f}$ the endpoint on the $\{x_j\}$ side. 
    We further decompose the terms of \eqref{eq:Edecomp} and \eqref{eq:Edecomp2}, writing
    \begin{align}
        \prob(E \cap G_{n,K}(f)) &= \sum_{\mathcal{C}} \prob(E \cap G_{n, K}(f) \cap\{ \mathfrak{C}(K) = \mathcal{C}\})\label{eq:Edecompa} ,\\
        \prob(\Gamma_T(0, \dots, x_k; \mathcal{D}) \cap G_{n, K}(f)) &= \sum_{\mathcal{C}}  \prob(\Gamma_T(0, \dots, x_k; \mathcal{D}) \cap G_{n, K}(f) \cap \{ \mathfrak{C}(K) = \mathcal{C}\}) \label{eq:Edecomp2a}\ .
    \end{align}
    
    We note that for the event in either of the above sums to be nonempty, the set $\mathcal{C}$ must be connected and contain $\overline{f}$, and the edge $f$ must be a cut-edge for each connection from a vertex of $Q(K)$ to a vertex $x_i^{(n)}$ for $0 \leq i \leq k$.
    Furthermore, $\mathcal{T}_{k+1}\left(y, x_1^{(n)}, \dots, x_k^{(n)}\right)$ must be well-defined and equal $T$ for each $y \in Q(K)$.
    
    We consider only such values of $\mathcal{C}$ in what remains.
    We note that, conditionally on  $\{\mathfrak{C}(K) = \mathcal{C}\} \cap \{Q(K) = q\}$, the event $\Gamma_T\left(0, \dots, x_k^{(n)}; \mathcal{D}\right) \cap G_{n, K}(f)$ occurs if and only if the following event does:
    \begin{equation}
    \label{eq:E2nda}
    \left\{0 \stackrel{\Z^d \setminus [\mathcal{C} \cup \mathcal{D}]}{\longleftrightarrow} q\right\}\ . 
    \end{equation}
     Similarly, conditionally on the event $\{\mathfrak{C}(K) = \mathcal{C}\} \cap \{Q(K) = q\}$ for any $q$ such that the event is nonempty, the event $E \cap G_{n, K}(f)$ occurs if and only if the following event does:
    \begin{equation}
    \label{eq:E2nd}
    E \cap \left\{0 \stackrel{\Z^d \setminus \mathcal{C}}{\longleftrightarrow} q\right\}\ . 
    \end{equation}

    We now complete the proof of \eqref{eq:Gtree2} and hence \eqref{eq:iic2}. An argument similar to the one proving Lemma~\ref{lem: GT-truncation}, which we will briefly sketch, shows that
    \begin{equation} \label{eq:doublecross}
        \lim_{R \to \infty} \lim_{K \to \infty} \sup_{q, \mathcal{C}} \left| \frac{\prob\left(0 \stackrel{\Z^d \setminus [\mathcal{C} \cup \mathcal{D}]}{\longleftrightarrow} q\right)}{\prob\left(0 \stackrel{\Z^d \setminus \mathcal{C}}{\longleftrightarrow} q\right) \prob\left(\left. 0 \sa{\Z^d \setminus \mathcal{D}} \partial B(R) \right| 0 \stackrel{\Z^d \setminus \mathcal{C}}{\longleftrightarrow} q \right)} -1\right| = 0\ ,
        \end{equation}
    where the supremum is over the same values of $q$ and $\mathcal{C}$ considered in \eqref{eq:nook}. Indeed, if $\{0 \stackrel{\Z^d \setminus \mathcal{C}}{\longleftrightarrow} q\}$ and $\{0 \sa{\Z^d \setminus \mathcal{D}} \partial B(R) \}$ occur but $\{0 \stackrel{\Z^d \setminus \left[ \mathcal{C} \cup \mathcal{D} \right]}{\longleftrightarrow} q\}$ does not, then the event $\{0 \lra \partial B(R)\} \circ \{\mathcal{D} \sa{\Z^d \setminus \mathcal{C}} q\} $ occurs. There is a $c = c(\mathcal{D})$ such that 
    \[\prob\left(0 \sa{\Z^d \setminus \left[\mathcal{C} \cup \mathcal{D} \right]} q \right) \geq c \prob\left(\mathcal{D} \sa{\Z^d \setminus \mathcal{C}} q \right)\]
    uniformly in $q$ and $\mathcal{C}$ as above, from which \eqref{eq:doublecross} follows.

    Using the last observation and the fact that, for $K$ large enough, $E$ and the values of $\mathfrak{C}(K)$ and $Q(K)$ depend on different bonds, we write (with all sums restricted to the class of $\mathcal{C}$ described above)
    \begin{equation}
    \label{eq:nook}
    \begin{split}
        &\prob( E \cap G_{n, K}(f))\\
        =&\sum_{q, \mathcal{C}} \prob(\mathfrak{C}(K) = \mathcal{C}, \, Q(K) = q, 0 \stackrel{\Z^d \setminus \mathcal{C}}{\longleftrightarrow} q) \prob(E \mid Q(K) = q, 0 \stackrel{\Z^d \setminus \mathcal{C}}{\longleftrightarrow} q)\ .
        \end{split}
    \end{equation}
    
    For each $\delta > 0$, we may choose a $K_0$ large such that, for all $K \geq K_0$, for all large $n$,
    \[(1 - \delta) \leq \frac{\prob(E \mid Q(K) = q, 0 \stackrel{\Z^d \setminus \mathcal{C}}{\longleftrightarrow} q)}{\nu(E)} \leq (1 + \delta)\ , \]
    uniformly in the values of $\mathcal{C}$ and $q$; this is a consequence of the IIC result \eqref{eqn: IIC-conv2}. 
    
    Examining the remaining factor of \eqref{eq:nook}, we note
    \begin{align*}
        \sum_{f \in B(K^d)} \sum_{q, \mathcal{C}} \prob(\mathfrak{C}(K) = \mathcal{C}, \, Q(K) = q, 0 \stackrel{\Z^d \setminus \mathcal{C}}{\longleftrightarrow} q) = \prob(G_{n, K})\ .
    \end{align*}
    We apply the statements of the last two displays in \eqref{eq:nook} to see that for $K \geq K_0$,
    \[(1-\delta) \nu(E) \leq\liminf_{n \to \infty} \prob(E \mid G_{n, K}) \leq \limsup_{n \to \infty} \prob(E \mid G_{n, K}) \leq (1 + \delta) \nu(E)\ .\]
    Taking $n \to \infty$ and then $K \to \infty$ completes the proof of \eqref{eq:Gtree} and hence the proof of \eqref{eq:iic1}.

    We write $b_0$ for the far endpoint of the segment represented by the edge $e_0$ in the connectivity tree.
     We now observe that on the event
    \[G_{n,K}(f)\cap\{\clust(K)=\mathcal C, Q(K)=q\},\]
    we have
    \begin{itemize}
\item[(a)] every simple open path between two vertices on the far side of $f$ stays on the far side. It follows that all decorations $D_e$ with $e\ne e_0$, and the far part $D(\overline{f},b_0)$ below, are
functions of cluster in the configuration $\omega^f$ where $f$ is closed.
\item[(b)] every simple open path from $0$ to a far-side vertex $w$ traverses $f$
exactly once, and the distance splits exactly \emph{at the pivotal edge}:
\begin{equation}\label{eq:split}
D(0,w)=D(0,\underline{f})+1+D(\overline{f},w).
\end{equation}
\end{itemize}
So we have the factorization
\begin{equation}\label{eq:factor}
\W^{\boldsymbol{\lambda}}_T\;=\;\W_{\mathrm{far}}\cdot e^{-a \lambda_{0} n^{-2} D(0,\underline{f})},
\qquad
\W_{\mathrm{far}}:=e^{- a \lambda_{0} n^{-2}\bigl(1+D(\overline{f},b_0)\bigr)}
\prod_{e\ne e_0}e^{-a n^{-2}\lambda_e D_e}.
\end{equation}
$\W_{\mathrm{far}}$ is a function of the edges with at least one endpoint in $\mathcal{C}$ and $0\le \W_{\mathrm{far}}\le1$.

\begin{lemma}
    Let $R>2K^d$.
    \[\mathbb{E}[\W^{\boldsymbol{\lambda}}, G_{n,K}(f)]=\mathbb{E}[\W_{\mathrm{far}}, G_{n,K}(f)]+o(n^{(4-d)k-2})\]
    \begin{proof}
        Define
        \[L_R(f)=\big\{0\xleftrightarrow{B(R)} \underline{f}\big\}.\]
        On $L_R(f)$, we have
        \[D(0,\underline{f})\le C_dR^d,\]
        so
        \begin{equation}\label{eqn: distanceL}
        0\le 1-e^{-a\lambda_0 n^{-2} D(0,\underline{f})}\le C\frac{a\lambda_0 R^d}{n^2}\rightarrow 0
        \end{equation}
        for $R=n^{1/d}$.
        On $L_R(f)^c \cap \{0 \lra \underline{f}\}$, there is a path $0\lra \underline{f}$ that exits $B(R)$, while $0, \underline{f}\in B(K^d)$. This implies the existence of two disjoint arms to $\partial B(R)$, from $0$ to the first exit from $B(R)$ and from the last entrance to $\underline{f}$. This gives
        \[G_{n,K}(f)\cap L_R(f)^c \subset \{0\lra \partial B(R)\}\circ \{\underline{f}\lra \partial B(R)\}\circ \Gamma(\overline{f},x_1^{(n)},\ldots,x_k^{(n)}),\]
        so, by BK and the arm probability bound \eqref{eqn: KN},
        \begin{equation}\label{eqn: Lnc}
        \mathbb{E}[\W^\lambda, G_{n,K}(f),L_R(f)^c]\le CR^{-4}n^{-(d-4)k-2}=o(n^{-(d-4)k-2}).
        \end{equation}
        Combining \eqref{eqn: distanceL} and \eqref{eqn: Lnc}, we obtain the desired result.
    \end{proof}
\end{lemma}

     We now begin by using the decomposition
     \begin{align*}
     &\mathbb{E}[\W^\lambda,\, G_{n, K}(f)]\\
     =~&(1+o(1))\sum_{q,\mathcal{C}} \mathbb{E}[\W_{\mathrm{far}}, \mathfrak{C}(K) = \mathcal{C}, \, Q(K) = q]  \mathbb{P}(0 \stackrel{\Z^d \setminus \mathcal{C}}{\longleftrightarrow} q),\\
        &\mathbb{E}[\W^\lambda,\,\Gamma_{T,0}(\ldots;\mathcal{D})\cap G_{n, K}(f)]\\
        =~&(1+o(1))\sum_{q, \mathcal{C}} \mathbb{E}[\W_{\mathrm{far}}, \mathfrak{C}(K) = \mathcal{C}, \, Q(K) = q]  \mathbb{P}(0 \stackrel{\Z^d \setminus [\mathcal{C} \cup \mathcal{D}]}{\longleftrightarrow} q).
        \end{align*}
    So we have
    \[\frac{\mathbb{E}\big[\W^\lambda;\,
\Gamma_{T,i}(\ldots ;\mathcal{D})\cap G_{n,K}\big]}
{\mathbb{E}\big[\W^\lambda ;\,
G_{n,K}\big]}=(1+o(1))\frac{\sum_{q,\mathcal{C}} \mathbb{E}[\W_{\mathrm{far}}, \mathfrak{C}(K) = \mathcal{C}, \, Q(K) = q]\mathbb{P}(0 \stackrel{\Z^d \setminus \mathcal{C}}{\longleftrightarrow} q )}{\sum_{q,\mathcal{C}} \mathbb{E}[\W_{\mathrm{far}}, \mathfrak{C}(K) = \mathcal{C}, \, Q(K) = q]\mathbb{P}(0 \stackrel{\Z^d \setminus [\mathcal{C} \cup \mathcal{D}]}{\longleftrightarrow} q)}\]
From \eqref{eq:doublecross}, we have
\[\limsup_{n\rightarrow\infty}\lim_{K \to \infty} \sup_{q, \mathcal{C}} \left| \frac{\prob\left(0 \stackrel{\Z^d \setminus \mathcal{C}}{\longleftrightarrow} q\right) \nu\big(\Xi(\mathcal{D})\big) }{\prob\left(0 \stackrel{\Z^d \setminus [\mathcal{C} \cup \mathcal{D}]}{\longleftrightarrow} q\right)} -1\right|=0\]
from which the desired conclusion follows.
\end{proof}

\section{Decorated two-point convergence with local obstacle}\label{sec:dec2}
    \begin{lemma}\label{lem:genlem1}
Let $\epsilon > 0, \lambda_1,\lambda_2\ge0$ be fixed. We have
\begin{align}\label{eq:genlem1}
\lim_{K \to \infty} \lim_{n \to \infty} \sup_{(x_1, x_2) \in F(\epsilon, n)} \left|\frac{\wto{\mathbb{E}}\left[\nu_{\mathbf{e}_1}(\Xi_{\mathbf{e}_1}(\wto{\clust}_{B(2K)}(0) )); \,  \Gamma(0, x_1, x_2) \cap \mathcal{P}_0 \right]}{ \tau(0, x_1) \tau(0, x_2)}  - \rho \beta\right| = 0\ .\end{align}

\begin{align}\label{eq:genlem1-dec}
\lim_{K \to \infty} \lim_{n \to \infty} \sup_{\substack{|x_1|,|x_2|\le \epsilon^{-1}n\\|x_1-x_2|\ge \epsilon n}} \left|\frac{\wto{\mathbb{E}}\left[\nu_{\mathbf{e}_1}(\Xi_{\mathbf{e}_1}(\wto{\clust}_{B(2K)}(0) ))e^{-an^{-2}(\lambda_1 D(0,x_1)+\lambda_2 D(0,x_2))}; \,  \Gamma(0, x_1, x_2) \cap \mathcal{P}_0 \right]}{ \tau(0, x_1) \tau(0, x_2)\psi_d(\lambda_1|x_1/n|^2)\psi_d(\lambda_2|x_2/n|^2)}  - \rho \beta\right| = 0\ .\end{align}

\end{lemma}

Before proving the lemma, we state a key result. 
\begin{theorem}\label{thm: avoid-dist}
Let $F\subset\Z^d$ be finite, and suppose that $0$ lies in the infinite component of $\Z^d\setminus F$. For $x\in\Z^d$ let
\begin{equation*}
A_x^F:=\{0\longleftrightarrow x\text{ in }\Z^d\setminus F\},
\qquad
D^F(0,x):=d_{\Z^d\setminus F}(0,x),
\end{equation*}
and define
\begin{equation*}
\Xi(F):=\{0\longleftrightarrow\infty\text{ in }W_{0}\setminus F\}
\end{equation*}
under the IIC measure $\nu$. For every $\lambda\ge 0$ and $\epsilon>0$,
\begin{equation*}
\sup_{\epsilon n\le |x|\le \epsilon^{-1}n}
\left|
\frac{\mathbb{E}\bigl[\exp(-\lambda a n^{-2}D^F(0,x)), A_x^F\bigr]}{\tau(0,x)}
-\nu(\Xi(F))\,\psi_d\bigl(\lambda |x/n|^2\bigr)
\right|\longrightarrow 0.
\end{equation*}
Here
\begin{equation*}
a=\frac{z_d}{2d\alpha_d \alpha\beta},
\qquad
\psi_d(s)=\mathbb{E}[\exp(-sT_d)],
\end{equation*}
where $T_d$ is distributed as in \eqref{eqn: limit-density}.
\end{theorem}
Theorem~\ref{thm: avoid-dist} is Theorem 2 of \cite{CCHS2}.

\subsection{Proof of Lemma \ref{lem:genlem1}}
\begin{proof}[Proof of Lemma \ref{lem:genlem1}]
We will introduce a parameter $M \leq \frac{1}{4}K^{1/d}$ for the purpose of the argument. Throughout the proof, we will assume we are in the regime $K \ll n$. At the conclusion of the proof, we take $n \to \infty$ (hence $|x_i| \to \infty$ for $i = 1, 2$), then $K \to \infty$.

Introduce the weight
\[\mathcal{W}^{(\lambda_1,\lambda_2)}=e^{-n^{-2}a(\lambda_1 D(0,x_1)+\lambda_2 D(0,x_2))}.\]
Then $0\le \mathcal{W}^{(\lambda_1,\lambda_2)}\le 1$.

The proof is an application of the distance convergence result in Theorem \ref{thm: avoid-dist}. It begins with a series of reductions obtained from applying results obtained elsewhere in the paper. For clarity, we break it into steps, each of which invokes a different lemma or proposition.

\paragraph{Step 1.} As at \eqref{eqn: piv-exists} above, it is easy to see that the event
\[\Gamma(0, x_1, x_2) \cap \mathcal{P}_0 \cap \{\upiv(0, x_2) = \varnothing\} \]
has much smaller probability than the numerator of the expression appearing in \eqref{eq:genlem1-dec}. We expand this numerator over the first pivotal $f$ for $\{0\leftrightarrow x_2\}$, obtaining
\begin{equation}
\label{eq:fexpand}
\sum_{f} \wto{\mathbb{E}}\left[\W^{(\lambda_1,\lambda_2)},\nu_{\mathbf{e}_1}\big(\Xi_{\mathbf{e}_1}(\wto{\clust}_{B(2K)}(0))\big); \,  \Gamma(0, x_1, x_2) \cap \mathcal{P}_0 \cap \{f = \upiv(0, x_2)\} \right]\ .
\end{equation}
On the event $\{f = \upiv(0, x_2)\}$, we have
\[D(0,x_2)=D(0,\underline{f})+1+D^{\mathbb{Z}^d\setminus \clust(\underline{f})}(\overline{f},x_2).\]
We apply Lemma~\ref{beta-lem2} to re-express this as 
\begin{equation}
\label{eq:fexpand2}
\beta \sum_{f} \wto{\mathbb{E}}\left[\W^{\lambda_1}\W^{\lambda_2}_{\mathrm{av}} \W_{\mathrm{err}}\,\nu_{\mathbf{e}_1}(\Xi_{\mathbf{e}_1}(\wto{\clust}_{B(2K)}(\{0, \overline{f}\}))); \,  0 \Longleftrightarrow \underline{f}, \, 0 \leftrightarrow x_1, \, \overline{f} \leftrightarrow x_2 \text{ off } \wto{\clust}(\underline{f}) \right]\ ,
\end{equation}
where 
\begin{align*}
\W^{\lambda_1}&=\exp(-an^{-2}\lambda_1 D(0,x_1)),\\
\W^{\lambda_2}_{\mathrm{av}}&=\exp(-an^{-2} \lambda_2 D(\overline{f},x_2)),\\
\W_{\mathrm{err}}&=\exp(-an^{-2}\lambda_2(1+D(0,\underline{f}))).
\end{align*}
Recall from \eqref{eqn: betadef} that $\beta=p_c/(1-p_c)$.

\paragraph{Step 2.} By Lemma~\ref{lem: bub-trunc}, we can replace the sum over all $f \in \edges$ with a sum over $f$ within distance $M \leq K$ of $0$, up to a correction of smaller order than $n^{4-2d}$. Specifically, defining the quantity
\begin{equation}
\label{eq:fexpand3}
\beta \sum_{f \in B(M)} \wto{\mathbb{E}}\left[\W^{\lambda_1}\W^{\lambda_2}_{\mathrm{av}}\W_{\mathrm{err}}\nu_{\mathbf{e}_1}(\Xi_{\mathbf{e}_1}(\wto{\clust}_{B(2K)}(\{0, \overline{f}\} )); \,  0 \Longleftrightarrow \underline{f}, \, 0 \leftrightarrow x_1, \, \overline{f} \leftrightarrow x_2 \text{ off } \wto{\clust}(\underline{f})\right],
\end{equation}
we have
\begin{equation}
    \label{eq:fexpand4}
    \left|\eqref{eq:fexpand3} - \eqref{eq:fexpand2} \right| \leq C n^{4-2d} M^{6-d} \quad \text{uniformly in $M$, $K$, $n$, $x_1$, and $x_2$}
\end{equation}
for some $C = C(\epsilon) > 0$.
The right-hand side of \eqref{eq:fexpand3} will be shown to be of order $n^{4-2d}$, the order of the denominator of the expression inside the limit in \eqref{eq:genlem1}.

 \paragraph{Step 3.} We perform a truncation on the indicator function appearing in \eqref{eq:fexpand3}. Note that the integrand of that expression is positive, and note further that
 \[  \{0 \Longleftrightarrow \underline{f}, \, 0 \leftrightarrow x_1, \, \overline{f} \leftrightarrow x_2 \text{ off } \wto{\clust}(\underline{f})\} \supseteq \{0 \sda{B(2K)} \underline{f}, \, 0 \leftrightarrow x_1, \, \overline{f} \leftrightarrow x_2 \text{ off } \wto{\clust}(\underline{f})\}. \]
 It follows that, defining
 \begin{equation}
     \label{eq:step3}
     \beta \sum_{f \in B(M)} \wto{\mathbb{E}}\left[\W^{\lambda_1}\W^{\lambda_2}_{\mathrm{av}}\W_{\mathrm{err}}\nu_{\mathbf{e}_1}(\Xi_{\mathbf{e}_1}(\wto{\clust}_{B(2K)}(\{0, \overline{f}\} )); \,  0 \sda{B(2K)} \underline{f}, \, 0 \leftrightarrow x_1, \, \overline{f} \leftrightarrow x_2 \text{ off } \wto{\clust}(\underline{f})\right]
 \end{equation}
 we have
 \begin{equation}
     \label{eq:step3a}
     \begin{split}
     0 &\leq \eqref{eq:fexpand3} - \eqref{eq:step3}\\
     &\leq \beta \sum_{f \in B(M)}\Big[ \wto \prob(0 \Longleftrightarrow \underline{f}, \, 0 \leftrightarrow x_1, \, \overline{f} \leftrightarrow x_2 \text{ off } \wto{\clust}(\underline{f})) - \wto \prob (0 \sda{B(2K)} \underline{f}, \, 0 \leftrightarrow x_1, \, \overline{f} \leftrightarrow x_2 \text{ off } \wto{\clust}(\underline{f}))\Big].
     \end{split}
 \end{equation}
 We can introduce another independent copy of our probability space, now denoted with double tildes (for instance, with probability measure $\dwt \prob$. With this notation, the difference in \eqref{eq:step3a} is exactly
\begin{equation}
    \label{eq:swissarmy}
    \wto{\mathbb{E}}[\mathbf{1}_{ 0\Longleftrightarrow \underline{f}, \, 0 \leftrightarrow x_1}\dwt{\mathbb{P}}\big( \, \overline{f}\lra x_2\big) ]
- \wto{\mathbb{E}}[\mathbf{1}_{ 0\stackrel{B(2K)}{\Longleftrightarrow} \underline{f}, \, 0 \leftrightarrow x_1}\dwt{\mathbb{P}}\big( \, \overline{f}\lra x_2 \big) ].
\end{equation}

 Applying Lemma~\ref{lem: double-connection-trunc} we see that, uniformly in the same parameters
\begin{equation}
\label{eq:doubleconnf}
\begin{split}
0 \leq \eqref{eq:swissarmy} \leq C K^{-2} n^{4-2d}\ .
\end{split}
\end{equation}
It follows that we can replace \eqref{eq:fexpand3} by \eqref{eq:step3} at the cost of an error of at most $C M^d K^{-2}$; we proceed to analyze \eqref{eq:step3}.

We proceed by analyzing \eqref{eq:doubleconnf}.
\paragraph{Step 4.}
We simplify the distances appearing in \eqref{eq:doubleconnf}. On the event $0\sda{B(2K)} \underline{f} $, we have $D(0,\underline{f})\le CK^d$, so
\begin{equation} \label{eq:whatterror} |\W_{\mathrm{err}}-1|\le C\frac{K^d}{n^2},\end{equation}
so we can further reduce \eqref{eq:step3} to
\begin{equation}
\label{eq:fgoal00}
\begin{split}
\beta \sum_{f \in B(M)} \wto{\mathbb{E}}\left[\W^{\lambda_1}\W^{\lambda_2}_{\mathrm{av}}\nu_{\mathbf{e}_1}(\Xi_{\mathbf{e}_1}(\wto{\clust}_{B(2K)}(\{0, \overline{f}\} )); \,  0 \sda{B(2K)} \underline{f}, \, 0 \leftrightarrow x_1, \, \overline{f} \leftrightarrow x_2 \text{ off } \wto{\clust}(\underline{f})\right]
\end{split}
\end{equation}
up to an error term of order
\begin{align*}
 C K^d n^{-2} \times M^d \tau(0, x_1) \tau(0, x_2)  \leq C K^d M^d n^{2-2d}   
\end{align*}
where the first factor comes from \eqref{eq:whatterror} and the second from integrating the indicator functions and performing the sum.

\paragraph{Step 5.}
We apply Lemma~\ref{lem: GT-truncation} to re-express \eqref{eq:fgoal00} as
\begin{equation}\label{eq:fgoal01}
        \begin{split}
        \wto{\mathbb{E}}\left[\begin{array}{c}\exp\left(-an^{-2} \lambda_1 D(0, x_1)\right) \mathbf{1}_{0 \sda{B(2K)} \underline{f}, \, 0 \leftrightarrow x_1}\\
        \qquad \times \dwt{\mathbb{E}}\Big[\Big. \nu_{\mathbf{e}_1}\Big(\Xi_{\mathbf{e}_1}\Big(\wto{\clust}_{B(\underline{f}; 2K)}\left(0\right) \cup \dwt{\clust}_{B(\overline{f}; 2K)}(\overline{f}) \Big)\Big) \exp \Big( - an^{-2}\lambda_2 D(\overline{f}, x_2)   \Big)   \mathbbm{1}_{\overline{f} \lra x_2} \,\Big|\, \partial_E \wto{\clust}_{B(K^d)}(0) \text{ closed} \Big] \end{array}\right].
        \end{split}
    \end{equation}
Treating $\wto{\clust}_{B(K^d)}(0)$ as fixed, we study the inner expectation of \eqref{eq:fgoal01}. Letting $\mathcal{D}$ denote a typical realization of $\dwt{\clust}_{B(\overline{f}; 2K)}(\overline{f})$, we can write this inner expectation as
\begin{align*}
    \sum_{\cD} \nu_{\mathbf{e}_1}\big(\Xi_{\mathbf{e}_1}\big(\wto{\clust}_{B(\underline{f}; 2K)}\left(0\right) \cup \cD \big)\big) \frac{\dwt{\mathbb{E}}\big[ \exp \left( - an^{-2}\lambda_2 D(\overline{f}, x_2)   \right)   \mathbbm{1}_{\overline{f} \lra x_2} \mathbbm{1}_{\dwt{\clust}_{B(\overline{f}; 2K)}(\overline{f}) = \cD } \mathbbm{1}_{\partial_E \wto{\clust}_{B(K^d)}(0) \text{ closed}} \big]}{\prob(\partial_E \wto{\clust}_{B(K^d)}(0) \text{ closed})}
\end{align*}
We apply Corollary~\ref{cor:iictwo} to see that the last expression has the same asymptotic behavior as $n \to \infty$ as
\begin{align}
    & \dwt{\mathbb{E}}\big[  \exp \left( - an^{-2}\lambda_2 D(\overline{f}, x_2)   \right)   \mathbbm{1}_{\overline{f} \lra x_2} \big] \sum_{\cD} \frac{\nu_{\mathbf{e}_1}\left(\Xi_{\mathbf{e}_1}\left(\wto{\clust}_{B(\underline{f}; 2K)}\left(0\right) \cup \cD \right)\right)\dwt{\nu}_{\overline{f}}\Big(\dwt{W}_{\overline{f}; B(\overline{f}; 2K)} = \cD,\, \partial_E \wto{\clust}_{B(K^d)}(0) \text{ closed} \Big)  }{\prob(\partial_E \wto{\clust}_{B(K^d)}(0) \text{ closed})}\nonumber\\
    =~&\dwt{\mathbb{E}}\left[  \exp \left( - an^{-2}\lambda_2 D(\overline{f}, x_2)   \right)   \mathbbm{1}_{\overline{f} \lra x_2} \right] \frac{\mathbb{E}_{\dwt{\nu}_{\overline{f}}}\Big[ \nu_{\mathbf{e}_1}\Big(\Xi_{\mathbf{e}_1}\Big(\wto{\clust}_{B(\underline{f}; 2K)}\left(0\right) \cup \dwt{W}_{\overline{f}; B(\overline{f}; 2K)} \Big)\Big) ;\,  \partial_E \wto{\clust}_{B(K^d)}(0) \text{ closed} \Big]  }{\prob(\partial_E \wto{\clust}_{B(K^d)}(0) \text{ closed})}\nonumber\\
    =~& \dwt{\mathbb{E}}\left[  \exp \left( - an^{-2}\lambda_2 D(\overline{f}, x_2)   \right)   \mathbbm{1}_{\overline{f} \lra x_2} \right] \frac{ \mathbb{E}_{\dwt{\nu}_{\overline{f}}}\Big[ \nu_{\mathbf{e}_1}\Big(\Xi_{\mathbf{e}_1}\Big(\wto{\clust}_{B(\underline{f}; 2K)}\left(0\right) \cup \dwt{W}_{\overline{f}; B(\overline{f}; 2K) \setminus \wto{\clust}_{B(K^d)}(0) } \Big)\Big) ;\,  \partial_E \wto{\clust}_{B(K^d)}(0) \text{ closed} \Big]  }{\prob(\partial_E \wto{\clust}_{B(K^d)}(0) \text{ closed})}\nonumber\\
    =~& \dwt{\mathbb{E}}\left[  \exp \left( - an^{-2}\lambda_2 D(\overline{f}, x_2)   \right)   \mathbbm{1}_{\overline{f} \lra x_2} \right] \mathbb{E}_{\dwt{\nu}_{\overline{f}}}\Big[ \nu_{\mathbf{e}_1}\Big(\Xi_{\mathbf{e}_1}\Big(\wto{\clust}_{B(\underline{f}; 2K)}(0) \cup \dwt{W}_{\overline{f}; B(\overline{f}; 2K) \setminus \wto{\clust}_{B(K^d)}(0) } \Big)\Big) ;\,  \Xi\Big(\wto{\clust}_{B(K^d)}(0) \Big)\Big] . \label{eq:todist}
\end{align}
Here we use the notation
\[W_{v,B(v,R)}:=W_v\cap B(v,R).\]
We will continue by analyzing \eqref{eq:todist}.

\paragraph{Step 6.} 
 By Theorem \ref{thm: avoid-dist}, we have 
 \[\dwt{\mathbb{E}}\left[\exp\left(-an^{-2} \lambda_2 D\left(\overline{f},x_2\right) \right) \mid \overline{f} \lra x_2 \right] =\psi_d\Big(\lambda_2\frac{|\overline{f}-x_2|^2}{n^2}\Big)(1+o(1))\]
 uniformly in $f$, $x_2$ as $n\rightarrow \infty$ if $M$ and $\epsilon$ are fixed.

Given $f\in B(M)$, we have
\[\left|\frac{x_2-\underline{f}}{n}-\frac{x_2}{n}\right|\rightarrow 0\]
uniformly for $x_2$ in the far regime.  Applying this in \eqref{eq:todist}, we see that the expression of interest, \eqref{eq:step3}, has the same limiting behavior as
\begin{equation}
    \label{eq:shortf}
    \beta\psi_d\left(\lambda_2\frac{|x_2|^2}{n^2}\right)  \tau(0, x_2) \sum_{f\in B(M)} \wto{\mathbb{E}}\left[\begin{array}{c}\exp\left(-an^{-2} \lambda_1 D(0, x_1)\right) \mathbf{1}_{0 \sda{B(2K)} \underline{f}, \, 0 \leftrightarrow x_1}\\ \mathbb{E}_{\dwt{\nu}_{\overline{f}}}\Big[ \nu_{\mathbf{e}_1}\Big(\Xi_{\mathbf{e}_1}\Big(\wto{\clust}_{B(\underline{f}; 2K)}\left(0\right) \cup \dwt{W}_{\overline{f}; B(\overline{f}; 2K) \setminus \wto{\clust}_{B(K^d)}(0) } \Big)\Big) ;\,  \Xi\left(\wto{\clust}_{B(K^d)}(0) \right)\Big] \end{array} \right]\ .
\end{equation}

\paragraph{Step 7.}
 An argument like that in Step 6 allows us to approximate the quantity using Corollary \ref{cor:iictwo} 
\[ \wto{\mathbb{E}}\left[\begin{array}{c}\exp\left(-an^{-2} \lambda_1 D(0, x_1)\right) \mathbf{1}_{0 \sda{B(2K)} \underline{f}, \, 0 \leftrightarrow x_1}\\ \mathbb{E}_{\dwt{\nu}_{\overline{f}}}\Big[ \nu_{\mathbf{e}_1}\Big(\Xi_{\mathbf{e}_1}\Big(\wto{\clust}_{B(\underline{f}; 2K)}\left(0\right) \cup \dwt{W}_{\overline{f}; B(\overline{f}; 2K) \setminus \wto{\clust}_{B(K^d)}(0) } \Big)\Big) ;\,  \Xi\Big(\wto{\clust}_{B(K^d)}(0) \Big)\Big] \end{array} \right]\]
by 
\[\tau(0, x_1)\psi_d(\lambda_1|x_1/n|^2) \mathbb{E}_{\wto{\nu}}\Big[ \mathbb{E}_{\dwt{\nu}_{\overline{f}}} \Big[\nu_{\mathbf{e}_1}\Big(\Xi_{\mathbf{e}_1}\Big(\wto{W}_{0; B(\underline{f}; 2K)} \cup \dwt{W}_{\overline{f}; B(\overline{f}; 2K) \setminus \wto{W}_{0; B(\underline{f}; K^d)}} \Big) \Big) \mathbf{1}_{\Xi\left(\wto{W}_{0; B(\underline{f}; 2K)}\right)} \Big] \, \mathbf{1}_{ 0 \sda{B(2K)} \underline{f}} \Big](1+o_n(1))\]
uniformly in $|x_1-x_2|\ge \epsilon$, $|x_1|,|x_2|\le \epsilon^{-1}n$ as $n\rightarrow\infty$.

\paragraph{Step 8.}
Combining the steps above, we find that \eqref{eq:shortf} equals
\begin{equation}
\label{eq:fexpand6}
\beta \psi_d(\lambda_1|x_1/n|^2) \psi_d(\lambda_2|x_2/n|^2) \tau(0,x_1)\tau(0,x_2) \sum_{f\in B(M)} p_f(K)(1+o_n(1)),
\end{equation}
where 
\[p_f(K):=\mathbb{E}_{\wto{\nu}}\Big[ \mathbb{E}_{\dwt{\nu}_{\overline{f}}} \Big[\nu_{\mathbf{e}_1}\Big(\Xi_{\mathbf{e}_1}\Big(\wto{W}_{0; B(\underline{f}; 2K)} \cup \dwt{W}_{\overline{f}; B(\overline{f}; 2K) \setminus \wto{W}_{0; B(\underline{f}; K^d)} } \Big) \Big) \mathbf{1}_{\Xi\left(\wto{W}_{0; B(\underline{f}; 2K)}\right)} \Big] \, \mathbf{1}_{ 0 \sda{B(2K)} \underline{f}} \Big]\]

and the $o_n(1)$ error is uniform for $|x_1-x_2|\ge\epsilon n$, $|x_i|\le \epsilon^{-1}n$.

Taking $K \to \infty$, the expectation $p(K)$ converges by bounded convergence to

\[ \mathbb{E}_{\wto{\nu}}\left[ \mathbb{E}_{\dwt{\nu}_{\overline{f}}} \left[\nu_{\mathbf{e}_1}\big(\Xi_{\mathbf{e}_1}(\wto{W} \cup \dwt{W}_{\overline{f}} \big) \mathbf{1}_{\Xi(\wto{W})} \right] \, \mathbf{1}_{ 0 \sda{} \underline{f}} \right] \]

By monotone convergence, we have 
\begin{align*}
&\lim_{M\rightarrow\infty} \lim_{K\rightarrow \infty} \sum_{f\in B(M)}p_f(K)\\
=&\sum_{f \in \edges}  \mathbb{E}_{\wto{\nu}}\left[ \mathbb{E}_{\dwt{\nu}_{\overline{f}}} \left[\nu_{\mathbf{e}_1}\big(\Xi_{\mathbf{e}_1}(\wto{W} \cup \dwt{W}_{\overline{f}} \big) \mathbf{1}_{\Xi(\wto{W})} \right] \, \mathbf{1}_{ 0 \sda{} \underline{f}} \right] \ , 
\end{align*}

which is $\rho$ by definition \eqref{eq:rhodef}. The proof is complete.
\end{proof}

\section{Joint convergence of pairwise distances} \label{sec:jointpw}
Fix $\epsilon > 0$. We work on the event
\[\Gamma(x_0,\ldots,x_{k-1})\cap \{\mathcal{T}(\omega)=T\},\]
for $(x_0,\ldots,x_{k-1})$ an arbitrary set of vertices in $G(\epsilon, n)$. We will abbreviate the first event of the above display as ``$\Gamma$'' when we need to shorten expressions. 
For an interior vertex $v\in T$, we denote by $g_v$ the pivotal realizing $v$, and let $v_{ij}$ denote the most recent common ancestor in $T$ of $x_i$ and $x_j$. Denote by $\mathrm{path}(i,j)\subset E(T)$ the path joining them.

We now make the observation that other than $g_{v_{ij}}$ itself, every pivotal edge in a subtree below $v_{ij}$ is pivotal for $\{x_i\lra x_j\}$, which leads to
\[\mathrm{dist}(x_i,x_j)\le  \mathrm{dist}(x_i,\bar{g}_{v_{ij}})+\mathrm{dist}(\bar{g}_{v_{ij}},x_j)=\sum_{e\in \mathrm{path}(i,j)}D_e + \#\mathrm{path}(i,j) -1.\]
To reverse the inequality, we will need the following
\begin{lemma}
    For an interior vertex $v$ of $T$, let $f_v$ denote the first pivotal edge past $g_v$ (along a path towards some leaf on the opposite side of $g_v$ from the root). 
    Then 
    \begin{enumerate}
    \item We have the probability bound
    \[\lim_{K\rightarrow\infty}\limsup_{n\rightarrow\infty} \sup_{(x_0,\ldots,x_{k-1}) \in G(\epsilon, n)} \mathbb{P}\left( \left. \cup_{v \in T} \left\{ \bar{g}_v \stackrel{B(g_v,K)}{ \Longleftrightarrow} f_v   \right\}^c \, \right|\, \Gamma\cap \{\mathcal{T}=T\}\right)=0.\]
        \item Suppose $(x_0,\ldots,x_{k-1}) \in G(\epsilon, n)$. On 
        \[\Gamma(x_0,\ldots,x_{k-1})\cap \{\mathcal{T}(\omega)=T\} \cap \cap_{v\in T} \left\{ \bar{g}_v \stackrel{B(g_v,K)}{\Longleftrightarrow} f_v   \right\}, \] for every $x_i$ and $x_j$ in the subtree rooted at $v$
        \[\mathrm{dist}(x_i,x_j)\ge \mathrm{dist}(x_i,\bar{g}_{v_{ij}})+\mathrm{dist}(x_j,\bar{g}_{v_{ij}})-C K^d.\]
    \end{enumerate}
    \begin{proof}
        \begin{enumerate}
            \item This follows by estimates similar to those in Lemmas \ref{lem: bub-trunc}, \ref{lem: double-connection-trunc}, giving
            \[\mathbb{P}\left(\Gamma, \mathcal{T}=T, \bar{g}_v \stackrel{B(g_v,K)}{\Longleftrightarrow} f_v   \right)\le C(K^{-2}+K^{6-d}) n^{(4-d)(k-1)-2}
            .\]
            \item The leg with $f_v$ is part of one of the subtrees below $v$, say the one containing $x_j$. On the event appearing in the lemma statement, $f_v$ is a pivotal edge shared by all leaves of that subtree, so every open path from $x_j$ to $x_i$ must include $w=\underline{f}_v$, in particular, the minimal $x_i\rightarrow x_j$. So we have
            \[\mathrm{dist}(x_i,x_j)=\mathrm{dist}(x_i,w)+\mathrm{dist}(x_j,w)\ge \mathrm{dist}(x_i,\bar{g}_{v_{ij}})+\mathrm{dist}(x_j,\bar{g}_{v_{ij}})-CK^d.\]
        \end{enumerate}

    \end{proof}
        \end{lemma}

        \begin{proof}[Proof of Corollary \ref{cor: distances}]
        The previous lemma immediately implies that, for any $0\le i< j\le k-1$
        \begin{align*}
        n^{-2}\mathrm{dist}(x_i,x_j)
        =~&n^{-2}\mathrm{dist}(x_i,\bar{g}_{v_{ij}})+n^{-2}\mathrm{dist}(x_j,\bar{g}_{v_{ij}})+O(n^{-2})\\
        =~&n^{-2}\sum_{e\in \mathrm{path}(i,j)}D_e+O(n^{-2})
        \end{align*}
        on the event $\cap_{v\in T} \left\{ \bar{g}_v \stackrel{B(g_v,K)}{\Longleftrightarrow} f_v   \right\}$.
        So:
        \begin{align*}
            &\mathbb{E}[e^{-an^{-2}
            \sum_{i<j}\lambda_{ij}D(x^{(n)}_i,x^{(n)}_j) },\Gamma, \mathcal{T}=T]\\
            =~&(1+o_n(1))\cdot\mathbb{E}\left[e^{-an^{-2}
            \sum_{i< j}\lambda_{ij}\sum_{e\in \mathrm{path}(i,j)} D_e }\,,\cap_{v\in T} \left\{ \bar{g}_v \stackrel{B(g_v,K)}{\Longleftrightarrow} f_v   \right\}, \Gamma, \mathcal{T}=T\right]\\
            &+o_K(1)\mathbb{P}(\Gamma, \mathcal{T}=T)\\
            =~& (1+o_n(1))\cdot\mathbb{E}[e^{-an^{-2}
            \sum_{i< j}\lambda_{ij}\sum_{e\in \mathrm{path}(i,j)} D_e }\,, \Gamma, \mathcal{T}=T]\\
            =~&(1+o_n(1))\cdot\mathbb{E}[e^{-an^{-2}
            \sum_{e} \Lambda^T_e(\mathbf{\lambda}) D_e }\,, \Gamma, \mathcal{T}=T].
        \end{align*}
        The result now follows from Theorem \ref{thm: decorated-conv}.
        \end{proof}

\section{Subcritical $k$-point function} \label{sec:subcritk}
In this section, we use the usual coupling between percolation at $p$ and $p_c$. The argument is nearly identical to the one in \cite{CCHS2}: we first sample a percolation configuration at density $p_c$, and then retain each open edge independently with probability
\[q_n=\frac{p_n}{p_c}.\]

Fix $T\in \mathfrak{I}_k$.On $\{\Gamma_n, \mathcal{T}=T\}$, let $P_{n,e}$ be the number of $p_c$-open pivotal edges along the branch corresponding to $e\in E(T)$. Up to the bounded number of pivotals at branch vertices, the total number of pivotals that must survive the thinning is
\[P_n^{\mathrm{tot}}=\sum_{e\in T} P_e+O_k(1)\] where,  for an edge in $T$, $P_e$ denotes the number of pivotal edges along the path realizing this edge. 

The same thinning argument which already appears in \cite[Section 7]{CCHS2} along a single branch can be applied separately to finitely many branches to show
\begin{equation}\label{eqn: Pn-expand}
\begin{split}
\mathbb{P}_{p_n}(\Gamma_n,\mathcal{T}=T)&=\mathbb{E}_{p_c}\big[\mathbf{1}_{\Gamma_n,\mathcal{T}=T} q_n^{P_n^T}\big]+o(\mathbb{P}_{p_c}(\Gamma_n))\\
&=
\mathbb{E}_{p_c}\big[\mathbf{1}_{\Gamma_n,\mathcal{T}=T}\left(1-\frac{m}{p_c n^2}\right)^{P_n^{\mathrm{tot}}}\big]+o(n^{(4-d)(k-1)-2}).
\end{split}
\end{equation}
The discrepancies between the error terms and the approximation come from requiring survival of critical pivotals. Two types of error events occur: single deletions in ``large" pivotal bubbles, or at least two deletions in bounded bubbles. These are controlled as in \cite[Section 4]{CCHS2}. A union bound over the $2k-3$ branches gives the required estimate.
Writing
\[\left(\frac{p_n}{p_c}\right)^{P_n^\mathrm{tot}}=\exp\big(P_n^{\mathrm{tot}}\log\big(1-\frac{m}{p_c n^2}\big)\big)\]
By a Taylor expansion, we have
\[\log\big(1-\frac{m}{p_c n^2}\big)=-\frac{m}{n^2p_c}+O(n^{-4}).\]
So the main term in \eqref{eqn: Pn-expand} is approximated by a decorated $k$-point function, where the decorations along each branch are exponentials of pivotal counts:
\begin{equation}
\mathbb{P}_{p_n}(\Gamma_n, \mathcal{T}=T)=\widetilde{\tau}^{(\frac{m}{p_c},\ldots,\frac{m}{p_c})}_{k,p_c}(x^{(n)}_1,\ldots,x^{(n)}_k)+o(n^{(4-d)(k-1)-2}), 
\end{equation}
where
$\tilde{\tau}_{k, p_c}^{(\mu)}$ is defined analogously to $\tau_{k,p}^{(\mu)}$, with $a D_e$ replaced by $a_{\mathrm{piv}} P_e$, with $P_e$ the pivotal count along edge $e\in T$. Here $a_{\mathrm{piv}}$ is the constant in \eqref{eqn: aP-def}. The analog of Theorem \ref{thm: distances} with distances replaced by pivotal counts, with the same proof, now gives the desired result.

\vskip 0.2in

\noindent {\bf Acknowledgements.} We are grateful to Wendelin Werner and Perla Sousi for comments and encouragement. The research of S.~C.~was supported by NSF grant DMS-2154564. The research of J.~H.~was supported by NSF grant DMS-1954257. The research of P.S. was supported by NSF grants DMS-2154090 and DMS-2238423, and a Simons Fellowship. Part of this work was completed while P.~S. was at SLMath. In preparing the current revision, AI tools were used to verify several computations, especially the correctness of the limiting formulas. In particular, \eqref{eqn:  T-sum} was missing a factor of $\beta$ in the first version of this paper. The proof ideas, and any remaining mistakes, are due to the authors.

\end{document}